\documentclass[11pt,letterpaper,twoside,reqno]{amsart}
\title[Gravitating vortices and symplectic reduction by stages]{Gravitating vortices\\ and symplectic reduction by stages}

\author[\'Alvarez-C\'onsul]{Luis \'Alvarez-C\'onsul}
\address{Instituto de Ciencias Matem\'aticas (CSIC-UAM-UC3M-UCM)\\ Nicol\'as Cabrera 13--15, Cantoblanco\\ 28049 Madrid, Spain}
\email{l.alvarez-consul@icmat.es}

\author[Garcia-Fernandez]{Mario Garcia-Fernandez}
\address{Instituto de Ciencias Matem\'aticas (CSIC-UAM-UC3M-UCM)\\ Nicol\'as Cabrera 13--15, Cantoblanco\\ 28049 Madrid, Spain}
\email{mario.garcia@icmat.es}

\author[Garc\'{\i}a-Prada]{Oscar Garc\'{\i}a-Prada}
\address{Instituto de Ciencias Matem\'aticas (CSIC-UAM-UC3M-UCM)\\ Nicol\'as Cabrera 13--15, Cantoblanco\\ 28049 Madrid, Spain}
\email{oscar.garcia-prada@icmat.es}

\author[Pingali]{Vamsi Pritham Pingali}
\address{Department of Mathematics, Indian Institute of Science, Bangalore, India - 560012}
\email{vamsipingali@iisc.ac.in}

\author[Yao]{Chengjian Yao}
\address{ Institute of Mathematical Sciences, ShanghaiTech University, 393 Middle Huaxia Road, Pudong,
	Shanghai, 201210 China.}
\email{yaochj@shanghaitech.edu.cn}

\thanks{
The first three authors are partially supported by the Spanish Ministry of Science and Innovation, through the `Severo Ochoa Programme for Centres of Excellence in R\&D' (CEX2019-000904-S).
The first and the third authors are partially supported by MICINN under grants PID2019-109339GB-C31 and PID2022-141387NB-C21.
The second author is partially supported by MICINN under grants PID2019-109339GA-C32, PID2022-141387NB-C22 and CNS2022-135784. The fourth author is partially supported by grant F.510/25/CAS-II/2018(SAP-I) from UGC
(Govt. of India). The last named author is supported by the NSFC grant No. 12371207.}

\usepackage{hyperref,url,bm,xcolor}
\usepackage{esint}
\usepackage{amssymb,latexsym}
\usepackage{amsmath,amsthm}
\usepackage{mathrsfs}
\usepackage[latin1]{inputenc}
\usepackage{enumerate, comment}
\usepackage[all]{xy} \CompileMatrices
\SelectTips{cm}{12} 

\theoremstyle{plain}
\newtheorem{theorem}{Theorem}[section]
\newtheorem{lemma}[theorem]{Lemma}

\newtheorem{proposition}[theorem]{Proposition}

\theoremstyle{definition}
\newtheorem{definition}[theorem]{Definition}
\newtheorem{definition-theorem}[theorem]{Definition-Theorem}
\newtheorem{example}[theorem]{Example}
\theoremstyle{remark}
\newtheorem{remark}[theorem]{Remark}

\newtheorem*{acknowledgements}{Acknowledgements}

\numberwithin{equation}{section} 

\newcommand{\tr}{\operatorname{tr}}
\newcommand{\pr}{p}

\newcommand{\Ad}{\operatorname{Ad}}
\newcommand{\Aut}{\operatorname{Aut}}
\newcommand{\dbar}{\bar{\partial}}

\newcommand{\CC}{{\mathbb C}}
\newcommand{\RR}{{\mathbb R}}

\renewcommand{\(}{\left(}
\renewcommand{\)}{\right)}

\newcommand{\defeq}{\mathrel{\mathop:}=} 
\newcommand{\iprod}{\mathbin{\lrcorner}} 

\newcommand{\surj}{\to\kern-1.8ex\to}

\newcommand{\lto}{\longrightarrow}
\newcommand{\lra}[1]{\stackrel{#1}{\longrightarrow}}

\DeclareMathOperator{\Ric}{Ric}

\newcommand{\cA}{\mathcal{A}}

\newcommand{\cJ}{\mathcal{J}}

\newcommand{\cT}{\mathcal{T}}
\newcommand{\cF}{\mathcal{F}}

\newcommand{\cG}{\mathcal{G}}

\newcommand{\cS}{\mathcal{S}}
\newfont{\gl}{eufm10 scaled \magstep1}

\newcommand{\Lie}{\operatorname{Lie}}
\newcommand{\LieG}{\operatorname{Lie} \cG}
\newcommand{\cX}{\widetilde{\mathcal{G}}}
\newcommand{\LieX}{\operatorname{Lie} \cX}
\newcommand{\cH}{\mathcal{H}} 
\newcommand{\LieH}{\Lie\cH}

\newcommand{\SU}{\operatorname{SU}}
\newcommand{\PP}{{\mathbb P}}

\newcommand{\AMY}{I} 

\allowdisplaybreaks

\begin{document}
\begin{abstract}
We undertake a novel approach to the existence problem for gravitating vortices on a Riemann surface based on symplectic reduction by stages, which seems to be new in the PDE as well as the gauge theory literature. The main technical tool for our study is the reduced $\alpha$-K-energy, for which we establish convexity properties by means of finite-energy pluripotential theory, as recently applied to the study of constant scalar curvature K\"ahler metrics. Using these methods, we prove that the existence of solutions to the gravitating vortex equations on the sphere implies the polystability of the effective divisor defined by the zeroes of the Higgs field. This approach also enables us to establish the uniqueness of gravitating vortices in any admissible K\"ahler class, in the absence of automorphisms. Lastly, we also prove the existence of solutions for the gravitating vortex equations for genus $g\geq 1$ for certain ranges of the coupling constant $\alpha$ and the volume.
\end{abstract}

\maketitle

\section{Introduction}\label{sec:intro}

This paper is devoted to the study of the existence problem for the gravitating vortex equations.
Let $\Sigma$ be a compact Riemann surface of  genus $g$. Let $L$ be a holomorphic line bundle over $\Sigma$ and $\bm\phi \in H^0(\Sigma,L)$ a holomorphic section of $L$. We fix a parameter $0 < \tau \in \RR$ and a
coupling constant $\alpha \in \RR$.
The gravitating vortex equations  for a K\"ahler metric on $\Sigma$ with K\"ahler form $\omega$ and a Hermitian metric $h$ on $L$ are
\begin{equation}
\left\{
\label{eqn:GV0}
\begin{array}{l}
iF_{h} + \frac{1}{2} (\lvert\bm\phi\rvert_{h}^2 - \tau)\omega
= 0,\\
S_{\omega} + \alpha ( \Delta_{\omega} + \tau )(\lvert\bm\phi\rvert_{h}^2 - \tau )  = c.
\end{array}
\right.
\end{equation}
Here $F_h$ is  the curvature of the Chern connection of $h$, $\lvert\bm\phi\rvert_h$ is the pointwise norm of $\bm\phi$ with respect to $h$, $S_\omega$ is the scalar curvature of $\omega$ 
and $\Delta_\omega$ is the Laplace operator for the metric $\omega$, whose  action  on a smooth function $f$ is defined by $\Delta_\omega f \omega = 2i \dbar\partial f$. 
The constant $c\in\RR$ is topological and is obtained by integrating~\eqref{eqn:GV0} over $\Sigma$ to be the following.
\begin{equation}\label{eq:constantcintro}
c=\frac{2\pi(\chi(\Sigma)-2\alpha\tau N)}{V},
\end{equation}
where $V = \int_{\Sigma} \omega$ and
$$
N=\int_{\Sigma}c_1(L).
$$
The first equation in \eqref{eqn:GV0} is known as the abelian vortex equation, and for fixed $\omega$ it is well known (see~\cite{Brad,Prada2,Prada,Noguchi} and Appendix~\ref{sub:variational-approach-vortex-eq}) that the existence of solutions with $\bm\phi\neq 0$  (called vortices) is equivalent to
\begin{equation}\label{tau-bound}
V > \frac{4\pi N}{\tau}.
\end{equation}
The second equation in \eqref{eqn:GV0} reflects a back-reaction of the vortex on the K\"ahler metric $\omega$, whose intensity is measured by the coupling constant $\alpha$, 
similarly to matter couples to the gravitational field in Einstein's field equations in general relativity, with a push proportional to the gravitational constant. This analogy motivates referring to the solutions of \eqref{eqn:GV0} as \emph{gravitating vortices}, as first introduced in \cite{Al-Ga-Ga1}.

Beyond this appealing analogy, a striking observation made in \cite{Al-Ga-Ga1}, based on earlier work by Linet \cite{Linet} and Comtet-Gibbons \cite{ComtetGibbons} (see also \cite{Yang1992}), is that, in the special case $c=0$ and genus $g = 0$, the gravitating vortices can be regarded as special solutions of the Einstein-Maxwell-Higgs equations in general relativity \cite{Chae,CBruhat,Friedrich,Tadmon}. The solutions obtained from \eqref{eqn:GV0} impose a natural product ansatz $\mathbb{R}^{1,1} \times \Sigma$ for the space-time metric and minimize a suitable energy functional. Having a special physical significance, they are referred to as \emph{cosmic strings}, while the corresponding equations are known as the Einstein--Bogomol'nyi equations or the self-dual Einstein--Maxwell--Higgs equations. Partly motivated by this interesting relation to physics, there has been some recent activity in the study of the gravitating vortex equations \eqref{eqn:GV0}, both in the complex geometry \cite{Al-Ga-Ga2,Al-Ga-Ga-P,Al-Ga-Ga-P-Y,FPY,Pingali,Yao, Yao2} and PDE communities \cite{HLS,HLSErratum}.

The basic formal structure underlying the equations \eqref{eqn:GV0} is provided by the theory of Hamiltonian reduction in symplectic geometry. The gravitating vortex equations describe the zeros of a moment map for the Hamiltonian
action of an infinite-dimensional gauge group $\cX$ \cite{Al-Ga-Ga,Al-Ga-Ga-P}, on a phase space given by unitary connections, Higgs fields, and complex structures on the underlying smooth surface $S$. Compared to the moment map intepretation of other equations in complex geometry, as for the renown constant scalar curvature (cscK) problem \cite{Don,Fujiki}, the main novelty here is the structure of the group of symmetries, given by a non-trivial extension of the form
\begin{equation}\label{eq:coupling-term-moment-map-}
1\to\cG\lto\cX \to \cH\to 1,
\end{equation}
where $\cH$ denotes the group of Hamiltonian symplectomorphisms of a fixed symplectic structure on the surface $S$ and $\cG$ is the gauge group of a unitary line bundle with degree $N$.

As discovered in \cite{Al-Ga-Ga}, the special structure of the \emph{extended gauge group} $\cX$ reflects on a higher level of difficulty on the existence problem for the corresponding moment map equations. Not only does the group of symmetries lack a formal complexification, as is the case for the cscK problem in K\"ahler geometry \cite{Don}, but also the corresponding infinite-dimensional symmetric space (whose points are given by pairs $(\omega,h)$ as in \eqref{eqn:GV0}) does not admit a natural Riemannian metric compatible with the characteristic affine connection (see \cite[Remark 4.4]{Al-Ga-Ga}). This symmetric space, denoted by abuse of notation $\cX^c/\cX$, has furthermore an intricate submersion structure  with totally geodesic fibres \cite[Proposition 3.17]{Al-Ga-Ga} over the space of K\"ahler potentials \cite{Mabuchi}. This weak structure of $\cX^c/\cX$ prevents us from directly applying to \eqref{eqn:GV0} the strong methods from pluripotential theory recently developed for the study of the space of K\"ahler potentials \cite{Chen}, and so succesfully applied to the cscK problem \cite{Darvas}.

Motivated by this problem, and building upon the symplectic interpretation of the gravitating vortex equations \cite{Al-Ga-Ga-P}, in the present work we propose and carry out a novel approach to the existence problem for gravitating vortices on a Riemann surface based on symplectic reduction by stages, which seems to be new both in the PDE and in the field theory literature (cf. Example No. 3 in~\cite[pp. 108--109]{Ma-Mi-Or-Pe-Ra}). This approach is a two-step process in which  we first consider the symplectic reduction by the normal subgroup $\cG \subset \cX$, and then consider the residual Hamiltonian $\cH$-action. In the K\"ahler context, this procedure is known as K\"ahler reduction by stages, and we shall explain the main formal ingredients in Section \ref{section:setupabstract}, in a situation where one has a finite dimensional K\"ahler manifold. Effectively, in our setup, the reduction by stages approach to gravitating vortices stands for considering the following PDE
\begin{equation}\label{eqn:GV0reducedintro}
S_{\omega} + \alpha ( \Delta_{\omega} + \tau )(\lvert\bm\phi\rvert_{h_\omega}^2 - \tau )  = c
\end{equation}
for a K\"ahler metric $\omega$ on a fixed Riemann surface $\Sigma$. Here, $h_\omega$ is the unique solution to the abelian vortex equation with respect to $\omega$ established in \cite{Brad,Prada2,Prada,Noguchi}.
From the analytical viewpoint, a remarkable fact about the previous equation is that it is a non-local PDE, depending on global information on $\Sigma$ via a (simple) stability condition.

Using our new approach, in the present work we prove three main results which constitute an important step forward in the theory for the gravitating vortex equations \eqref{eqn:GV0}. In our first theorem we establish the first uniqueness result for the equations (\ref{eqn:GV0}) for arbitrary genus $g$, giving a complete solution to \cite[Conjecture 5.5]{Al-Ga-Ga-P} in the stable case. We should stress the fact that our next theorem provides, in particular, the first known uniqueness result for the self-dual Einstein--Maxwell--Higgs equations, corresponding to the case $c = 0$ (cf. \cite{Chae}). Namely, we have the following result (Theorem \ref{thm:uniqueness}).

\begin{theorem} \label{uniqueness-intro}
Suppose either \textup{(i)} $g(\Sigma)\geq 1$, or \textup{(ii)} $g(\Sigma)=0$ and $\bm\phi$ vanishes at least at $3$ points. Then smooth solutions to the gravitating vortex equations with fixed volume $V>\frac{4\pi N}{\tau}$ are unique.
\end{theorem}

Before going into the specifics of our other two main theorems, let us briefly summarise the state of the art of the existence problem for the gravitating vortex equations~\eqref{eqn:GV0}. The study of the existence problem for ~\eqref{eqn:GV0} in the special case $c=0$ and genus $g=0$ goes back to Spruck and Yang~\cite{YangSpruck,YangSpruck2}, where the equations
reduce to a single PDE for a function on the sphere and have the aforementioned interesting relation to physics.
An existence theorem \cite[Theorem 6.1]{Al-Ga-Ga-P}  was proven for surfaces of genus $g\geq 2$.
The case of  genus $g=1$ was studied in \cite{Al-Ga-Ga1}, wherein it was shown that the gravitating vortex
equations (with $\bm\phi\neq 0$) always have a solution in the weak coupling limit
$0< \alpha \ll 1$ (see~\cite[Theorem~4.1]{Al-Ga-Ga1} for a
precise formulation).

Unlike the cases of genus $g\geq 1$, it was shown in \cite{Al-Ga-Ga-P} that in
genus $g=0$ there are obstructions to the existence of solutions of \eqref{eqn:GV0}, making the
problem  comparatively closer to the  Calabi problem on the existence of K\"ahler--Einstein
metrics, where algebro-geometric stability obstructions appear on
compact K\"ahler manifolds with $c_1>0$. Based on the moment map interpretation of the gravitating vortex
equations, it was suspected that in genus $g=0$ the existence of solutions had to be equivalent to some stability condition. 

The first evidence pointing out to such a correspondence lay in Yang's existence result~\cite{Yang} for \eqref{eqn:GV0} when $c=0$ and $V \gg \frac{4\pi N}{\tau}$ (see \cite{HLS,HLSErratum} for an alternative proof). In this case, $\Sigma$ is forced topologically to be $\PP^1$ and \eqref{eqn:GV0} are known as the  Einstein--Bogomol'nyi equations or the self-dual Einstein--Maxwell--Higgs equations. Yang's result was extended in \cite{FPY} to the case $c \geq 0$ and any admissible volume $V > \frac{4\pi N}{\tau}$. These results can be reformulated  in the language of Mumford's Geometric Invariant Theory (GIT) as follows.

\begin{theorem}[\cite{FPY,Yang,Yang2}]\label{th:Yangintro}
	Suppose that~\eqref{tau-bound} is satisfied and that $\alpha \in (0,1/\tau N]$. Let $D=\sum_j n_jp_j$
	be an effective divisor on $\PP^1$ corresponding to a pair $(L,\bm\phi)$.
	Then, the gravitating vortex equations \eqref{eqn:GV0} on $(\PP^1,L,\bm\phi)$ have
	solutions with volume $V > \frac{4\pi N}{\tau}$, provided that the divisor $D$ is GIT polystable for the
	canonical linearized $\textup{SL}(2,\CC)$-action on the space of effective
	divisors.
\end{theorem}

In \cite[Theorem 1.3]{Al-Ga-Ga-P} the authors stated the following  converse to Theorem~\ref{th:Yangintro}.

\begin{theorem}\label{th:Yangconjectureintro}
	If $(\PP^1,L, \bm\phi)$ admits a solution of the gravitating vortex
	equations with $\alpha > 0$, then \eqref{tau-bound} holds and the
	divisor $D$ is polystable for the $\textup{SL}(2,\CC)$-action.
\end{theorem}

While the result in Theorem \ref{th:Yangconjectureintro} is right, it turns out that the proof  given in  \cite{Al-Ga-Ga-P} is not. More specifically, the proof of \cite[Theorem 1.3]{Al-Ga-Ga-P} contains an essential flaw (which nevertheless does not affect other parts of the paper). The initial goal of the present work was to give a correct proof, which we achieve in Theorem~\ref{thm:existence-implies-stability}.

In this paper we also improve the existence results given in \cite{Al-Ga-Ga-P} for the genus $g\geq 1$. Namely, we prove the following result (see Theorem \ref{thm:existenceforggeq1}).

\begin{theorem}\label{thm:existenceg>1intro}
	Let $g\geq 1$. Assume  the following three numerical conditions
	\begin{enumerate}
		\item[\textup{(1)}] $V- \frac{4\pi N}{\tau}>0$,
		\item[\textup{(2)}] $\alpha\tau \left( \frac{8\pi N}{\tau} - V\right) < \frac{2g-2}{N}\left( V- \frac{4\pi N}{\tau}\right)$,
		\item[\textup{(3)}] $2\alpha\tau m<1$,
	\end{enumerate}
	where $m$ is the maximum multiplicity of the zeroes of $\bm\phi$. Then there exists a unique smooth solution to the gravitating vortex equations.
	\label{thm:existenceforggeq1-intro}
\end{theorem}

The significance of our main results is better understood from the point of view of moduli theory and the classical Teichm\"uller spaces. Ever since the pioneering work of Fricke and Teichm\"uller, it has been convenient to enhance the geometric data of the Riemann surface $\Sigma$ with marked points or other types of decoration. The simplest decoration for which this moduli space has not been so well studied yet is the case of an effective divisor $D$ on the surface, and the results of this paper pave the way to a complex analytic approach (cf., e.g., \cite[Section 2.1.3]{Hassett}, for the corresponding algebro-geometric moduli problem). As a matter of fact, Theorem \ref{uniqueness-intro} and Theorem \ref{thm:existenceg>1intro} directly imply that in genus $g \geq 1$, and for an effective range of the parameter $\alpha$, a suitable moduli space of solutions to \eqref{eqn:GV0} can be interpreted as a Teichm\"uller space for pairs $(\Sigma,D)$ (cf. \cite[Section 3.3]{Al-Ga-Ga-P}). 

In the special case of genus $g = 0$, Theorem \ref{uniqueness-intro} and Theorem \ref{th:Yangconjectureintro}, combined with the main results in \cite{FPY,Yang}, establish a bijection between an open set in the moduli space of gravitating vortices with fixed volume $V>\frac{4\pi N}{\tau}$ and coupling constant $\alpha \in (0,1/\tau N]$ and the stable locus of the moduli space of binary quantics, given by the GIT quotient
$$
S^N \PP^1 /\!\!/ \operatorname{SL}(2,\CC),
$$
bringing us closer to the solution of \cite[Conjecture 5.6]{Al-Ga-Ga-P}. Important steps towards the compactification of  the moduli spaces of gravitating vortices with $c \geq 0$ and genus $g = 0$ have been taken in \cite{FPY,Yao2}. In the special case $c = 0$, these spaces can be interpreted as moduli spaces of cosmic string solutions to the Einstein-Maxwell-Higgs equations in general relativity, and we speculate that our current developments may have an impact in future advances in theoretical physics, such as the study of cosmic string dynamics \cite{ViShe} via the method of \emph{moduli space approximation} \cite{LS,Man,MR}.

An important ingredient emerging from  the reduction by stages approach to the study of  the gravitating vortex equations is the reduced $\alpha$-K-energy. This is a functional
$$
\mathcal{K}_\alpha = \mathcal{K} + \mathcal{M}_\alpha \colon \mathcal{P}_{\omega_0} \to \mathbb{R}
$$
defined on the infinite-dimensional space $\mathcal{P}_{\omega_0}$ of K\"ahler potentials determined after fixing a reference background K\"ahler metric $\omega_0$ with total volume $V > 0$. Here, $\mathcal{K}$ is the usual K-energy in the space of K\"ahler potentials \cite{Mabuchi} while $\mathcal{M}_\alpha$ is a new functional which arises via the reduction by stages mechanism. Key to the proof of our main theorems is the fact that this last functional is continuous on the completion of the space of smooth K\"ahler potentials introduced by Darvas \cite{Darvas0} (see Lemma \ref{lem:continuityofMalphaoncompletion}), whereby the reduced $\alpha$-K-energy inherits good semicontinuity properties (see Theorem \ref{thm:extended-vortex-K-energy}).

Equations \eqref{eqn:GV0}  were introduced in \cite{Al-Ga-Ga1} in an attempt to find symmetric solutions to the K\"ahler--Yang--Mills
equations \cite{Al-Ga-Ga}. These are equations for a K\"ahler metric on a compact complex manifold and a Hermitian metric on a holomorphic vector bundles over it. Indeed, solutions to
\eqref{eqn:GV0} are equivalent to $\SU(2)$-invariant solutions the K\"ahler--Yang--Mills equations on the  rank 2 holomorphic vector bundle $E$ over $\Sigma\times \PP^1$, where $\PP^1$ is the Riemann sphere,
similarly to the approach to the abelian vortex equation in \cite{Prada2}.
The results in this paper can therefore be naturally interpreted whithin the theory for the K\"ahler--Yang--Mills equations, leading to new existence and uniqueness results.

To our knowledge, this is the first paper where symplectic reduction by stages is used in the context of the study of canonical metrics in K\"ahler geometry and gauge theory. As illustrated in this work, this is a very powerful tool that may very likely be applied to other situations, including  the K\"ahler--Yang--Mills equations from which  gravitating vortices initially emerged.

The present paper is organized as follows: the main formal ingredients of the K\"ahler reduction by stages are explained in Section \ref{section:setupabstract}. In Section \ref{sec:reducedK} we give an explicit formula for the reduced $\alpha$-K-energy. Section \ref{sec:extensionKen} is devoted to the extension of the reduced $\alpha$-K-energy to the completion of the space of smooth K\"ahler potentials introduced by Darvas. In Section \ref{section:Laplacianestimate}  we study the convexity properties of the (extended) reduced $\alpha$-K-energy. In Section \ref{sec:geometric-applications} we apply the results obtained in the previous sections on the reduced $\alpha$-K-energy and convexity to prove the main results of our paper, including Theorem
\ref{th:Yangconjectureintro} (Theorem \ref{thm:existence-implies-stability}), Theorem \ref{uniqueness-intro} (Theorem \ref{thm:uniqueness}),  and Theorem \ref{thm:existenceforggeq1-intro} (Theorem \ref{thm:existenceforggeq1}). An important ingredient for Theorem \ref{th:Yangconjectureintro} is the proof of properness of the reduced $\alpha$-K-energy given in Theorem~\ref{thm:K-alpha-proper}.

\begin{acknowledgements}
The authors wish to thank Marco Castrill\'on L\'opez and Tam\'as Darvas for useful discussions. This research was supported in part by the International Centre for Theoretical Sciences (ICTS) for participating in
the programme - Vortex Moduli (code:ICTS/Vort2023/02).
\end{acknowledgements}

\section{Reduction by stages and gravitating vortices}\label{section:setup}

\subsection{K\"ahler reduction by stages}\label{section:setupabstract}

In this section we describe some basic aspects of the theory of K\"ahler reduction by stages, which we will formally apply as a guideline for our main results on the gravitating vortex equations. For a more complete
account of this theory, we refer to \cite{Greb}.

Let $(\cT,\omega_\cT)$ be a symplectic manifold endowed with a symplectic left action of a Lie group $\cX$. We assume that the group is given by a Lie group extension of the form
\begin{equation}
\label{eq:Ext-Lie-groups}
1 \to \cG \lra{\iota} \cX \lra{\pr} \cH \to 1.
\end{equation}
The extension~\eqref{eq:Ext-Lie-groups} determines an extension of Lie
algebras
\begin{equation}
\label{eq:Ext-Lie-alg}
0 \to \LieG \lra{\iota} \LieX \lra{\pr} \LieH \to 0,
\end{equation}
where the use of the same symbols $\iota$ and $\pr$ should lead to no confusion. 
as a short exact sequence of vector spaces.
We are interested in the case where~\eqref{eq:Ext-Lie-alg} does not split as a sequence of Lie algebras. Assume that the $\cX$-action is Hamiltonian and there exists a $\cX$-equivariant moment map
$$
\mu_{\cX}\colon \cT\lto (\LieX)^*.
$$
Provided that $0 \in (\LieX)^*$ is a regular value of $\mu_{\cX}$ and that $\cX$ acts freely and properly on $\mu_{\cX}^{-1}(0)$, the symplectic reduction of $\cT$ by $\cX$, defined by
$$
\mathcal{M} =  \mu_{\cX}^{-1}(0)/\cX,
$$
inherits a natural structure of symplectic manifold. Given the extension structure of the group $\cX$, the symplectic quotient can be obtained in a two-step process instead: with the same hypothesis, we first take
the symplectic reduction by the normal subgroup $\cG \subset \cX$,
$$
\mathcal{N} = \mu_{\cG}^{-1}(0)/\cG,
$$
where $\mu_{\cG}\colon \cT \to (\LieG)^*$ is a moment map for the $\cG$-action obtained by composition of $\mu_{\cX}$ with the natural map $\iota^*\colon (\LieX)^* \to (\LieG)^*$. The symplectic manifold
$\mathcal{N}$ inherits a residual Hamiltonian $\cH$-action, with moment map
$$
\mu_{\cH}\colon \mathcal{N} \lto (\LieH)^*.
$$
More explicitly, $\cH$ is identified with the right quotient $\cX/\cG$ with induced left action
$$
(\tilde g \cdot \cG) \cdot [x] = [\tilde g \cdot x] \in \mathcal{N}
$$
and the moment map $\mu_{\cH}$ is given by
\begin{equation}\label{eq:muHexp}
\langle \mu_{\cH}([x]),y \rangle = \langle \mu_{\cX}(\tilde x), \tilde y \rangle,
\end{equation}
for any $\tilde x \in \cT$ such that $[\tilde x] = [x] \in \mathcal{N}$ and any choice of $\tilde y \in \LieX$ such that $p(\tilde y) = y$. The fact that $\mu_\cH$ is well defined follows easily from the
$\cG$-equivariance of $\mu_{\cX}$, using that $x \in \mu_{\cG}^{-1}(0)$ and that the adjoint $\cG$-action on $\LieX$ preserves the fibres of $p \colon \LieX \to \LieH$. Secondly, we do Hamiltonian reduction of
$\mathcal{N}$ by $\cH$, obtaining a natural symplectomorphism
$$
\mu_{\cH}^{-1}(0)/\cH \cong \mu_{\cX}^{-1}(0)/\cX,
$$
induced by the inclusion $\mu_{\cX}^{-1}(0) \subset \mu_{\cG}^{-1}(0)$.

Assume now that $(\cT,\mathbf{J},\omega_\cT)$ is a K\"ahler manifold and that $\cX$ acts further by holomorphic isometries. Assume that $\cX$ is compact and that its complexification $\cX^c$ acts by biholomorphisms
on $\cT$. Then we can define the set of \emph{stable points} for the $\cX^c$-action by (see \cite{Mundet})
$$
\cT^s_{\cX^c} = \{ x \in \cT \mid (\cX^c \cdot x) \cap \mu_{\cX}^{-1}(0) \neq \emptyset \}.
$$
Since we are assuming that the $\cX$-action on $\mu_{\cX}^{-1}(0)$ is free, one can prove that $\cT^s_{\cX}$ is an open and $\cX^c$-invariant subset of $\cT$. By the general theory of K\"ahler reduction, there is a
natural diffeomorphism
\begin{equation}\label{eq:KempfisoX}
\mathcal{M} \cong \cT^s_{\cX^c} / \cX^c
\end{equation}
which endowes $\mathcal{M}$ with a structure of K\"ahler manifold. More generally, without the assumption on the isotropy group, $\mathcal{M}$ inherits a structure of stratified K\"ahler space (see \cite[Theorem
1.8]{Greb}). In the special case that $\cT$ is a projective manifold, the stable locus can be defined independently of the moment map by means of Geometric Invariant Theory.

As above, the extension structure of $\cX$ enables us to construct this K\"ahler structure in a two-step process. For this, we first consider the set of stable points for the $\cG^c$-action by
$$
\cT^s_{\cG^c} = \{ x \in \cT \mid (\cG^c \cdot x) \cap \mu_{\cG}^{-1}(0) \neq \emptyset \},
$$
where $\cG^c \subset \cX^c$ is the complexification of $\cG$. Then, there is a natural diffeomorphism
\begin{equation}\label{eq:Kempfiso}
\mathcal{N} \cong \cT^s_{\cG^c} / \cG^c
\end{equation}
which endowes $\mathcal{N}$ with a natural structure of K\"ahler manifold such that $\cH$ acts by holomorphic isometries. We then consider the complexification $\cH^c$ of $\cH$ and the corresponding set of stable
points for the induced $\cH^c$-action, provided that it exists,
$$
\mathcal{N}^s_{\cH^c} = \{ x \in \mathcal{N} \mid (\cH^c \cdot x) \cap \mu_{\cH}^{-1}(0) \neq \emptyset \}.
$$
Similarly as before, we obtain now a K\"ahler isometry (see \cite[Theorem 2.1]{Greb})
$$
\mathcal{M} \cong \mu_{\cH}^{-1}(0)/\cH \cong \mathcal{N}^s_{\cH^c}/ \cH^c.
$$
A key point in the previous construction is the existence of the holomorphic $\cH^c$-action on $\mathcal{N}$. This is ensured, for instance, if $\mathcal{N}$ is compact or if $\cT^s_{\cG^c}$ is a complex analytic
subset of $\cT$.

The upshot of the previous construction is a criterion to decide whether a point $x \in \cT$ is represented in the complex quotient $\cT^s_{\cX^c} / \cX^c$, that is, when
\begin{equation}\label{eq:existenceabs}
(\cX^c \cdot x) \cap \mu_{\cX}^{-1}(0) \neq \emptyset.
\end{equation}
For simplicity, and since this is the case in the situation of our interest, we shall assume
$$
\cT^s_{\cG^c} = \cT,
$$
that is, for any given $x \in \cT$ there exists $g_x \in \cG^c$ (unique up to the $\cG$-action) such that
$$
\mu_{\cG}(g_x \cdot x) = 0.
$$
Then, in order to check whether \eqref{eq:existenceabs} holds, as in the previous construction one considers
$$
[g_x \cdot x] \in \mathcal{N}
$$
and then $x \in \cT^s_{\cX^c}$ is equivalent to
\begin{equation}\label{eq:existencereduced}
(\cH^c \cdot [g_x \cdot x]) \cap \mu_{\cH}^{-1}(0) \neq \emptyset.
\end{equation}
Condition \eqref{eq:existencereduced} can be now addressed using standard machinery in K\"ahler geometry. For instance, consider the negatively curved symmetric space
$$
\mathcal{R} = \cH^c/\cH
$$
(of all the cosets $g\mathcal{H}$ where $g\in \mathcal{H}^c$) with metric given by a choice of bi-invariant metric on $\LieH$. Define a 1-form $\sigma_x \in \Omega^1(\mathcal{R})$ by
$$
\sigma_x(v) = \langle \mu_{\mathcal{H}}(h^{-1} \cdot [g_x \cdot x]), \Ad(h^{-1})(iv)  \rangle
$$
for $[h] \in \mathcal{R}$ and
$$
v \in T_{[h]}\mathcal{R} \cong \Ad(h) i\LieH.
$$
By the defining properties of the moment map $\mu_{\mathcal{H}}$, $\sigma_x$ is well-defined and closed and, since $\mathcal{R}$ is contractible, it integrates to a functional (see e.g. \cite{Mundet})
$$
\mathcal{K} \colon \mathcal{R} \lto \mathbb{R},
$$
normalized so that $\mathcal{K}([1]) = 0$. The functional $\mathcal{K}$ is geodesically convex, where geodesics on
$\mathcal{R}$ are given by 1-parameter subgroups
$$
t \longmapsto [h e^{ity}], \qquad y \in \LieH.
$$
Hence, by geodesic completeness of $\mathcal{R}$,  condition \eqref{eq:existencereduced} is equivalent to the existence of a critical point of $\mathcal{K}$. The latter reduces to the \emph{linear properness} of the
functional (see e.g. \cite{Mundet}), namely, the existence of a pair of positive constants $C_1,C_2 > 0$ such that
$$
d([1],[h]) \leq C_1 \mathcal{K}([h]) + C_2
$$
for all $[h] \in \mathcal{R}$, where $d(,)$ stands for the natural distance on the symmetric space. Note here that, by the special form of the geodesics and the polar decomposition $\cH^c = \exp(i \LieH) \cdot \cH$,
the previous condition is equivalent to
$$
\lvert y\rvert \leq C_1 \mathcal{K}([e^{i y}]) + C_2
$$
for all $y \in \LieH$ (and possibly different constants). Finally, we note that the linear properness of $\mathcal{K}$ can be recasted, similarly as in the Hilbert--Mumford criterion in algebraic geometry, in terms
of a \emph{stability condition} for $[g_x \cdot x] \in \mathcal{N}$, given by numerical criterion
$$
\lim_{t \to +\infty} \sigma_x([e^{tiy}]) > 0,
$$
for all $y \in \LieH$.

As stated in the introduction and we will see throughout the whole paper, the application of the previous abstract picture to the situation of gravitating vortices is far from straighforward.

\subsection{Reduction by stages approach to gravitating vortices}\label{section:setupsub}
In this section we explain how the existence and uniqueness problem for the gravitating vortex equations \eqref{eqn:GV0} fits into the general abstract framework of Section \ref{section:setupabstract}, and present
our overall strategy to the problem via reduction by stages. We follow closely \cite{Al-Ga-Ga2,Al-Ga-Ga-P}.

Fix a smooth compact oriented surface $S$ of genus $g(S)$ endowed with a smooth volume form $\omega \in \Omega^2(S)$. We will regard $\omega$ as a symplectic form on $S$. Let $(L,h)$ be a smooth Hermitian line
bundle over $S$ of degree
$$
N = \int_S c_1(L) \in \mathbb{Z}.
$$
Let $\cJ$ be the space of almost complex structures on $S$ compatible with $\omega$. Let $\cA$ be the space of unitary connections on $(L,h)$ and $\cS = \Omega^0(L)$ the space of smooth sections of $L$. Consider the
space of triples
\begin{equation}\label{eq:spc-triples}
\cJ\times\cA\times\cS,
\end{equation}
endowed with the symplectic structure
\begin{equation}\label{eq:symplecticT}
\omega_{\alpha} = \omega_\cJ + 4 \alpha \omega_\cA + 4\alpha \omega_\cS
\end{equation}
(for a choice of positive real coupling constant $\alpha > 0$). Here, for $(J,A, \bm\phi) \in \cJ\times\cA\times\cS$, we denote
$$
\omega_\cJ(\dot J_1,\dot J_2) = \frac{1}{2}\int_S \tr(J\dot J_1 \dot J_2)\omega, \;\,\omega_\cA(a_1,a_2) = - \int_S a_1 \wedge a_2,\;\,\omega_\cS(\dot{\bm\phi}_1,\dot{\bm\phi}_2) = -\operatorname{Im} \int_S
h(\dot{\bm\phi}_1,\dot{\bm\phi}_2)\omega
$$
where $\dot{\bm\phi}_1,\dot{\bm\phi}_2 \in T_\phi \cS = \Omega^0(L)$, $a_1,a_2 \in T_A\cA = i\Omega^1(S)$, and $\dot J_1,\dot J_2 \in T_J \cJ$, identified with the space of $g_J$-symmetric tensors which anticommute
with $J$, where $g_J = \omega(,J)$.

Let $\cH$ be the group of Hamiltonian diffeomorphisms of $(S,\omega)$ and $\Aut (L,h)$ the extended group of unitary automorphisms of the $(L,h)$. Any such automorphism $\psi$ covers a unique
diffeomorphism $\check{\psi}\colon S\to S$ and we set $p(\psi) = \check{\psi}$. We define the
\emph{extended gauge group}
\[
\cX \subset \Aut (L,h),
\]
as the group of automorphisms which cover elements of $\cH$. Then the unitary gauge group of $(L,h)$  is the
normal subgroup $\cG\subset\cX$ of automorphisms covering the identity, and we have an exact sequence of Lie groups
\begin{equation}
\label{eq:coupling-term-moment-map-1}
  1\to \cG \lra{\iota} \cX \lra{\pr} \cH \to 1,
\end{equation}
where $\iota$ is the inclusion map. The natural diagonal left action of $\cX$ on $\cJ\times\cA\times\cS$ by push-forward is Hamiltonian (see \cite[Section 3.2]{Al-Ga-Ga-P}), and there is an equivariant moment map
$\mu_{\cX}\colon\cJ\times\cA\times\cS\to(\LieX)^*$ given
by
\begin{equation}\label{eq:mutriples}
\begin{split}
\langle \mu_{\cX}(J,A,\bm\phi),\zeta\rangle = & - 4\alpha \int_S A\zeta\!\cdot\!\(F_{A} - \frac{i}{2} (\lvert\bm\phi\rvert_{h}^2 - \tau)\omega\)\\
&- \int_S  f\(S_{J}\omega + 2i\alpha\mathrm{d}(\mathrm{d}_A \bm\phi,\bm\phi)_h - 2i\alpha \tau F_A\),
\end{split}
\end{equation}
where $S_J$ denotes the scalar curvature of $g_J = \omega(,J)$ and $F_A$ is the curvature of the connection $A$. Here, $A \zeta \in iC^\infty(S,\mathbb{R})$ is the vertical part of $\zeta \in \LieX$ with respect to
$A$, identified with a vector field on the total space of $L$, and $f$ is the Hamiltonian function of the vector field $p(\zeta) \in \LieH$ on $S$, normalized so that $\int_S f \omega = 0$. The parameter $\tau \in
\mathbb{R}$, often called \emph{symmetry breaking parameter}, can be chosen at will in the previous formula.

Following the discussion in the previous section, one can formally consider the symplectic reduction of $\cT$ by $\cG$, that is, $\mu_{\cG}^{-1}(0)/\cG$, which gives a relative moduli space of solutions of the
vortex equation
\begin{equation}
\label{eqn:V0}
iF_{A} + \frac{1}{2} (\lvert\bm\phi\rvert_{h}^2 - \tau)\omega
 = 0
\end{equation}
over the space of complex structures $\cJ$, with possibly non-holomorphic section $\phi$. This space inherits then a Hamiltonian $\cH$-action with moment map given by the second line in \eqref{eq:mutriples}. Notice
that in order for the symplectic $\cG$-reduction to be non-empty, one already needs to impose (provided that $\bm\phi\neq 0$)
\begin{equation}
\label{eqn:ineqtauV}
V> \frac{4\pi N}{\tau},
\end{equation}
where $V = \int_S \omega$ is the total volume of $S$ with respect to $\omega$.

The moduli space of solutions of the gravitating vortex equations arises when we restrict to a natural symplectic submanifold of the space of triples $\cJ\times\cA\times\cS$, dictated by its K\"ahler geometry.
Similarly as in \cite[Proposition 2.2]{Al-Ga-Ga}, the space \eqref{eq:spc-triples} has a formally integrable almost complex structure, given by
$$
\mathbf{J}(\dot J,a,\dot{\bm\phi}) = (J \dot J,- a(J \cdot), i \dot{\bm\phi})
$$
where $(\dot J,a,\dot{\bm\phi})$ is a tangent vector at $(J,A,\bm\phi)$, which is compatible with \eqref{eq:symplecticT} when $\alpha >0$, thus inducing a K\"ahler structure in this case. Consider the subspace of
\emph{integrable triples}
$$
\cT = \{(J,A,\bm\phi) \mid \dbar_{J,A} \bm\phi = 0 \}
$$
where $\dbar_{J,A}\bm\phi$ denotes the $(0,1)$-part of $d_A \bm\phi$ with respect to $J$. This space is in bijection with the space of Riemann surfaces $\Sigma_J = (S,J)$ of genus $g(S)$ equipped with a holomorphic
line bundle $(L,\dbar_{J,A})$ with holomorphic section $\phi$. This implies, in particular, that the degree of $L$ must be a positive integer
$$
N > 0,
$$
provided that $\cT$ is non-empty, and we shall assume this in the sequel. By definition, $\cT$ is a complex subspace of \eqref{eq:spc-triples} (away from its singularities) preserved by the $\cX$-action, and hence
it inherits a Hamiltonian $\cX$-action.

\begin{proposition}[\cite{Al-Ga-Ga-P}]\label{prop:momentmap-inttriples}
The $\cX$-action on $\cT$ is Hamiltonian with $\cX$-equivariant moment map $\mu_{\cX}\colon\cT\to(\LieX)^*$ given by
\begin{equation}\label{eq:prop-mutriples}
\begin{split}
\langle \mu_{\cX}(J,A,\bm\phi),\zeta\rangle & = - 4\alpha \int_S A\zeta\!\cdot\! \(F_{A} - \frac{i}{2} (\lvert\bm\phi\rvert_{h}^2 - \tau)\omega\) - \int_S  f\(\omega\(S_{J} + \alpha
\Delta_{g}\lvert\bm\phi\rvert^2_h\) - 2\alpha \tau iF_A
\)
\end{split}
\end{equation}
for all $(J,A,\bm\phi)\in\cT$ and $\zeta\in\LieX$, where $\Delta_g = 2i \Lambda_\omega \dbar_J \partial_J$ is the Laplacian operator of the Riemannian metric $g_J = \omega(\cdot,J\cdot)$.
\end{proposition}

Formula \eqref{eq:prop-mutriples} follows from \eqref{eq:mutriples} using the K\"ahler identities and the fact that $\bm\phi$ is a holomorphic section of $(L,\dbar_{J,A})$, combined with the Bochner formula
\begin{equation}\label{eq:Bochner.1}
\frac{1}{2}\Delta_\omega\lvert\bm\phi\rvert_h^2=-i\Lambda\partial\bar{\partial}\lvert\bm\phi\rvert_h^2= i\Lambda_\omega F_A \lvert\bm\phi\rvert_h^2-\lvert\partial_A\bm\phi\rvert_{h,\omega}^2,
\end{equation}
where $\lvert\partial_A\bm\phi\rvert_{h,\omega}^2=i\Lambda_\omega (\partial_A\bm\phi\wedge(\partial_A\bm\phi)^*)$. It is not difficult to see that the zeroes of the moment map $\mu_{\cX}$ correspond precisely to
integrable triples $(J,A,\bm\phi) \in \cT$ satisfying the gravitating vortex equations \eqref{eqn:GV0} with constant
\begin{equation}\label{eq:constantc}
c = \frac{2\pi(\chi(S) - 2 \alpha \tau N)}{V}.
\end{equation}
The moduli space of solutions of the gravitating vortex equations, given by the symplectic quotient
$$
\mathcal{M} = \mu_{\cX}^{-1}(0)/\cX
$$
inherits, by means of standard gauge-theoretical considerations \cite{GFTipler}, a K\"ahler structure on its smooth locus.
$\cT$.
We should notice that the role played by the orbits of the complexification in the definition of the stable locus of $\cT$ is now played by the integral leaves of the distribution
$$
D_{p} = \{Y_{\zeta_1}(p) + \mathbf{J} Y_{\zeta_2}(p) \mid \zeta_1,\zeta_2\in \LieX \} \subset T_p \cT,
$$
for $p \in \cT$, where $Y_{\zeta_j}(p)$ is the infinitesimal action of $\zeta_j$ at $p$.

In order to yield a deeper understanding of the global structure of the moduli space $\mathcal{M}$, including its singularities, we adopt a \emph{reduction by stages approach} to the gravitating vortex equations.
For this, we build on the following existence theorem for the vortex equation \eqref{eqn:V0}, established independently by Noguchi, Bradlow and Garc\'ia-Prada. For completeness, in Appendix~\ref{sub:variational-approach-vortex-eq} we provide a direct proof of this theorem with the methods of the present work.

\begin{theorem}[{\cite{Brad,Prada,Noguchi}}]\label{th:B-GP}
Let $\Sigma$ be a compact connected Riemann surface endowed with a holomorphic line bundle $L$ with non-zero holomorphic section
$$
0 \neq \bm\phi \in H^0(\Sigma,L).
$$
Let $\omega$ be a K\"ahler form on $\Sigma$ and assume that the symmetry breaking parameter $\tau$ satisfies the inequality \eqref{eqn:ineqtauV}. Then there exists a unique Hermitian metric $h_\omega$ on $L$ solving
the vortex equation~\eqref{eqn:V0}, where $A$ is the Chern connection of $h_\omega$.
\end{theorem}

In the language of Section \ref{section:setupabstract}, this result can be interpreted as follows. Assume that $\omega$ and $\tau$ are chosen, \emph{a priori}, so that the inequality \eqref{eqn:ineqtauV} is
satisfied. Then for any $p = (J,A,\bm\phi) \in \cT$ there exists a complex gauge transformation $g_p \in \cG^c$, unique up to the action of the unitary gauge group $\cG$, such that $g_p \cdot p$ is a solution of the
vortex
equation~\eqref{eqn:V0}, that is,
$$
\mu_\cG(g_p \cdot p) = 0,
$$
where $\mu_\cG = \iota^* \circ \mu_{\cX}$ for $\iota$ the inclusion in \eqref{eq:coupling-term-moment-map-1}. The metric $h_\omega$ in Theorem \ref{th:B-GP} corresponds to the pull-back of the fixed
metric $h$ by $g_p$. Therefore $\cT^s_{\cG^c} = \cT$ and the symplectic $\cG$-reduction of $\cT$ by $\cG$, that is,
$$
\mathcal{N} = \mu_{\cG}^{-1}(0)/\cG,
$$
admits a more amenable description as the relative symmetric product $S^N \Sigma_J$ of the universal family of Riemann surfaces $\Sigma_J = (S,J)$ over the space of complex structure $\cJ$ on $S$ \cite{Prada}. The
induced Hamiltonian $\cH$-action on $\mathcal{N}$ has then an equivariant moment map
\begin{equation}\label{eq:prop-mutriplesreduced}
\begin{split}
\langle \mu_{\cH}(J,[(A, \bm\phi)]),y\rangle & = - \int_S  f\(\omega\(S_{J} + \alpha \Delta_{g}\lvert\bm\phi\rvert^2_h\) - 2i\alpha \tau F_{A}\),
\end{split}
\end{equation}
where $y \in \LieH$ is a Hamiltonian vector field on $(S,\omega)$ with normalized Hamiltonian $f$. Note that, by construction, the connection $A$ in the previous formula solves the vortex equation on the Riemann
surface $\Sigma_J$. By Section \ref{section:setupabstract}, the K\"ahler moduli space $\mathcal{M}$ can be now understood as the K\"ahler reduction of $\mathcal{N}$ by $\cH$.

The main goal of the present paper is to give a characterization of the complexified $\cX$-orbits on $\cT$ which cut the zero locus of the moment map $\mu_{\cX}$, that is, which contain a solution of the gravitating
vortex equations. By the previous discussion and \cite[Section 3]{Al-Ga-Ga}, a complexified orbit is represented by a compact connected Riemann surface $\Sigma$ (with genus $g(S)$) endowed with an effective divisor
$$
D = \sum_j n_j p_j,
$$
corresponding to a holomorphic line bundle $L\to \Sigma$ with holomorphic section $\phi \in H^0(\Sigma,L)$. Moving along the complexified orbit is equivalent to, up the action of $\cX$, varying the K\"ahler metric
on $\Sigma$ with fixed total volume $V$ and the Hermitian metric $h$ on $L$. This is the \emph{complex viewpoint} on the gravitating vortex equations presented in Section \ref{sec:intro}. Following the reduction by
stages approach, we now consider the following equation for a K\"ahler metric on $\omega$ with volume $V$, which we shall call the \emph{reduced gravitating vortex equation}
\begin{equation}\label{eqn:GV0red}
S_{\omega} + \alpha ( \Delta_{\omega} + \tau )( \lvert\bm\phi\rvert_{h_\omega}^2 - \tau )  = c,
\end{equation}
where $c$ is given by \eqref{eq:constantc} and $h_\omega$ is the unique solution of the vortex equation associated to $\omega$ via Theorem \ref{th:B-GP}. From a physical view point, this deformation of the usual
constant scalar curvature condition on the K\"ahler metric $\omega$ takes implicit account of the back-reaction of the metric on the vortices.

Similarly as in Section \ref{section:setupabstract}, consider the infinite-dimensional space
$$
\mathcal{P} = \left\{\omega\textrm{ K\"ahler }\Bigm\vert\;\int_\Sigma\omega=V\right\}
$$
of K\"ahler metrics with total volume $V>0$, which can regarded as the quotient $\cH^c/\cH$. This has a natural structure of infinite-dimensional symmetric space, endowed with the Mabuchi $L^2$-metric
\cite{Mabuchi}. Applying the formal properties of the reduced moment map \eqref{eq:prop-mutriplesreduced}, one can define a closed 1-form on $\mathcal{P}$. By contractibility of $\mathcal{P}$, this 1-form integrates
to a functional, whose critical points are the solutions of the reduced gravitating vortex equation \eqref{eqn:GV0red} and which is convex along geodesics on $\mathcal{P}$ (see \cite[Proposition 3.10]{Al-Ga-Ga} and
Section \ref{section:setupabstract}). Instead of working with $\mathcal{P}$ directly, it will be useful to fix a reference background K\"ahler metric $\omega_0 \in \mathcal{P}$ and consider the space of K\"ahler
potentials
\begin{equation}\label{eq:sp-Kahler-potentials.1}
\mathcal{P}_{\omega_0} = \left\{\varphi \mid \omega_\varphi := \omega_0 + 2i\partial\bar\partial \varphi>0 \right\} \subset C^\infty(\Sigma,\RR).
\end{equation}
Define a 1-form on $\mathcal{P}_{\omega_0}$ by
\begin{equation}\label{eq:sigmaGVreduced}
\sigma_\alpha|_\varphi(\dot\varphi) = - \int_S  \dot \varphi\(\omega_\varphi\(S_{\omega_\varphi} + \alpha \Delta_{\omega_\varphi}\lvert\bm\phi\rvert^2_{h_\varphi}\) - 2\alpha \tau iF_{h_\varphi} - c \omega_\varphi\)
\end{equation}
where $\dot \varphi \in C^\infty(\Sigma,\RR) $ is tangent at $\varphi$ and $h_\varphi$ is the unique solution of the vortex equation \eqref{eqn:GV0red} associated to $\omega_\varphi$ via Theorem \ref{th:B-GP}.
Applying the formal properties of the reduced moment map \eqref{eq:prop-mutriplesreduced} one can prove that $\sigma_\alpha$ is closed and, since the space $\mathcal{P}_{\omega_0}$ is contractible, it integrates to
a functional which we call the \emph{reduced $\alpha$-K-energy}
$$
\mathcal{K}_\alpha \colon \mathcal{P}_{\omega_0} \lto \RR
$$
(upon a choice of constant $C \in \mathbb{R}$ such that $\mathcal{K}_\alpha(\omega_0) = C$). We state next the main result of this section, whose proof is given in Section \ref{sec:reducedK}. By a smooth path on
$\mathcal{P}_{\omega_0}$, denoted $\varphi\colon[0,1]\to \mathcal{P}_{\omega_0},\, t\mapsto\varphi_t$, we mean a smooth function $\varphi \in C^\infty([0,1]\times\Sigma)$ such that $\varphi_t := \varphi(t,\cdot) \in
\mathcal{P}_{\omega_0}$ for all $t \in
[0,1]$. In the sequel, we denote $\varphi'_t\defeq\partial\varphi_t/\partial t$ and $\varphi''_t\defeq\partial\varphi'_t/\partial t$.

\begin{proposition}\label{prop:convexalphaK}
The reduced $\alpha$-K-energy $\mathcal{K}_\alpha$ is convex along smooth geodesic paths $\varphi_t$ on $\mathcal{P}_{\omega_0}$, given by smooth solutions to the partial differential equation
\begin{equation}\label{eq:geodesiceq}
\varphi''_t - \lvert\mathrm{d}\varphi'_t\rvert^2_{\omega_{\varphi_t}} = 0.
\end{equation}
Furthermore, its critical points are given by solutions of the reduced gravitating vortex equation \eqref{eqn:GV0red} on $\mathcal{P}_{\omega_0}$.
\end{proposition}

\begin{remark}\label{rem:liftRtoRom}
Consider $\mathcal{P} \subset \mathcal{P}_{\omega_0}$, regarded as the subspace of normalized K\"ahler potentials $\varphi$ with
$$
\int_\Sigma \varphi \omega_\varphi =  - \int_\Sigma \varphi \omega_0
$$
(see Section \ref{sub:metric-completion-Kahler-potentials.1} for the choice of normalization). Then one can prove that the pull-back of $\sigma_\alpha$ via the inclusion map coincides with the 1-form obtained via
the reduced moment map \eqref{eq:prop-mutriplesreduced} (see \cite[Section 3.3]{Al-Ga-Ga1}). The corresponding functional is convex along smooth geodesic paths $\omega_t = \omega_0 + \mathrm{d}\mathrm{d}^c
\varphi_t$ on the space of K\"ahler metrics, given by
\begin{equation*}\label{eq:geodesiceqR}
\mathrm{d}\mathrm{d}^c(\varphi''_t - \lvert\mathrm{d}\varphi'_t\rvert^2_{\omega_t}) = 0.
\end{equation*}
\end{remark}

\subsection{Reduced $\alpha$-K-energy}\label{sec:reducedK}

The goal of this section is to provide a direct and self-contained proof of Proposition \ref{prop:convexalphaK}, which provides an alternative to the abstract argument in the previous section (based on
\cite{Al-Ga-Ga}), and an explicit formula for the reduced $\alpha$-K-energy that will be needed in the rest of this article.

We start by recalling the construction of the $\alpha$-K-energy for gravitating vortices following \cite{Al-Ga-Ga1,Al-Ga-Ga-P}. We fix a compact connected Riemann surface $\Sigma$ (with genus $g(\Sigma)$) endowed
with a holomorphic line bundle $L\to \Sigma$ with holomorphic section $\bm\phi \in H^0(\Sigma,L)$. Let $\widetilde{\mathcal{P}}$ denote the product of the space of K\"ahler metrics on $\Sigma$ with fixed volume $V >
0$ with the space of Hermitian metrics on $L$. It will be useful to work with a covering of this space, given by a choice of K\"ahler potential, namely,
$$
\widetilde{\mathcal{P}}_{\omega_0} = \left\{(h,\varphi) \mid \omega_\varphi := \omega_0 + 2i\partial\bar\partial \varphi>0 \right\},
$$
where by $h$ we denote Hermitian metrics on $L$ and $\omega_0$ is a choice of background K\"ahler form with $\int_\Sigma \omega_0 = V$. We shall also fix a background Hermitian metric $h_0$ on $L$, and denote $h_f =
e^{2f} h_0$ for any $f \in C^\infty(\Sigma,\RR)$. Thus $\widetilde{\mathcal{P}}_{\omega_0}$ can be regarded as an open subset
\begin{equation}\label{eq:sp-metric-potentials.1}
\widetilde{\mathcal{P}}_{\omega_0} \subset C^\infty(\Sigma,\RR) \times C^\infty(\Sigma,\RR).
\end{equation}
We fix a symmetry breaking parameter $\tau$ satisfying the inequality \eqref{eqn:ineqtauV}. Following \cite{Yao}, we define a map
$$
\wp \colon \widetilde{\mathcal{P}}_{\omega_0} \lto C^\infty(\Sigma,\RR) \times C^\infty(\Sigma,\RR)
$$
by
\begin{equation}\label{eq:wp.1}
\begin{split}
\wp_1(f, \varphi) & = i \Lambda_{\omega_\varphi}F_{h_f} +  \frac{1}{2}\lvert\bm\phi\rvert_{h_f}^2 - \frac{\tau}{2},\\
\wp_2(f, \varphi) & =   - S_{\omega_\varphi} -\alpha \Delta_{\omega_\varphi} \lvert\bm\phi\rvert_{h_f}^2 +
2i\alpha\tau \Lambda_{\omega_\varphi} F_{h_f} + c,
\end{split}
\end{equation}
where $\wp = (\wp_1,\wp_2)$, we use the identification $(f,\varphi) \equiv (h_f, \varphi) \in \widetilde{\mathcal{P}}_{\omega_0}$, and $c$ is the topological constant defined in \eqref{eq:constantc}. Define a
$1$-form on the space $\widetilde{\mathcal{P}}_{\omega_0}$ by
\begin{equation}\label{defn:primitiveabs}
\begin{split}
\widetilde{\sigma}_\alpha|_{(f, \varphi)}(\dot f, \dot \varphi)
& =
4\alpha \int_\Sigma \dot f \wp_1(f, \varphi)\omega_\varphi
+
\int_\Sigma \dot \varphi \wp_2(f, \varphi)\omega_\varphi\\
\end{split}
\end{equation}
where $(\dot f, \dot \varphi) \in C^\infty(\Sigma,\RR) \times C^\infty(\Sigma,\RR)$ represents a tangent vector at $(f,\varphi) \in \widetilde{\mathcal{P}}_{\omega_0}$. More explicitly, we have
\begin{align}
\nonumber
\begin{split}
\widetilde{\sigma}_\alpha|_{(f, \varphi)}(\dot f, \dot \varphi) & =
4\alpha
\int_\Sigma \dot f
\left( iF_{h_f} + \frac{1}{2} (\lvert\bm\phi\rvert_{h_f}^2 - \tau) \omega_\varphi
\right)
\\&\quad
-
\int_\Sigma \dot \varphi
\left( \text{Ric }\omega_\varphi - 2\alpha i\partial\bar\partial \lvert\bm\phi\rvert_{h_f}^2 - 2\alpha\tau iF_{h_f}
- c\omega_\varphi
\right)
\end{split}
\\\label{defn:primitive}
\begin{split}
& =
- \int_\Sigma \dot \varphi
\left( \text{Ric }\omega_\varphi - \langle S\rangle \omega_\varphi\right)
+
4\alpha
\int_\Sigma \dot f
\left( iF_{h_f} + \frac{1}{2} (\lvert\bm\phi\rvert_{h_f}^2 - \tau) \omega_\varphi
\right)
\\&\quad
-
\int_\Sigma \dot \varphi
\left( \left( \langle S\rangle - c\right) \omega_\varphi - 2\alpha i\partial\bar\partial \lvert\bm\phi\rvert_{h_f}^2 - 2\alpha\tau iF_{h_f}
\right),
\end{split}
\end{align}
where $\langle S\rangle$ denotes the average of the scalar curvature for any K\"ahler metric $\omega$ with total volume $V$
$$
\langle S\rangle := \frac{1}{V}\int_\Sigma S_\omega \omega = \frac{4\pi(1 - g(\Sigma))}{V}.
$$

\begin{lemma}\label{lem:sigmaclosed}
The 1-form $\widetilde{\sigma}_\alpha$ is closed.
\end{lemma}

\begin{proof}
Note that, for any Hermitian metric $h_0$ on $L$ and $f\in C^\infty(\Sigma,\mathbb{R})$,
\begin{equation}\label{eq:transformation-Hermitian-metric.1}
F_{h_f}=F_{h_0}-2\partial\bar{\partial}f.
\end{equation}
Using this, one has the following formulae (see \cite[Equation (3.3)]{Yao}):
\begin{align}\label{eq:formulaeYao}\nonumber
\delta\wp_1|_{(f, \varphi)}(\dot f, \dot \varphi)
& =
\Delta_\omega \dot f
+
\frac{1}{2}(\tau - \lvert\bm\phi\rvert_h^2) \Delta_\omega\dot \varphi
+
\dot f \lvert\bm\phi\rvert_h^2
+
\wp_1(f, \varphi)\Delta_\omega \dot \varphi,
\\\nonumber
\delta\wp_2|_{(f, \varphi)}(\dot f,  \dot \varphi)
& =
\frac{1}{2}\Delta_\omega^2 \dot \varphi
-
2\alpha \Delta_\omega (\dot f\lvert\bm\phi\rvert_h^2)
+
2\alpha\tau \Delta_\omega \dot f
+
\wp_2(f, \varphi)\Delta_\omega \dot \varphi,
\\\nonumber
S_\omega
& =
-\wp_2(f, \varphi) + 2\alpha(\tau-\lvert\bm\phi\rvert_h^2)\wp_1(f, \varphi)
+
2\alpha\lvert\mathrm{d}_A\bm\phi\rvert_h^2
+
\alpha(\tau-\lvert\bm\phi\rvert_h^2)^2 + c,
\\\nonumber
4\alpha\langle \dot f, \delta\wp_1(\dot f, \dot \varphi)\rangle
& +
\langle \dot \varphi, \delta\wp_2(\dot f, \dot \varphi)\rangle
\\\nonumber
& =
4\alpha \lVert \mathrm{d}\dot f + \eta_{\dot \varphi}\lrcorner (iF_h - \wp_1\omega)\rVert^2
+
4\alpha \lVert J\eta_{\dot \varphi}\lrcorner \mathrm{d}_A\bm\phi-\dot f\bm\phi\rVert^2
+
2\lVert\bar\partial \nabla^{1,0}\dot \varphi\rVert^2
\\\nonumber
&\quad
+ 4\alpha \int_\Sigma \left(
\dot f\Delta_\omega \dot \varphi
+
(\tau-\lvert\bm\phi\rvert_h^2) \lvert\nabla^{1,0}\dot\varphi\rvert^2
\right)\wp_1(f, \varphi)\omega
\\
&\quad
+\int_\Sigma \left(
\dot \varphi\Delta_\omega \dot \varphi
-2\lvert\nabla^{1,0}\dot\varphi\rvert^2
\right)\wp_2(f, \varphi)\omega,
\end{align}
where $\delta \wp_j$ denotes the variation of $\wp_j$, and $(h_f, \omega_\varphi)$ is abbreviated as $(h,\omega)$ to simplify the notation. Here we use the notation
$$
\left| df\right|_\omega^2 = 2\left|\nabla^{1,0}f\right|_\omega^2=2g^{j\bar k}\frac{\partial f}{\partial z^j}\overline{\frac{\partial f}{\partial z^k}}, \qquad \Delta_\omega f= - \Lambda_\omega \mathrm{d}\mathrm{d}^c
f =  - 2 g^{j\bar k}\frac{\partial^2 f}{\partial z^j\partial \bar z^k},
$$
for any real valued function $f$, where $\omega = i g_{j\bar k}dz^j\wedge d\bar z^k$ in local complex coordinates.

Regarding now $(\dot g,\dot u), (\dot f,\dot \varphi) \in C^\infty(\Sigma,\RR) \times C^\infty(\Sigma,\RR)$ as a pair of constant vector fields on $\widetilde{\mathcal{P}}_{\omega_0}$, the above formulae imply
\begin{align}\label{eq:differential-mathfrak-S.1}
\nonumber
\delta_{(\dot g,\dot u)}&(\widetilde{\sigma}_\alpha\dot f, \dot \varphi))
=
\delta_{(\dot g, \dot u)}
\left(
4\alpha \int_\Sigma \dot f \wp_1(f, \varphi)\omega
+
\int_\Sigma \dot \varphi \wp_2(f, \varphi)\omega
\right)
\\\nonumber
& =
4\alpha
\int_\Sigma \dot f
\Big(
\delta\wp_1(\dot g, \dot u) - \wp_1 \Delta_\omega \dot u
\Big) \omega
+
\int_\Sigma \dot \varphi
\Big(
\delta\wp_2(\dot g, \dot u) - \wp_2\Delta_\omega \dot u \Big) \omega
\\\nonumber
& =
4\alpha
\int_\Sigma \left[ \dot f
\Big(
\Delta_\omega \dot g + \frac{1}{2}(\tau-\lvert\bm\phi\rvert_h^2) \Delta_\omega \dot u
+ \dot g \lvert\bm\phi\rvert_h^2
+ \wp_1 \Delta_\omega \dot u
\Big)
-
\dot f \wp_1 \Delta_\omega \dot u\right]\omega
\\
&
\qquad + \int_\Sigma \left[ \dot \varphi
\Big(
\frac{1}{2}\Delta_\omega^2 \dot u
-2\alpha \Delta_\omega (\dot g \lvert\bm\phi\rvert_h^2)
+2\alpha\tau \Delta_\omega \dot g
+ \wp_2 \Delta_\omega \dot u
\Big)
- \dot \varphi \wp_2 \Delta_\omega \dot u\right]\omega
\\\nonumber
& =
\frac{1}{2}\int_\Sigma \Delta_\omega \dot\varphi \Delta_\omega \dot u \omega
+
4\alpha
\int_\Sigma \left[ \nabla_\omega \dot f\cdot \nabla_\omega \dot g
+ \frac{1}{2} (\tau-\lvert\bm\phi\rvert_h^2)
(\dot f \Delta_\omega \dot u
+ \dot g \Delta_\omega \dot \varphi)
+
\dot f\dot g \lvert\bm\phi\rvert_h^2\right]\omega,
\end{align}
which is symmetric in $(\dot f, \dot \varphi)$ and $(\dot g, \dot u)$, and therefore $\widetilde{\sigma}_\alpha$ is a closed 1-form.
\end{proof}

\begin{remark}
Note that \eqref{defn:primitiveabs} and \cite[Equation (5.2)]{Al-Ga-Ga-P} differ by a sign in the first summand. This is due to a different parametrization of the space of Hermitian metrics on $L$.
\end{remark}

By the previous result and the contractibility of $\widetilde{\mathcal{P}}_{\omega_0}$, the 1-form $\widetilde{\sigma}_\alpha$ integrates to a functional, called the \emph{$\alpha$-K-energy},
$$
\widetilde{\mathcal{K}}_\alpha \colon \widetilde{\mathcal{P}}_{\omega_0} \lto \RR,
$$
normalized such that $\widetilde{\mathcal{K}}_\alpha (0,0) = 0$. It is easily seen that the critical points of $\widetilde{\mathcal{K}}_\alpha$ correspond to gravitating vortices, that is, solutions $(h,\omega)$ of
the gravitating vortex equations \eqref{eqn:GV0}.

\begin{proposition}\label{prop:explicit}
For any choice of background metrics $(h_0,\omega_0)$, we have the following  explicit formula for $\widetilde{\mathcal{K}}_\alpha$:
	\begin{equation*}
	\widetilde{\mathcal{K}}_\alpha(f, \varphi) =
	\mathcal{K}(\varphi)
	+
	\mathcal{M}_\alpha(f,\varphi)
	\end{equation*}
where
	 \[
	\mathcal{K}(\varphi)= \frac{1}{2}\int_\Sigma \log \frac{\omega_\varphi}{\omega_0} \omega_\varphi +
	\frac{\langle S\rangle}{2}\int_\Sigma \varphi \left(\omega_\varphi - \omega_0\right)
	\]
is the standard K-energy functional and
	\begin{align*}
	\mathcal{M}_\alpha(f,\varphi)
	& =
	2\alpha
	\int_\Sigma
	f\left(
	iF_{h_0} + iF_{h_f}
	\right)
	+
	\alpha \int_\Sigma
	\left(
	\lvert\bm\phi\rvert_{h_f}^2 \omega_\varphi
	-
	\lvert\bm\phi\rvert_{h_0}^2\omega_0
	\right)
	-2\alpha\tau
	\int_\Sigma f\omega_\varphi \\
	&\qquad +
	\left( c- \langle S\rangle \right) \int_\Sigma \varphi i\partial\bar\partial \varphi
	+ \int_\Sigma \varphi\left(-\Ric \omega_0+ 2\alpha\tau iF_{h_0} +c\omega_0\right).
	\end{align*}
vanishes.
\end{proposition}

\begin{proof}
This follows by direct integration along the straight line $(f_t, \varphi_t)=(tf,t\varphi)$ connecting $(0,0)$ to $(f, \varphi)$. Using the abbreviations $h_t=h_0e^{2f_t}$, $\omega_t=\omega_0+2i\partial\bar\partial
\varphi_t$, we have
\begin{align*}
	\widetilde{\mathcal{K}}_\alpha&(f, \varphi)
 =
	4\alpha\int_0^1 \mathrm{d}t
	\int_\Sigma
	f\left(
	i F_{h_t} + \frac{1}{2} (\lvert\bm\phi\rvert_{h_t}^2-\tau)
	\omega_t
	\right)
\\&\quad
+	\int_0^1 \mathrm{d}t
	\int_\Sigma \varphi
	\Big(-\text{Ric } \omega_t + 2\alpha i\partial\bar\partial \lvert\bm\phi\rvert_{h_t}^2 + 2\alpha\tau iF_{h_t}
	+
	c\omega_t
	\Big)\\
	& =
	4\alpha\int_0^1 \mathrm{d}t
	\int_\Sigma
	f\left(
	i F_{h_t} + \frac{1}{2} (\lvert\bm\phi\rvert_{h_t}^2-\tau)
	\omega_t
	\right)
\\
	& \quad
	+\int_0^1 \mathrm{d}t\int_\Sigma \varphi \cdot	i\partial\bar\partial\left(\log\frac{\omega_t}{\omega_0} + 2\alpha\lvert\bm\phi\rvert_{h_t}^2 - 4\alpha\tau f_t+2c \varphi_t\right)
\\&\quad
	+ \int_\Sigma \varphi\left(-\Ric \omega_0+ 2\alpha\tau iF_{h_0} +c\omega_0\right)\\
	& =
	4\alpha\int_0^1 \mathrm{d}t
	\int_\Sigma
	f\Big(
	(1-t) iF_{h_0} + t\cdot iF_{h_1}
	\Big)
	+
	\alpha \int_0^1 \mathrm{d}t\int_\Sigma 2f\lvert\bm\phi\rvert_{h_t}^2 \omega_t
	-2\alpha\tau \int_0^1\mathrm{d}t\int_\Sigma f\omega_t\\
	& \quad
	+
	\frac{1}{2}\int_0^1 \mathrm{d}t
	\int_\Sigma \log \frac{\omega_t}{\omega_0} (-\Delta_{\omega_t} \varphi)\omega_t
	+
	\alpha\int_0^1\mathrm{d}t \int_\Sigma \lvert\bm\phi\rvert_{h_t}^2 \cdot 2i\partial\bar\partial \varphi \\
	&\quad
	-
	4\alpha\tau \int_0^1 \mathrm{d}t \int_\Sigma t f \cdot i\partial\bar\partial \varphi
	+
	2c \int_0^1 dt \int_\Sigma \varphi \cdot i\partial\bar\partial \left( t\varphi\right)
	\\
	&\quad
	+
	 \int_\Sigma \varphi\left(-\Ric \omega_0+ 2\alpha\tau iF_{h_0} +c\omega_0\right)
	\\
	& =
	2\alpha
	\int_\Sigma
	f\left(
	iF_{h_0} + iF_{h_1}
	\right)
	+
	\alpha \int_0^1 \mathrm{d}t\int_\Sigma
	\left(
	\frac{\mathrm{d}}{\mathrm{d}t}
	\left(
	\lvert\bm\phi\rvert_{h_t}^2 \omega_t
	\right)
	-
	\lvert\bm\phi\rvert_{h_t}^2 \cdot 2i\partial\bar\partial \varphi
	\right)
	\\
	&\quad
	-\alpha\tau
	\int_\Sigma f\left( \omega_1+\omega_0\right)+
	\frac{1}{2}\int_0^1 \mathrm{d}t
	\int_\Sigma
	\left(
	\frac{\mathrm{d}}{\mathrm{d}t}
	\left( \log \frac{\omega_t}{\omega_0} \omega_t \right)
	+
	(\Delta_{\omega_t} \varphi)\omega_t
	\right)
	\\
	&\quad
	+
	\alpha\int_0^1\mathrm{d}t \int_\Sigma \lvert\bm\phi\rvert_{h_t}^2\cdot 2i\partial\bar\partial \varphi
	-\alpha\tau
	\int_\Sigma  f (\omega-\omega_0)
	+ c \int_\Sigma \varphi i\partial\bar\partial \varphi
	\\
	&\quad
	+
	 \int_\Sigma \varphi\left(-\Ric \omega_0+ 2\alpha\tau iF_{h_0} +c\omega_0\right)
	\\
	& =
	2\alpha
	\int_\Sigma
	f\left(
	iF_{h_0} + iF_{h_1}
	\right)
	+
	\alpha \int_\Sigma
	\left(
	\lvert\bm\phi\rvert_{h_1}^2 \omega_1
	-
	\lvert\bm\phi\rvert_{h_0}^2\omega_0
	\right)
	-\alpha\tau
	\int_\Sigma f\left( \omega_1+\omega_0\right)\\
	& \quad
	+
	\frac{1}{2}
	\int_\Sigma
	\log \frac{\omega_1}{\omega_0} \omega_1
	-
	\alpha\tau
	\int_\Sigma  f (\omega_1-\omega_0)
	+
	c \int_\Sigma \varphi i\partial\bar\partial \varphi
	 +
	 \int_\Sigma \varphi\left(-\Ric \omega_0+ 2\alpha\tau iF_{h_0} +c\omega_0\right)
	\\
	& =
	2\alpha
	\int_\Sigma
	f\left(
	iF_{h_0} + iF_{h_1}
	\right)
	+
	\alpha \int_\Sigma
	\left(
	\lvert\bm\phi\rvert_{h_1}^2 \omega_1
	-
	\lvert\bm\phi\rvert_{h_0}^2\omega_0
	\right)
	-2\alpha\tau
	\int_\Sigma f\omega_1
	+
	\left( c- \langle S\rangle \right) \int_\Sigma \varphi i\partial\bar\partial \varphi\\
	&\quad  + \frac{1}{2}
	\int_\Sigma
	\log \frac{\omega_1}{\omega_0} \omega_1
	+
	\frac{\langle S\rangle}{2} \int_\Sigma \varphi\left(\omega_1 - \omega_0\right)
	 +
	  \int_\Sigma \varphi\left(-\Ric \omega_0+ 2\alpha\tau iF_{h_0} +c\omega_0\right).
&\qedhere
\end{align*}
\end{proof}

Next we calculate the second variation of the functional $\widetilde{\mathcal{K}}_\alpha$ along smooth curves. For this, given a function $\dot \varphi \in C^\infty(\Sigma,\RR)$ and a K\"ahler form $\omega$, we
denote by $\eta_{\dot \varphi}$ the Hamiltonian vector field of $\dot \varphi$ with respect to $\omega$, that is,
$$
\eta_{\dot \varphi} = -i\nabla^{1,0}\dot \varphi+i\nabla^{0,1}\dot \varphi=- i \dot \varphi^k\frac{\partial}{\partial z^k} + i \dot \varphi^{\bar k}\frac{\partial}{\partial \bar z^k},$$
where $\omega=ig_{k\bar l}dz^k\wedge d\bar z^l$ and $\dot \varphi^k=g^{k\bar l} \dot \varphi_{\bar l}$ in holomorphic coordinates.

\begin{proposition}\label{prop:second}
In the notation in Proposition~\ref{prop:explicit}, for any smooth curve $(f_t, \varphi_t) \in \widetilde{\mathcal{P}}_{\omega_0}$, we have
\begin{align*}
\frac{\mathrm{d}^2 \mathcal{K}}{\mathrm{d}t^2}\left(f_t, \varphi_t\right) & = 2 \lVert \bar\partial \nabla^{1,0}\varphi'_t\rVert^2
- \int_\Sigma \left( \varphi_t'' - \left| \mathrm{d}\varphi_t'\right|_{\omega_t}^2\right)\left(S_{\omega_t}- \langle S\rangle \right)\omega_t,\\
\frac{\mathrm{d}^2 \mathcal{M}_\alpha}{\mathrm{d}t^2}(f_t, \varphi_t)
& =
4\alpha \lVert \mathrm{d}f'_t + \eta_{\varphi'_t}\lrcorner  iF_{h_t}
\rVert^2
+ 4\alpha \lVert J\eta_{\varphi'_t}\lrcorner \mathrm{d}_{A_t} \bm\phi - f'_t\bm\phi\rVert^2	\\
&\qquad +
4\alpha
\int_\Sigma \left(
f''_t - 2 J \eta_{\varphi_t'} \lrcorner \mathrm{d} f'_t - iF_{h_t}(\eta_{\varphi_t'}, J\eta_{\varphi_t'})
\right)
\wp_1(f_t, \varphi_t)\omega_t\\
&\qquad
-
\int_\Sigma \left( \varphi_t'' - \left| \mathrm{d}\varphi_t'\right|_{\omega_t}^2\right)
\left(
\alpha \Delta_{\omega_t}\left|\bm\phi\right|_{h_t}^2
- 2\alpha\tau \Lambda_{\omega_t}iF_{h_t}
+
(\langle S\rangle - c) \right)
\omega_t,
\end{align*}	
where $h_{f_t}$, $A_{h_{f_t}}$ and $\omega_{\varphi_t}$ are abbreviated as $h_t$, $A_t$ and $\omega_t$, respectively. Consequently, the second derivative of $\widetilde{\mathcal{K}}_\alpha$ along
smooth curves is given by
\begin{equation}
\begin{split}
\frac{\mathrm{d}^2
\widetilde{\mathcal{K}}_\alpha}{\mathrm{d}t^2}(f_t, \varphi_t)
& = \frac{\mathrm{d}^2 \mathcal{K}}{\mathrm{d}t^2}\left(f_t, \varphi_t\right) + \frac{\mathrm{d}^2 \mathcal{M}_\alpha}{\mathrm{d}t^2}\left(f_t, \varphi_t\right)\\	
& =
4\alpha \lVert \mathrm{d}f'_t + \eta_{\varphi'_t}\lrcorner  iF_{h_t}
\rVert^2
+ 4\alpha \lVert J\eta_{\varphi'_t}\lrcorner \mathrm{d}_{A_t} \bm\phi - f'_t\bm\phi\rVert^2
+
2 \lVert \bar\partial \nabla^{1,0}\varphi'_t\rVert^2 \\
&
\qquad
+
4\alpha
\int_\Sigma \Big(
f''_t - 2 J \eta_{\varphi'_t} \lrcorner \mathrm{d} f'_t - iF_{h_t}(\eta_{\varphi'_t}, J\eta_{\varphi'_t})
\Big)
\wp_1(f_t, \varphi_t)\omega_t\\
 & \qquad 	+
\int_\Sigma \Big(\varphi''_t- \lvert\mathrm{d}\varphi'_t\rvert_{\omega_t}^2 \Big)
\wp_2(f_t, \varphi_t)
\omega_t.
\end{split}
\end{equation}
\end{proposition}

\begin{proof}
The derivative of $\widetilde{\mathcal{K}}_\alpha$ along any smooth curve is given by
	\begin{equation}
	\frac{\mathrm{d}\widetilde{\mathcal{K}}_\alpha}{\mathrm{d}t}(f_t, \varphi_t)
	=
	4\alpha \int_\Sigma f'_t \wp_1(f_t, \varphi_t) \omega_t
	+
	\int_\Sigma \varphi_t' \wp_2(f_t, \varphi_t)\omega_t,
	\end{equation}
	and thus the second derivative is
	\begin{align*}
	\frac{\mathrm{d}^2\widetilde{\mathcal{K}}_\alpha}{\mathrm{d}t^2}&(f_t, \varphi_t)
	=
	4\alpha
	\int_\Sigma f''_t \wp_1(f_t, \varphi_t)\omega_t
	+ f'_t\delta\wp_1(f'_t, \varphi_t')\omega_t
	+ f'_t \wp_1(f_t, \varphi_t)\omega_t'\\
	&\quad +
	\int_\Sigma \varphi_t''\wp_2(f_t, \varphi_t)\omega_t
	+ \varphi_t' \delta\wp_2(f'_t, \varphi_t') \omega_t
	+ \varphi_t' \wp_2(f_t, \varphi_t)\omega_t'\\
	& =
	4\alpha \langle f'_t, \delta\wp_1(f'_t, \varphi_t')\rangle_{L^2(\omega_t)}
	+
	\langle \varphi_t', \delta\wp_2(f'_t, \varphi_t')\rangle_{L^2(\omega_t)}\\
	&\quad +
	4\alpha
	\int_\Sigma \left(
	f''_t + f'_t \left(\log\frac{\omega_t}{\omega_0}\right)' \right)
	\wp_1(f_t, \varphi_t)
	\omega_t
	+
	\int_\Sigma
	\left(
	\varphi_t'' + \varphi_t' \left( \log\frac{\omega_t}{\omega_0}\right)'\right)
	\wp_2(f_t, \varphi_t)
	\omega_t\\
	& =
	4\alpha \lVert \mathrm{d}f'_t + \eta_{\varphi'_t}\lrcorner \left( iF_{h_t} - \wp_1\omega_t
	\right)\rVert^2
	+ 4\alpha \lVert J\eta_{\varphi'_t}\lrcorner \mathrm{d}_{A_t} \bm\phi - f'_t\bm\phi\rVert^2
	+
	2 \lVert \bar\partial \nabla^{1,0}\varphi'_t\rVert^2 \\
	&
	\quad +   4\alpha \int_\Sigma\wp_1 \big(
	f'_t \Delta_\omega \varphi'_t
	+
	(\tau - \lvert\bm\phi\rvert_{h_t}^2) \lvert\nabla^{1,0}\varphi'_t\rvert^2
	\big)\omega_t
	+ 4\alpha
	\int \left(
	f''_t + f'_t \left(\log\frac{\omega}{\omega_0}\right)' \right)\wp_1 \omega_t\\
	& \quad
	+
	\int_\Sigma
	\wp_2 (\varphi'_t\Delta_{\omega_t} \varphi'_t - 2\lvert\nabla^{1,0}\varphi'_t\rvert^2) \omega_t
	+
	\int_\Sigma
	\left(
	\varphi_t'' + \varphi_t' \left( \log\frac{\omega_t}{\omega_0}\right)'\right)
	\wp_2(f_t, \varphi_t)
	\omega_t\\
	& =
	4\alpha \lVert \mathrm{d}f'_t + \eta_{\varphi'_t}\lrcorner \left( iF_{h_t} - \wp_1\omega_t
	\right)\rVert^2
	+ 4\alpha \lVert J\eta_{\varphi'_t}\lrcorner \mathrm{d}_{A_t} \bm\phi - f'_t\bm\phi\rVert^2
	+
	2 \lVert \bar\partial \nabla^{1,0}\varphi'_t\rVert^2 \\
	&
	\quad
	+
	4\alpha
	\int_\Sigma \left(
	f''_t +
	(\tau - \lvert\bm\phi\rvert_{h_t}^2) \lvert\nabla^{1,0}\varphi'_t\rvert^2  \right)\wp_1(f_t, \varphi_t) \omega_t\\
	& \quad
	+
	\int_\Sigma
	\left(
	\varphi_t''  - 2\lvert\nabla^{1,0}\varphi'\rvert^2
	\right)
	\wp_2(f_t, \varphi_t)
	\omega_t.
	\end{align*}
Completing the square in the above formula in a different way we obtain the desired result:	
	\begin{align*}
	\left|\mathrm{d}f'_t\right.
	+  \eta_{\varphi'_t}&\left.\lrcorner\left( iF_{h_t} - \wp_1\omega_t
	\right)\right|_{\omega_t}^2
	+
	\Big(
	f''_t +
	(\tau - \lvert\bm\phi\rvert_{h_t}^2) \lvert\nabla^{1,0}\varphi'_t\rvert^2  \Big)\wp_1(f_t, \varphi_t)\\
	 = &
	\left|
	\mathrm{d} f'_t + \eta_{\varphi'_t}\lrcorner iF_{h_t}
	\right|_{\omega_t}^2
	- 2\wp_1(f_t, \varphi_t)
	\langle \mathrm{d}f'_t + \eta_{\varphi'_t}\lrcorner iF_{h_t}, \mathrm{d}\varphi'_t\rangle_{\omega_t}
	+
	\wp_1(f_t, \varphi_t)^2
	\langle \mathrm{d}\varphi'_t, \mathrm{d}\varphi'_t\rangle_{\omega_t} \\
	& +
	\Big(
	f''_t +
	(\tau - \lvert\bm\phi\rvert_{h_t}^2) \lvert\nabla^{1,0}\varphi'_t\rvert^2  \Big)\wp_1(f_t, \varphi_t)\\
	 = &
	\Big(
	f''_t +
	\frac{1}{2} (\tau - \lvert\bm\phi\rvert_{h_t}^2) \lvert\mathrm{d}\varphi'_t\rvert^2
	+ \wp_1 \lvert\mathrm{d}\varphi'_t\rvert^2
	- 2\langle \mathrm{d}f'_t, \mathrm{d}\varphi'_t\rangle_\omega
	- 2
	\langle  \eta_{\varphi'_t}\lrcorner iF_{h_t}, \mathrm{d}\varphi'_t\rangle_{\omega_t}
	\Big)\wp_1(f_t, \varphi_t)\\
	& +
	\left|
	\mathrm{d} f'_t + \eta_{\varphi'_t}\lrcorner iF_{h_t}
	\right|_{\omega_t}^2  \\
	 = &
	\Big(
	f''_t
	- 2\langle \mathrm{d}f'_t, \mathrm{d}\varphi'_t\rangle_{\omega_t}
	-
	\langle  \eta_{\varphi'_t}\lrcorner iF_{h_t}, \mathrm{d}\varphi'_t\rangle_{\omega_t}
	\Big)\wp_1(f_t, \varphi_t)
	+
	\left|
	\mathrm{d} f'_t + \eta_{\varphi'_t}\lrcorner iF_{h_t}
	\right|_{\omega_t}^2.
        &\qedhere
	\end{align*}
\end{proof}

\begin{remark}
The previous result establishes the convexity of the $\alpha$-K-energy along \emph{smooth geodesics} on $\widetilde{\mathcal{P}}_{\omega_0}$, given by smooth solutions $(f_t, \varphi_t) \in
\widetilde{\mathcal{P}}_{\omega_0}$ of the following system of partial differential equations:
\begin{equation}\label{eqn:geodesic}
	\left\{
	\begin{array}{l}
	\varphi''_t- \lvert\mathrm{d}\varphi'_t\rvert_{\omega_{\varphi_t}}^2
    = 0,\\
	f''_t - 2 J \eta_{\varphi'_t} \lrcorner \mathrm{d} f'_t - iF_{h_{f_t}}(\eta_{\varphi'_t}, J\eta_{\varphi'_t})
	 = 0.	
	\end{array}
	\right.
\end{equation}
In particular, this gives an alternative and more direct proof of \cite[Proposition 3.10]{Al-Ga-Ga}, for the K\"ahler--Yang--Mills equations, in the special case of gravitating vortices (see \cite[Proposition
3.4]{Al-Ga-Ga2}). Note that \eqref{eqn:geodesic} and \cite[Equation (5.1)]{Al-Ga-Ga-P} differ by a sign in the third summand of the second equation. This is due to a different parametrization of the space of
Hermitian metrics on $L$.
\end{remark}

We are now ready to prove Proposition \ref{prop:convexalphaK}. By Theorem \ref{th:B-GP}, we can define a map
$$
\mathbf{h} \colon \mathcal{P}_{\omega_0} \lto \widetilde{\mathcal{P}}_{\omega_0},\quad \varphi \longmapsto (h_{\omega_\varphi},\varphi),
$$
where $\mathcal{P}_{\omega_0} = \left\{\varphi\in C^\infty(\Sigma, \mathbb{R}) \mid \omega_\varphi := \omega_0 + 2i\partial\bar\partial \varphi>0 \right\}$ denotes the space of K\"ahler potentials and
$h_{\omega_\varphi}$ is the unique Hermitian metric $h_\omega$ on $L$ solving the vortex equation~\eqref{eqn:V0}.

\begin{proof}[Proof of Proposition \ref{prop:convexalphaK}]
The linearization of the vortex equation~\eqref{eqn:V0} with respect to the Hermitian metric $h$ is given by (see \eqref{eq:formulaeYao})
$$
\delta\wp_1|_{(f, \varphi)}(\dot f) = (\Delta_\omega +  \lvert\bm\phi\rvert_h^2)\dot f,
$$
which is invertible, since $\bm\phi \neq 0$ by assumption and since $\Delta_\omega$ has positive spectrum. Therefore, applying the implicit function theorem on a Hilbert completion of $C^\infty(\Sigma,\RR)$, we
obtain that the map $\mathbf{h}$ is smooth as a map of Fr\'echet spaces. By definition of $\sigma_\alpha$ in \eqref{eq:sigmaGVreduced}, it is then easy to conclude that
$$
\mathbf{h}^* \widetilde{\sigma}_\alpha = \sigma_\alpha.
$$
Hence, $\sigma_\alpha$ is closed and by contractibility of $\mathcal{R}$ it induces the reduced $\alpha$-K-energy
$$
\mathcal{K}_\alpha = \mathbf{h}^* \widetilde{\mathcal{K}}_\alpha,
$$
given explicitly by (see Proposition~\ref{prop:explicit})
\begin{equation*}
\begin{split}
\mathcal{K}_\alpha(\varphi) & = \frac{1}{2}\int_\Sigma \log \frac{\omega_\varphi}{\omega_0} \omega_\varphi +
	\frac{\langle S\rangle}{2}\int_\Sigma \varphi \left(\omega_\varphi - \omega_0\right) \\
	&\qquad + 2\alpha \int_\Sigma f_\varphi \left(iF_{h_0} + iF_{h_\varphi}\right)
	+
	\alpha \int_\Sigma \left(\lvert\bm\phi\rvert_{h_\varphi}^2 \omega_\varphi
	-
	\lvert\bm\phi\rvert_{h_0}^2\omega_0\right)
	-
	2\alpha\tau \int_\Sigma f_\varphi\omega_\varphi\\
	&\qquad
	+
	\left( c- \langle S\rangle \right) \int_\Sigma \varphi i\partial\bar\partial \varphi+
	\int_\Sigma \varphi\left(-\Ric \omega_0+ 2\alpha\tau iF_{h_0} +c\omega_0\right),
\end{split}
\end{equation*}
where $h_\varphi = e^{2f_\varphi}h_0$ is the unique solution of the vortex equation \eqref{eqn:GV0red} associated to $\omega_\varphi$ via Theorem \ref{th:B-GP}. By the explicit formula for $\sigma_\alpha$ in
\eqref{eq:sigmaGVreduced}, critical points of $\mathcal{K}_\alpha$ are given by solutions of the reduced gravitating vortex equation \eqref{eqn:GV0red} on $\mathcal{P}_{\omega_0}$. Furthermore, by Proposition
\ref{prop:second}, the second derivative of the reduced K-energy along a smooth path $\varphi_t$ in $\mathcal{P}_{\omega_0}$ is given by
\begin{equation*}
	\begin{split}
	\frac{\mathrm{d}^2
		\mathcal{K}_\alpha}{\mathrm{d}t^2}(\varphi_t)
		& =
	4\alpha \lVert \mathrm{d}f'_t + \eta_{\varphi'_t}\lrcorner  iF_{h_t}
        \rVert^2
	+ 4\alpha \lVert J\eta_{\varphi'_t}\lrcorner \mathrm{d}_{A_t} \bm\phi - f'_t\bm\phi\rVert^2
	+
	2 \lVert \bar\partial \nabla^{1,0}\varphi'_t\rVert^2 \\
 & \qquad 	+
	\int_\Sigma \Big(\varphi''_t- \lvert\mathrm{d}\varphi'_t\rvert_{\omega_t}^2 \Big)
	\wp_2(f_t, \varphi_t)
	\omega_t,
\end{split}
\end{equation*}
for $\omega_t = \omega_0 + \mathrm{d}\mathrm{d}^c \varphi_t$ and $h_t = e^{2f_t}h_0 = h_{\varphi_t}$, which proves the statement.
\end{proof}

To finish this section, we review in our set-up the well-known construction of smooth geodesics in the space of K\"ahler potentials from 1-parameter subgroups of holomorphic automorphisms. This example will be
applied in Section~\ref{section:Polystability}.

\begin{example}\label{ex:autgeod}
Given a K\"ahler metric $\omega$ on the complex projective line $\Sigma=\PP^1$, a holomorphic vector field on $\PP^1$ whose imaginary part is
Hamiltonian with respect to $\omega$ determines a (smooth) geodesic ray emanating from $\omega$ in the space of K\"ahler potentials.
More precisely, let $\bm v$ be a holomorphic $(1,0)$-vector field on $\PP^1$ such that $\operatorname{Im}\bm v\iprod\omega=\mathrm{d}\Phi$, for a real-valued smooth function $\Phi$.
Let $\{ \sigma_t\}_{t\in\RR}$ be the family of holomorphic automorphisms of $\PP^1$ generated by the real vector field $\operatorname{Re}\bm v=J\operatorname{Im}\bm v$, where $J$ is the almost complex structure of
$\PP^1$.
Then the K\"ahler form
$$
\omega_t=\omega+2i\partial\bar\partial\varphi_t:=\sigma_t^*\omega
$$
satisfies $2i\partial\bar\partial \dot\varphi_t= 2i\partial\bar\partial \sigma_t^*\Phi$.
Normalizing $\varphi_t$ and $\Phi$ by the condition
\[
\int_{\PP^1}\dot{\varphi}_t\omega_t=\int_{\PP^1}\Phi\omega=0,
\]
we have $\dot \varphi_t=\Phi_t:=\sigma_t^*\Phi$.
Using now the equation
$$
\iota_{\text{Re}\bm v}\sigma_t^*\omega = \sigma_t^*\left( \iota_{\text{Re}\bm v} \omega\right)= \sigma_t^*\left(-J\mathrm{d}\Phi\right)  = - J\mathrm{d}\Phi_t
$$
and $\iota_{\text{Im}\bm v} \sigma_t^*\omega = \mathrm{d}\Phi_t$, we get that $\iota_{\bm v} \omega_t = 2i \bar\partial \Phi_t$. This implies $\bm v= 2\uparrow_{\omega_t} \bar\partial \Phi_t$, i.e., $\bm v^i =2
g_t^{i\bar j}\frac{\partial \Phi_t}{\partial \bar z_j}= 2\dot \varphi_t^i$. As a consequence,
$$
\ddot \varphi_t = \mathcal{L}_{\text{Re}\bm v}\Phi_t= \frac{1}{2} \left( \bm v + \overline{\bm v}\right)\Phi_t= \bm v\left(\Phi_t\right)
$$
and, therefore,
\begin{equation}\label{eq:1PS.geodesic-ray.1}
\ddot \varphi_t-2\left|\nabla^{1,0}_{\omega_t}\dot\varphi_t\right|_{\omega_t}^2=0,
\end{equation}
i.e., $\{\varphi_t\}_{t\in \mathbb{R}}$ is a smooth geodesic line in $\mathcal{P}_{\omega_0}$.

The above construction of smooth geodesics will be applied in Theorem~\ref{thm:existence-implies-stability}.
\end{example}

\begin{remark}
If instead of making the normalization $\int_{\PP^1}\dot\varphi_t \omega_t^n = 0$ in Example \ref{ex:autgeod} we take $\dot\varphi_t = \Phi_t + b$ for a fixed constant $b\in \mathbb{R}$, then the identity $\ddot
\varphi_t = \bm v(\Phi_t)$ still holds. Therefore we also get a smooth geodesic line in $\mathcal{P}_{\omega_0}$. This gives a family of geodesic lines $\{\varphi_t + bt\}_{t\in \mathbb{R}}$, labelled by a real
number $b$, all lifting the same geodesic line in the space of K\"ahler metrics to the space of K\"ahler potentials (see Remark \ref{rem:liftRtoRom}).
\end{remark}

\section{Extension of the reduced $\alpha$-K-energy}
\label{sec:extensionKen}

\subsection{The metric completion of the space of K\"ahler potentials}
\label{sub:metric-completion-Kahler-potentials.1}

Over the next three sections, we will construct an extension of the reduced $\alpha$-K-energy, as defined in Proposition \ref{prop:convexalphaK}, to the completion of the space of smooth K\"ahler potentials
$\mathcal{P}_{\omega_0}$ introduced by Darvas~\cite{Darvas0} (following Chen's work~\cite{Chen}). This extension will be used in Section~\ref{section:Properness}, where we will apply the general principles developed
by Darvas--Rubinstein~\cite{Darvas} to prove the properness of the reduced $\alpha$-K-energy in the space of K\"ahler metrics.

We start by reviewing briefly the Finsler-metric completion $(\mathcal{E}_1,d_1)$ of the space of smooth K\"ahler potentials $\mathcal{P}_{\omega_0}$ considered in~\cite{Darvas0}, in the case of our interest. For
this, we fix a compact Riemann surface $\Sigma$ and a smooth background K\"ahler metric $\omega_0$ with total volume
$$
\int_{\Sigma} \omega_0 = V.
$$
We consider the 
functional
defined as (see e.g.~\cite{Darvas17,Darvassurvey,Darvas})
\begin{equation}\label{eq:Aubin-Mabuchi.1}
\AMY(u)\defeq\frac{1}{2V}\int_\Sigma u(\omega_0+\omega_u).
\end{equation}
%
%
Following~\cite{Darvas0} (see also~\cite[Definition 4.2]{Darvas}), we define the \emph{path-length distance} $d_1$ on $\mathcal{P}_{\omega_0}$ by
\[
d_1(\psi_0,\psi_1)\defeq\inf_\psi\int_0^1\mathrm{d}t\int_\Sigma\lvert\psi'_t\rvert\omega_{\psi_t},
\]
for all $\psi_0,\psi_1\in\mathcal{P}_{\omega_0}$, where $\psi$ ranges over all smooth paths $\psi\colon[0,1]\to \mathcal{P}_{\omega_0},\, t\mapsto\psi_t$, from $\psi_0$ to $\psi_1$, that is, smooth maps
$\psi=\psi(t,z)\in C^\infty([0,1]\times\Sigma)$ such that $\psi(0,\cdot)=\psi_0$ and $\psi(1,\cdot)=\psi_1$. We will also consider the subspace $\mathcal{R}\subset\mathcal{P}_{\omega_0}$ given by the smooth K\"ahler
potentials $\psi$ that are normalized so that
\[
\AMY(\psi)=0.
\]

By application of~\cite[Theorem 3.5]{Darvas0}, $d_1$ is a metric on $\mathcal{P}_{\omega_0}$ and we also consider the induced metric in $\mathcal{R}$. In the next result, which follows by direct application of the
results in \cite{Darvas0,GZ}, we describe the metric completion of  $(\mathcal{P}_{\omega_0},d_1)$ and  $(\mathcal{R},d_1)$ in the case of our interest. Recall that the space of $\omega_0$-plurisubharmonic functions
on $\Sigma$ is defined by
$$
\textnormal{PSH}(\Sigma, \omega_0) = \{\psi \equiv - \infty \textrm{ or } \psi \in C^{\uparrow} \textrm{ and } \omega_\psi = \omega_0 + 2i \partial \dbar \psi \geq 0 \textrm{ in the sense of currents} \},
$$
where $C^{\uparrow}$ is the set of upper semi-continuous integrable functions on $(\Sigma,\omega_0)$. Note the standard fact that if $\psi \in \textnormal{PSH}(\Sigma, \omega_0)$, then $\omega_{\psi}$ is a measure.

\begin{theorem}[\cite{Darvas0}]\label{thm:Potentialcompletion}
The metric completion of $(\mathcal{P}_{\omega_0},d_1)$ is given by the set
\begin{equation}\label{ineq:E1def}
\mathcal{E}_1 = \Bigg{\{} \psi \in \textnormal{PSH}(\Sigma, \omega_0)\;\Biggm\vert\;\int_\Sigma\lvert\psi\rvert\omega_{\psi}<\infty \Bigg{\}}.
\end{equation}
The metric on $\mathcal{E}_1$ is defined by
$$
d_1(\psi_0,\psi_1) = \lim_{k \to \infty}d_1(\psi_0^k,\psi_1^k),
$$
for any $\psi_0,\psi_1 \in \mathcal{E}_1$, where $\{\psi_j^k\} \subset \mathcal{P}_{\omega_0}$ is any sequence decreasing pointwise to $\psi_j$. Moreover, the completion $\overline{\mathcal{R}}\subset\mathcal{E}_1$
of $\mathcal{R}$ is
\begin{equation}\label{eq:normalisationofphi}
\overline{\mathcal{R}} = \mathcal{E}_1\cap \AMY^{-1}(0),
\end{equation}
where $\AMY$ is given by the same formula as in \eqref{eq:Aubin-Mabuchi.1}.
\end{theorem}

\begin{proof}
By~\cite[Theorem 4.17]{Darvas0}, the metric completion of $(\mathcal{P}_{\omega_0},d_1)$ is the set of $\omega_0$-plurisubharmonic functions on $\Sigma$ that satisfy (see \cite{GZ})
\begin{equation}\label{eq:Edef}
\displaystyle \int_{\Sigma}\omega_{\psi} =V,
\end{equation}
jointly with the inequality in \eqref{ineq:E1def}. Using the fact that $\dim_\mathbb{C}\Sigma=1$, by~\cite[Lemma~1.2]{GZ} we have that \eqref{eq:Edef} is superfluous, and hence the first part of the statement
follows. The
well-definition of the distance in $\mathcal{E}_1$ follows from \cite[Section 4]{Darvas0}. The last part of the statement, about the completion $\overline{\mathcal{R}}$, is a consequence of~\cite[Lemma 5.2]{Darvas}.
\end{proof}

In our next result we observe that in complex dimension $1$, $d_1$ is closely related to the familiar convergence in $W^{1,2}$ for PSH functions, where $W^{1,2}=W^{1,2}(\Sigma,\omega_0)$ is the Sobolev space of
$L^2$ functions on $(\Sigma,\omega_0)$ whose first distributional derivatives are in $L^2$. The following lemma is probably well known to experts, but we have not been able to find a precise reference in the
literature and so we include a proof for completeness. We will use the following notation for the average with respect to $\omega_0$ of an integrable function $u$ on $\Sigma$:
$$
\fint  u = \frac{1}{V}\int_\Sigma u \omega_0.
$$

\begin{lemma}\label{lem:metric-completion-Kahler-potentials.1}
On a compact Riemann surface $\Sigma$, we have:
\begin{enumerate}\label{lem:equivalent-norm}
\item[\textup{(1)}]
$\mathcal{E}_1= W^{1,2}\cap \textnormal{PSH}(\Sigma, \omega_0)$ as sets.
\item[\textup{(2)}]
For any sequence $\{u_k\}\subset \overline{\mathcal{R}}$ and any $u\in \overline{\mathcal{R}}$, we have $\displaystyle{\lim_{k\to \infty}d_1(u_k, u)=0}$ if and only if
    $\displaystyle{\lim_{k\to \infty}\lVert u_k-u\rVert_{W^{1,2}}}=0$.
\end{enumerate}
\end{lemma}

\begin{proof}
To prove (1), let $u\in \mathcal{E}_1$. By the Poincar\'e inequality, it suffices to show that $\displaystyle \int_{\Sigma} \sqrt{-1}\partial u\wedge\bar{\partial} u <\infty$ for concluding that $u\in W^{1,2}$. To
this end we note that
$$
\displaystyle\Big\lvert\int_{\Sigma} u \sqrt{-1}\partial \bar{\partial} u\Big\rvert < \infty.
$$
Now there exists a sequence of smooth $\omega_0$-PSH functions $u_i$ decreasing to $u$ (see~\cite[Theorem 1]{BK}). From~\cite[Lemma 3.47]{Darvassurvey}, we see that
$$
\displaystyle -\int_{\Sigma} \sqrt{-1}\partial u_i\wedge\bar{\partial}u_i=\int_{\Sigma} u_i \sqrt{-1}\partial \bar{\partial} u_i \rightarrow \int_{\Sigma} u \sqrt{-1}\partial \bar{\partial} u.
$$
Thus a subsequence of $u_i-\fint u_i$ converges weakly in $W^{1,2}$, strongly in $L^2$, and almost everywhere to a $W^{1,2}$ function $v$. Thus $u$ is in $W^{1,2}$.

Now, let $u\in W^{1,2} \cap\textnormal{PSH}(\Sigma,\omega_0)$. Without loss of generality (by subtracting the supremum), assume that $u\leq 0$. Let $\{\rho_j\}$ be a partition of unity subordinate to coordinate
balls. Then
$$
\displaystyle \int_{\Sigma} -u\omega_u = -\sum_j \int_{\Sigma} \rho_j u \omega_u.
$$
Let $v_{j,i}$ be a sequence obtained by convolving $u_{0,j}+u$ where $\sqrt{-1}\partial \bar{\partial} u_{0,j}=\omega_0$. Note that $v_{j,i}$ decreases to $u_{0,j}+u$ and converges to it in $W_{loc}^{1,2}$, and so
does $\rho_j v_{j,i}$. Now
\begin{equation}\label{eq:integrationbypartsE1}
\begin{split}
\displaystyle \sum_j \int_{\Sigma} \rho_j (v_{j,i}-u_{0,j}) \omega_u &= \sum_j \int_{\Sigma}  \left(\rho_j (v_{j,i}-u_{0,j}) \omega_0+ \sqrt{-1} \partial \bar{\partial}(\rho_j (v_{j,i}-u_{0,j}))  u \right)
\\&
= \sum_j \int_{\Sigma}  \left(\rho_j (v_{j,i}-u_{0,j}) \omega_0- \sqrt{-1} \partial (\rho_j (v_{j,i}-u_{0,j}))  \bar{\partial}u \right),
\end{split}\end{equation}
which converges to $\int_{\Sigma} \left(u \omega_0 -\sqrt{-1}\partial u \wedge\bar{\partial} u\right)>-\infty$. Hence using the monotone convergence theorem, we see that $\displaystyle \int_{\Sigma} \vert u \vert
\omega_u<\infty$ and $u\in \mathcal{E}_1$.

Finally, to prove (2), by \cite[Lemma 5.1]{Darvas} we see that $d_1(u_k,u)\rightarrow 0$ iff $\AMY(u_k)=0\rightarrow \AMY(u)=0$ and $u_k\rightarrow u$ in $L^1$. Thus if $u_k\rightarrow u$ in $W^{1,2}$,
$d(u_k,u)\rightarrow 0$. If $d_1(u_k,u)\rightarrow 0$, then $u_k\rightarrow u$ in $L^1$. By the Poincar\'e inequality, the proof will follow if we prove that $\Vert \partial u_k - \partial u\Vert_{L^2}\rightarrow
0$. To this end, we claim that the following integration-by-parts formula holds for all $u,v\in \mathcal{E}_1$:
\begin{equation*}
\displaystyle \int_{\Sigma} u\partial \bar{\partial}v = -\int_{\Sigma} \partial u \wedge\bar{\partial} v.
\end{equation*}
Indeed, proceeding as in the proof of~\ref{eq:integrationbypartsE1}, we see (using the Monotone Convergence Theorem) that
\begin{equation*}\begin{split}
\displaystyle \int_{\Sigma} u \sqrt{-1}\partial \bar{\partial} v&=-\sum_j \lim_{i\rightarrow \infty}\int_{\Sigma} \sqrt{-1} \partial (\rho_j (v_{j,i}-u_{0,j}))\bar{\partial} v \\&
= -\int_{\Sigma} \sqrt{-1}\partial u\wedge \bar{\partial} v.
\end{split}\end{equation*}
Since $\AMY(u_k)=\AMY(u)=0$, we see that $\fint u_k = \frac{1}{2V} \displaystyle \int_{\Sigma} \sqrt{-1}\partial u_k\wedge \bar{\partial} u_k$ and likewise for $u$, and hence
$$
\Vert \partial u_k \Vert_{L^2}\rightarrow \Vert \partial u \Vert_{L^2}.
$$
Moreover, every $L^2$ convergent subsequence of $u_k$ converges to $u$, because $u_k$ converges to $u$ in $L^1$. Lastly, since $\Vert \partial u_k \Vert_{L^2}<C$, every subsequence has a further subsequence that
weakly converges in $W^{1,2}$ and strongly converges in $L^2$ (and hence) to $u$. Thus $\partial u_k$ weakly converges to $\partial u$ in $L^2$. Thus $\Vert \partial u_k - \partial u \Vert_{L^2}\rightarrow 0$.
\end{proof}

\subsection{Weak solutions to the vortex equation}

The next step to extend the reduced $\alpha$-K-energy (see Proposition \ref{prop:convexalphaK}) to the completion $\mathcal{E}_1$ of the space of smooth K\"ahler potentials $\mathcal{P}_{\omega_0}$ (see Lemma
\ref{lem:metric-completion-Kahler-potentials.1}), is to study the vortex equation for background metrics that satisfy low regularity conditions. More precisely, we are interested in (weak) K\"ahler metrics of the
form $\omega_\psi$ for $\psi \in \mathcal{E}_1$.

We start this section by giving a precise definition of the weak solutions of the vortex equation that we will use. As in the previous section, we fix a smooth K\"ahler metric $\omega_0$ on $\Sigma$ with total
volume $V$. Let $L$ be a holomorphic line bundle over $\Sigma$ with degree $N = \int_{\Sigma}c_1(L)>0$ and non-vanishing holomorphic section
$$
0 \neq \bm \phi \in H^0(\Sigma,L).
$$
We also fix a symmetry breaking parameter $\tau > 0$, satisfying (see~\eqref{eqn:ineqtauV})
\begin{equation*}
V>\frac{4\pi N}{\tau}.
\end{equation*}
Furthermore, we fix a background smooth Hermitian metric $h_0$ on $L$.

\begin{definition}\label{def:weak-vortex.1}
Consider the semi-positive current $\omega_{\psi}=\omega_0+2i\partial\bar\partial\psi$, where $\psi\in\mathcal{E}_1$.
A \emph{weak solution} to the vortex equation with background current $\omega_\psi$ is a function $f\in W^{1,2}(\Sigma, \omega_0)$ such that
\begin{equation}\label{def:weak-vortex-inequality.1}
\left|\bm\phi\right|_{h_0}^2e^{2f}\in W^{1,2}(\Sigma, \omega_0), \qquad \tau-\left|\bm\phi\right|_{h_0}^2e^{2f}\geq 0,
\end{equation}
where the inequality is satisfied almost everywhere in $\Sigma$, and the following identity is satisfied for all $v\in C^\infty(\Sigma, \mathbb{R})$:
\begin{equation}
\label{eqn:weak-vortex}
\int_\Sigma v\cdot
iF_{h_0}
+2i\partial v\wedge \bar\partial f
=\int_\Sigma
\frac{1}{2} v \left(\tau-\left|\bm \phi\right|_{h_0}^2e^{2f}\right) \omega_0
-
i \partial \left[ v \left(\tau-\left|\bm \phi\right|_{h_0}^2e^{2f}\right) \right]\wedge  \bar\partial \psi.
\end{equation}
\end{definition}

\begin{remark}\label{rem:weak-vortex.1}
By~\cite[Proposition 5.2]{Brad}, the inequality~\eqref{def:weak-vortex-inequality.1} follows for any smooth $f\in C^\infty(\Sigma,\mathbb{R})$ such that $h_0e^{2f}$ is a solution of the vortex equation with
background metric $\omega_\psi$, for $\psi \in \mathcal{P}_{\omega_0}$. Nonetheless, this inequality may fail if $\psi$ or $f$ have low regularity properties, and hence it is necessary to include it in
Definition~\ref{def:weak-vortex.1}.
\end{remark}

\begin{remark}\label{rem:weak-vortex.2}
Note that the condition that~\eqref{eqn:weak-vortex} is satisfied for all test functions $v\in C^\infty(\Sigma, \mathbb{R})$ is not \emph{a priori} equivalent to the interpretation of the vortex equation
\begin{equation}\label{eqn:current-vortex}
iF_{h_0}-2i\partial\bar\partial f=\frac{1}{2}\left(\tau-\left|\bm \phi\right|_{h_0}^2e^{2f}\right)\omega_\psi
\end{equation}
as an equation for currents. This is due to the fact that~\eqref{eqn:weak-vortex} is obtained multiplying~\eqref{eqn:current-vortex} by $v$ and formally integrating by parts \emph{once}, 
$-2i\int_\Sigma v\partial\bar\partial f$ to $2i\int_\Sigma\partial v\wedge\bar\partial f$,
while an interpretation of~\eqref{eqn:current-vortex} as an equation for currents would lead to formally integrating by parts twice.
$-2i\int_\Sigma(\bar\partial\partial v)f$ (something similar happens with the last term of this equation).\\
\end{remark}

In the next result we establish the existence and uniqueness of weak solutions of the vortex equation, in the sense of Definition \ref{def:weak-vortex.1}. For the proof, we recall that the the Green function
$G_{\omega_0}(P,Q)$ on $\Sigma$ with respect to the metric $\omega_0$ satisfies, by definition,
\begin{equation}\label{eq:Greenformula}
2i\partial\bar\partial G_{\omega_0} (P, Q) = \frac{1}{V}\omega_0(Q) - [P]
\end{equation}
in the sense of currents (our notation for $G_{\omega_0}$ agrees with \cite[Page 108]{Aubin} and \cite{Al-Ga-Ga-P}, but it differs by a negative sign from, e.g.,  \cite{FPY}).

\begin{proposition}\label{prop:W-1-2-estimate}
Given $\psi\in\mathcal{E}_1$, there exists a unique $f\in W^{1,2}$ satisfying the following property: for any sequence $\left\{\psi_n\right\}\subset \mathcal{P}_{\omega_0}$, with $\psi_n\to \psi$ in $W^{1,2}$ as
$n\to \infty$, consider the sequence of K\"ahler metrics $\omega_{\psi_n}$ on $\Sigma$. Then the sequence of solutions $h_n=h_0 e^{2f_n}$ of the vortex equation \eqref{eqn:V0} in Theorem \ref{th:B-GP},
corresponding to the background metrics $\omega_{\psi_n}$, satisfy $f_n\to f$ in $W^{1,2}$ as $n\to \infty$. Furthermore, the unique $f\in W^{1,2}$ satisfying this condition is a weak solution to the vortex equation
with background metric $\omega_\psi$, in the sense of Definition \ref{def:weak-vortex.1}.
\end{proposition}

\begin{proof}
We start obtaining uniform $W^{1,2}$ estimates for the sequences $f_n$ and $\left|\bm \phi\right|_{h_0}^2 e^{2f_n}$.
By definition, $f_n\in C^\infty(\Sigma,\mathbb{R})$ satisfies the vortex equation with $C^\infty$ background metric $\omega_{\psi_n}$, that is,
\begin{equation}
\label{eqn:vortex}
iF_{h_0}-2i\partial\bar\partial f_n=\frac{1}{2}\left(\tau-\left|\bm\phi\right|_{h_0}^2e^{2f_n}\right) \omega_{\psi_n},
\end{equation}
so~\cite[Proposition 5.2]{Brad} gives the pointwise inequality
\begin{equation}\label{eq:vortex-pointwise-inequality.W-1-2-estimate.1}
\left|\bm\phi\right|_{h_0}^2e^{2f_n}\leq\tau.
\end{equation}
Green's formula \eqref{eq:Greenformula}, together with~\eqref{eqn:vortex},~\eqref{eq:vortex-pointwise-inequality.W-1-2-estimate.1}
implies
\begin{equation}
\widetilde{f}_n(P)=\int_\Sigma G_{\omega_0}(P,Q)\Delta_{\omega_0} f_n(Q)\omega_0(Q)
\geq
-\int_\Sigma G_{\omega_0}(P,Q)iF_{h_0}(Q)
\geq-C,
\end{equation}
for all $P\in\Sigma$, where $\widetilde{f}_n=f_n-\fint_\Sigma f_n\omega_0$ and $C$ does not depend on $n$. Multiplying~\eqref{eqn:vortex} by $\widetilde f_n+C\geq 0$ and integrating yields (using the fact that
$iF_{h_0}\geq -C\omega_0$)
\begin{equation*}
-CV+ 2\int_\Sigma i \partial \widetilde f_n \wedge \bar\partial \widetilde f_n
\leq
\frac{\tau C}{2}V
-
\tau\int_\Sigma i\partial\widetilde{f}_n\wedge\bar{\partial}\psi_n.
\end{equation*}
It follows that
\[
\int_\Sigma \left|\nabla \widetilde f_n\right|_{\omega_0}^2 \omega_0
\leq
C+C\int_\Sigma\left|\nabla \psi_n \right|_{\omega_0}^2\omega_0,
\]
where $C$ only depends on $N, V, \tau$. Together with the Poincar\'e inequality for $\widetilde{f}_n$, this gives a uniform $W^{1,2}$ bound on the sequence $\widetilde{f}_n$.
We will obtain now a uniform bound for $\fint_\Sigma f_{n}\omega_0$.
Firstly,~\eqref{eq:vortex-pointwise-inequality.W-1-2-estimate.1} gives the \emph{a priori} estimate
$$
2f_n\leq -\log \left|\bm\phi\right|_{h_0}^2 + \log \tau,
$$
and therefore provides a uniform upper bound for $\fint_\Sigma f_n \omega_0$ (recall that $\log\left|\bm\phi\right|_{h_0}^2$ is integrable, as shown by integration in a local polar coordinates, since $\bm\phi$ is
holomorphic).
Suppose now that there exists a subsequence $\{f_{n_k}\}$ such that
$$
\fint_\Sigma f_{n_k}\omega_0 \to -\infty.
$$
Since $\widetilde{f}_{n_k}$ is uniformly bounded in $W^{1,2}$, there exists a further subsequence, still denoted $\{\widetilde{f}_{n_k}\}$, such that $\widetilde f_{n_k}\to \widetilde f$ in $L^2$ and $\widetilde
f_{n_k}\to \widetilde f$ almost everywhere, for some $\widetilde f\in W^{1,2}$. However,
$$
f_{n_k}=\widetilde{f}_{n_k}+\fint_\Sigma f_{n_k} \omega_0 \rightarrow -\infty
$$
almost everywhere. It immediately gives that $\left| \bm \phi\right|_{h_0}^2 e^{2f_{n_k}}\to 0$ almost everywhere. But the Bochner formula~\eqref{eq:Bochner.1} implies
\[
\int_\Sigma \left| \bm \phi\right|_{h_0}^2 e^{2f_n} \cdot i\partial\bar\partial \left( \left| \bm \phi\right|_{h_0}^2 e^{2f_n}\right)
\geq
-\int_\Sigma \left(\left|\bm\phi\right|_{h_0}^2 e^{2f_n}\right)^2 iF_{h_0e^{2f_n}},
\]
that is,
\begin{equation}
\label{eqn:nablaphi-L2-bound}
\frac{1}{2}\int_\Sigma \left| \nabla \left( \left| \bm \phi\right|_{h_0}^2 e^{2f_n} \right)\right|_{\omega_0}^2 \omega_0
\leq
\int_\Sigma \(\left|\bm\phi\right|_{h_0}^2 e^{2f_n}\)^2 iF_{h_0e^{2f_n}}
\leq
2\pi\tau^2 N.
\end{equation}
As a consequence, we can assume $ \left| \bm \phi\right|_{h_0}^2 e^{2f_{n_k}} $ converges to $0$ weakly in $W^{1,2}$ as $k\to +\infty$.  But this leads to a contradiction, because integrating~\eqref{eqn:vortex}, we
obtain
\begin{align*}
\tau \left( V- \frac{4\pi N}{\tau}\right)
&=
\int_\Sigma  \left| \bm \phi\right|_{h_0}^2 e^{2f_{n_k}}\! \left(\omega_0+ 2i\partial\bar\partial \psi_{n_k}\right)
\\&=
\int_\Sigma  \left| \bm \phi\right|_{h_0}^2 e^{2f_{n_k}} \omega_0
- 2
\int_\Sigma  i\partial \left(  \left| \bm \phi\right|_{h_0}^2 e^{2f_{n_k}} \right)\wedge \bar\partial \psi_{n_k},
\end{align*}
which converges to $0$, since the first term converges to $0$ by the Dominated Convergence Theorem and the second term converges to $0$ by the weak convergence of $ \left|\bm \phi\right|_{h_0}^2 e^{2f_{n_k}}$ and
the strong convergence of $\psi_{n_k}$ in $W^{1,2}$. Therefore $\fint_\Sigma f_n \omega_0$ is uniformly bounded from below and hence $f_n$ is uniformly bounded in $W^{1,2}$.

The above $W^{1,2}$ estimates for $f_n$ and $\left|\bm \phi\right|_{h_0}^2 e^{2f_n}$ imply that there exists a subsequence $f_{n_k}$ such that $f_{n_k}\to f$ weakly in $W^{1,2}$ and strongly in $L^2$, while
$\left|\bm \phi\right|_{h_0}^2 e^{2f_{n_k}}\to \left|\bm \phi\right|_{h_0}^2 e^{2f}$ weakly in $W^{1,2}$ and strongly in $L^2$. Since  $f_{n_k}\to f$ strongly in $L^2$, there exists a further subsequence, still
denoted by $f_{n_k}$, that converges to $f$ almost everywhere, so the inequalities~\eqref{eq:vortex-pointwise-inequality.W-1-2-estimate.1} imply~\eqref{def:weak-vortex-inequality.1} almost everywhere. Moreover,
multiplying the vortex equation~\eqref{eqn:vortex} for the pair $f_{n_k}, \psi_{n_k}$ by any $v\in C^\infty(\Sigma)$, integrating, and taking the limit $k\to\infty$, we see that $f$ satisfies the testing equation
\eqref{eqn:weak-vortex}. Therefore $f$ is a \emph{weak} solution to the vortex equation with background metric $\omega_{\psi}$.

Taking the difference of equation \eqref{eqn:vortex} for $f_n$ and $f_m$, and integrating against $f_n-f_m$, we obtain
\begin{align*}
\int_\Sigma \left|\nabla (f_n-f_m)\right|_{\omega_0}^2\omega_0
& = \tau\int_\Sigma\left(f_n - f_m\right) i\partial\bar\partial \left(\psi_n - \psi_m\right)\\
&\qquad  -
\int_0^1 \mathrm{d}t \int_\Sigma \left(f_n-f_m\right)^2 \left|\bm\phi\right|_{h_0}^2 e^{2\left[(1-t)f_m+tf_n\right]}\left[ (1-t)\omega_{\psi_m} + t\omega_{\psi_n} \right]\\
&\qquad - \int_0^1 \mathrm{d}t \int_\Sigma \left|\bm\phi\right|_{h_0}^2 e^{2\left[(1-t)f_m+tf_n\right]} (f_n-f_m)i\partial\bar\partial \left(\psi_n-\psi_m\right)\\
& \leq
\frac{1}{2}\int_\Sigma \left|\nabla (f_n-f_m)\right|_{\omega_0}^2\omega_0
+
\frac{\tau^2}{8}\int_\Sigma \left| \nabla (\psi_n-\psi_m)\right|_{\omega_0}^2\omega_0 \\
&\qquad  -
\frac{1}{2}\int_\Sigma \left(
\left|\bm\phi\right|_{h_0}^2 e^{2f_n} - \left|\bm\phi\right|_{h_0}^2 e^{2f_m}\right) i\partial\bar\partial \left(\psi_n-\psi_m\right)\\
& \leq
\frac{1}{2}\int_\Sigma \left|\nabla (f_n-f_m)\right|_{\omega_0}^2\omega_0
+
\frac{\tau^2}{8}\int_\Sigma \left| \nabla (\psi_n-\psi_m)\right|_{\omega_0}^2\omega_0 \\
& \qquad +
\frac{1}{2}\tau\sqrt{2\pi N}\left(\int_\Sigma\left|\nabla\left(\psi_n-\psi_m\right)\right|_{\omega_0}^2\omega_0\right)^\frac{1}{2}.
\end{align*}
Notice that in the last inequality we used the Cauchy--Schwarz inequality and \eqref{eqn:nablaphi-L2-bound}. Therefore,
\begin{equation}
\label{est:gradient-L2}
\int_\Sigma \left|\nabla \left(f_n - f_m\right)\right|_{\omega_0}^2\omega_0
\leq
\frac{\tau^2}{4}\int_\Sigma  \left| \nabla (\psi_n-\psi_m)\right|_{\omega_0}^2\omega_0
 +
\tau\sqrt{2\pi N} \left(  \int_\Sigma \left|\nabla \left(\psi_n-\psi_m\right)\right|_{\omega_0}^2 \omega_0\right)^\frac{1}{2}.
\end{equation}
By the above argument, any subsequence $f_{n_k}$ of $f_n$ has a further subsequence $f_{n_{k_l}}$ that converges strongly in $W^{1,2}$ to some $f\in W^{1,2}$, which is a weak solution to the vortex equation with
background metric $\omega_\psi$. The estimate~\eqref{est:gradient-L2} now implies that any two such subsequential limits $f$ and $f'$, which are in $W^{1,2}$, must differ by a constant.
But \eqref{eqn:weak-vortex} applied to the test function $v=1$ shows that any weak solution $f$ to the vortex equation with background metric $\omega_{\psi}$ satisfies the identity
\[
\int_\Sigma \left|\bm\phi\right|_{h_0}^2 e^{2f} \omega_0 -
2  i\partial \left( \left|\bm\phi\right|_{h_0}^2 e^{2f}\right)\wedge\bar\partial \psi
=
\tau \left( V - \frac{4\pi N}{\tau}\right)>0.
\]
Therefore, if two weak solutions $f$ and $f'$ with the same background metric differ by a constant $c$, i.e., $f'=f+c$, then it must hold that $c=0$. In conclusion, there exists a unique $f\in W^{1,2}$, which is a
weak solution to the vortex equation with background metric $\omega_\psi$, such that any subsequence $f_{n_k}$ of $f_n$ has a further subsequence that converges strongly in $W^{1,2}$ to $f$. Since $\{f_n\}$ is
uniformly bounded in $W^{1,2}$, this implies $f_n\to f$ in $W^{1,2}$ too. As this holds for any subsequence $f_{n_k}$, the same $f\in W^{1,2}$ works for any $\{\psi_n\} \subset \mathcal{P}_{\omega_0}$ with
$\psi_n\to \psi$ in $W^{1,2}$, as required.
\end{proof}

\subsection{The extended reduced $\alpha$-K-energy}
\label{sub:ext-grav-vortex-K-energy}

We address next the promised extension of the reduced $\alpha$-K-energy in Proposition \ref{prop:convexalphaK}, to the completion $\mathcal{E}_1$ of the space of smooth K\"ahler potentials $\mathcal{P}_{\omega_0}$.
We will see that, when further restricted to the space of normalized K\"ahler potentials $\overline{\mathcal{R}}$, the reduced $\alpha$-K-energy inherits good semicontinuity properties.
We keep the notation from the previous two sections. Following Proposition~\ref{prop:explicit}, the reduced $\alpha$-K-energy $\mathcal{K}_\alpha\colon \mathcal{P}_{\omega_0} \to\mathbb{R}$ is given by
\begin{equation}\label{def:K-energy.before-extension}
\mathcal{K}_\alpha(\varphi)=\widetilde{\mathcal{K}}_\alpha(f_\varphi,\varphi)=\mathcal{K}(\varphi)+\mathcal{M}_\alpha(f_\varphi,\varphi).
\end{equation}
Following Berman--Darvas--Lu~\cite{BDL1}, we extend the first term in~\eqref{def:K-energy.before-extension} to a functional (the extended K-energy)
$$
\mathcal{K}\colon\mathcal{E}_1\lto\mathbb{R}\cup\{+\infty\},
$$
using Chen--Tian's explicit formula in terms of the K\"ahler potential, that is,
\begin{equation}\label{def:extendedK-energy}
\mathcal{K}(\psi)
=\frac{V}{2}\text{Ent}(\omega_0,\omega_\psi)+\frac{\langle S\rangle}{2}\int_\Sigma\psi\left(\omega_\psi-\omega_0\right),
\end{equation}
for all $\psi\in\mathcal{E}_1$, where $\text{Ent}(\omega_0, \omega_\psi)$ is the entropy of the measure $\omega_\psi$ with respect to the measure $\omega_0$ defined via a Legendre transformation.

To construct an appropriate extension of the second summand $\mathcal{M}_\alpha\colon \mathcal{P}_{\omega_0} \to\mathbb{R}$ in~\eqref{def:K-energy.before-extension}, given $\psi\in \mathcal{E}_1$
we consider the weak solution $f_\psi\in W^{1,2}$ of the vortex equation with background metric $\omega_\psi$ given by Proposition~\ref{prop:W-1-2-estimate}. Using the explicit formula for $\mathcal{M}_\alpha$ in
Proposition~\ref{prop:explicit}, we define
$$
\mathcal{M}_\alpha\colon \mathcal{E}_1 \lto\mathbb{R}
$$
by the formula
\begin{equation}\label{eqn:M-extension}
\begin{split}
\mathcal{M}_\alpha(\psi)
& =
2\alpha \int_\Sigma \left( f_\psi \cdot 2 iF_{h_0} +  2 i \partial f_\psi\wedge \bar \partial f_\psi\right)\\
&\quad - 2\alpha\tau \int_\Sigma \left( f_\psi\omega_0 - 2i \partial f_\psi\wedge \bar\partial \psi \right)
-
\left( c - \langle S\rangle\right) \int_\Sigma i\partial \psi \wedge\bar\partial \psi\\
&\quad + \alpha\tau \left( V- \frac{4\pi N}{\tau} \right)
-
\alpha \int_\Sigma \left|\bm\phi\right|_{h_0}^2 \omega_0
+
\int_\Sigma \psi\left(-\Ric \omega_0+ 2\alpha\tau iF_{h_0} +c\omega_0\right).
\end{split}
\end{equation}
Notice that \eqref{eqn:M-extension} is a well-defined extension of $\mathcal{M}_\alpha \colon \mathcal{P}_{\omega_0} \to \mathbb{R}$, by Proposition~\ref{prop:explicit}, combined with
\eqref{ineq:E1def} and $f_\psi\in W^{1,2}$.
While the extended K-energy~\eqref{def:extendedK-energy} is only known to be lower semicontinuous (see \cite[Theorem 4.7]{BDL1}), the next result shows that \eqref{eqn:M-extension} is continuous with respect to the
metric $d_1$ when restricted to the completion $\overline{\mathcal{R}} \subset \mathcal{E}_1$ of the space of normalized K\"ahler potentials (see Theorem \ref{thm:Potentialcompletion}).

\begin{lemma}\label{lem:continuityofMalphaoncompletion}
The functional $\mathcal{M}_\alpha \colon \overline{\mathcal{R}} \to \mathbb{R}$ defined by \eqref{eqn:M-extension} is continuous with respect to the Finsler metric $d_1$.
\end{lemma}

\begin{proof}
Let $\psi\in\overline{\mathcal{R}}$ and $\left\{\psi_n\right\}$ be a sequence in $\overline{\mathcal{R}}$ satisfying $d_1(\psi_n, \psi)\to 0$ as $n\to \infty$. If all the functions $\psi_n$ are smooth
$\omega_0$-PSH and satisfy $d_1(\psi_n, \psi)\to 0$, then the $W^{1,2}$ convergence $f_{\psi_n}\to f_\psi$ (given by Lemma~\ref{lem:metric-completion-Kahler-potentials.1} and Proposition~\ref{prop:W-1-2-estimate})
and the explicit formula \eqref{eqn:M-extension} imply $\mathcal{M}_\alpha(\psi_n)\to \mathcal{M}_\alpha(\psi)$ as $n\to \infty$. More generally, if we only assume that $\left\{
\psi_n\right\}\subset\overline{\mathcal{R}}$, for each fixed $n$ we can find a sequence $\left\{\psi_{n,j}\right\}\subset\mathcal{R}$ such that $d_1\left( \psi_{n,j}, \psi_n\right)\to 0$ as $j\to \infty$. By the
already established convergence, applied to the smooth $\omega_0$-PSH functions $\psi_{n,j}$, we can pick a subsequence $\psi_{n,j_n}$ such that for all $n$,
\[
\lVert  \psi_{n,j_n} - \psi_n\rVert_{W^{1,2}}\leq \frac{1}{n}, \quad
\left| \mathcal{M}_\alpha(\psi_{n,j_n}) - \mathcal{M}_\alpha(\psi_n) \right|\leq \frac{1}{n}.
\]
The sequence $\psi_{n,j_n}$ of smooth $\omega_0$-PSH functions converges to $\psi$ in $W^{1,2}$ as $n\to \infty$ and so $\left| \mathcal{M}_\alpha(\psi_{n,j_n})- \mathcal{M}_\alpha(\psi)\right|\to 0$ as $n\to
\infty$. Therefore we conclude that $\lim_{n\to \infty}\mathcal{M}_\alpha(\psi_n) = \mathcal{M}_\alpha(\psi)$.
\end{proof}

Summing up the previous discussion, we consider now the extension
\begin{equation}\label{def:extended-alpha-K-energy.0}
\mathcal{K}_\alpha\colon\overline{\mathcal{R}}\to\mathbb{R}\cup\{+\infty\}
\end{equation}
of the functional~\eqref{def:K-energy.before-extension} given by
\begin{equation}\label{def:extended-alpha-K-energy}
\mathcal{K}_\alpha(\psi)=\mathcal{K}(\psi)+\mathcal{M}_\alpha(\psi),
\end{equation}
for all $\psi\in \overline{\mathcal{R}}$, with $\mathcal{K}$ and $\mathcal{M}_\alpha$ defined in \eqref{def:extendedK-energy} and \eqref{eqn:M-extension}, respectively.

\begin{theorem}\label{thm:extended-vortex-K-energy}
The functional $\mathcal{K}_\alpha\colon\overline{\mathcal{R}} \to\mathbb{R}\cup\{+\infty\}$ is the largest $d_1$-lower semicontinuous extension of the reduced $\alpha$-K-energy $\mathcal{K}_\alpha$ from
$\mathcal{R}$ to $\overline{\mathcal{R}}$.
\end{theorem}

\begin{proof}
This follows from the decomposition~\eqref{def:extended-alpha-K-energy}, Lemma~\ref{lem:continuityofMalphaoncompletion} and the fact that the functional $\mathcal{K}\colon\mathcal{E}_1\to\mathbb{R}\cup\{+\infty\}$
is the largest
$d_1$-lower semicontinuous  extension of the K-energy $\mathcal{K}$ from $\mathcal{P}_{\omega_0}$ to $\mathcal{E}_1$, by~\cite[Theorem 4.7]{BDL1}.
\end{proof}

\section{Convexity of the reduced $\alpha$-K-energy}
\label{section:Laplacianestimate}

\subsection{Finite-energy geodesics}\label{sec:finenergy}

The present section is devoted to study the convexity properties of the (extended) reduced $\alpha$-K-energy defined by  Theorem \ref{thm:extended-vortex-K-energy}. For simplicity, we will abuse notation and refer
to this functional simply as the $\alpha$-K-energy. We build on~\cite[Theorem~4.7]{BDL1}, which proves that the extended K-energy in~\eqref{def:extendedK-energy} 
Theorem~\ref{thm:uniqueness})
has good convexity properties in the completion $\mathcal{E}_1$ of the space of K\"ahler potentials described in Theorem \ref{thm:Potentialcompletion}.

To start, we summarize Darvas' construction of special geodesics on the completion $\mathcal{E}_1$, for the case of Riemann surfaces (see Theorem \ref{thm:finiteengeod}).
Even though $(\mathcal{E}_1,d_1)$ is a geodesic metric space by \cite[Theorem 4.17]{Darvas0}, that is, any two points in $\in\mathcal{E}_1$ can be connected by a geodesic segment, geodesics connecting different
points may not be unique (see~\cite[p. 208]{Darvas0}).
Nonetheless, one can obtain a distinguished geodesic joining any two given points in $(\mathcal{E}_1,d_1)$, via a sequence of weak geodesic segments, in the sense of Chen \cite{Chen}.
This type of geodesics are known as \emph{finite-energy geodesics}, and play a key role in the present paper.
To recall their construction note that, as discovered by Semmes and Donaldson~\cite{Don,Semmes}, the geodesic equation on the space of K\"ahler potentials \eqref{eq:geodesiceq} can be written as a Homogeneous
Complex Monge--Amp\`ere equation. Let $\varphi\colon[0,1]\to \mathcal{P}_{\omega_0},\, t\mapsto\varphi_t=\varphi(t,\cdot)$ be a smooth map,
$$
S = \{s \; | \; 0 < \textrm{ Re } s < 1 \} \subset \mathbb{C},
$$
and $X = S \times \Sigma$. Let $\varphi^c \in C^\infty(X)$ be the complexification of $\varphi$, defined by $\varphi^c(s,z) = \varphi(\textrm{Re } s,z)$. Then $\varphi_t$ is a smooth geodesic if and only if
$\varphi^c$ satisfies the equation
$$
\Omega_{\varphi^c}^2 = 0,
$$
where $\Omega_{\varphi^c} : = \tilde \omega_0 + \mathrm{d}\mathrm{d}^c \varphi^c$ and $\tilde \omega_0$ is the pullback of $\omega_0$ to $X$. Given two K\"ahler potentials, there may be no smooth
geodesic connecting them but, as shown by Chen \cite{Chen}, there exist certain weak geodesics connecting points in $\mathcal{P}_{\omega_0}$.

\begin{theorem}[\cite{Chen}]\label{thm:Chen.1}
Given $\varphi_0,\varphi_1 \in \mathcal{P}_{\omega_0}$, there exists a unique weak solution of the Dirichlet problem
\begin{itemize}

\item $(\tilde \omega_0 + \mathrm{d}\mathrm{d}^c \varphi^c)^2 = 0$,

\item $\varphi^c(s + ir,z) = \varphi^c(s,z)$ for all $z \in \Sigma$ and $r \in \mathbb{R}$,

\item $\lim_{s \to j} \varphi^c(s,\cdot) = \varphi_j$ uniformly on $\Sigma$, for $j = 0,1$,

\end{itemize}
such that $\varphi^c$ has bounded Laplacian on $X = S \times \Sigma$. Furthermore, $\varphi_t = \varphi^c(t,\cdot)$, for $t \in [0,1]$, gives a well-defined curve
$$
[0,1] \to \{\psi \in \textnormal{PSH}(\Sigma,\omega_0) \; | \; \Delta_{\omega_0} \psi \in L^\infty(\Sigma) \}
$$
which we will refer to as Chen's $C^{1,\bar{1}}$ geodesic.
\end{theorem}

Actually, Chen \cite{Chen} proved that the curve $t \mapsto \varphi_t$ can be approximated by smooth solutions of the Dirichlet problem for \emph{$\epsilon$-geodesics}, given by
\begin{itemize}

\item $\left( \varphi_\epsilon'' - \left| \mathrm{d}\varphi_\epsilon'\right|_{\omega_{\varphi_\epsilon}}^2\right) \omega_{\varphi_\epsilon}
= \epsilon\cdot \omega_0$,

\item $\lim_{t \to j} \varphi_\epsilon(t,\cdot) = \psi_j$ uniformly on $\Sigma$, for $j = 0,1$.

\end{itemize}
We are ready to state Darvas' construction of \emph{finite-energy geodesics} on $(\mathcal{E}_1,d_1)$ (see \cite[Theorem 4.17]{Darvas0}), in the case of our interest. Given $\psi_0,\psi_1 \in \mathcal{E}_1$, we
consider, as in \cite[Section 4]{Darvas0}, sequences $\{\psi_j^k\} \subset \mathcal{P}_{\omega_0}$ decreasing pointwise to $\psi_j$. Let $t \mapsto \varphi^k_t$ be Chen's (unique) $C^{1,\bar{1}}$ geodesic joining
$\psi_0^k$ with $\psi_1^k$. Define $t \mapsto \varphi_t$ as the decreasing pointwise limit
\begin{equation}\label{eq:finiteengeod}
\varphi_t = \lim_{k \to \infty}\varphi^k_t, \qquad t \in [0,1].
\end{equation}

\begin{theorem}[\cite{Darvas0}]\label{thm:finiteengeod}
The space $(\mathcal{E}_1,d_1)$ is a geodesic metric space. Furthermore, given arbitrary $\psi_0,\psi_1 \in \mathcal{E}_1$, the curve \eqref{eq:finiteengeod} takes values in $\mathcal{E}_1$, is independent of
choices, and defines a geodesic connecting $\psi_0$ with $\psi_1$.
\end{theorem}

The main goal of this section is to prove the convexity of the reduced $\alpha$-K-energy in Theorem \ref{thm:extended-vortex-K-energy} along the finite-energy geodesics constructed in Theorem \ref{thm:finiteengeod}
(Theorem \ref{thm:convexfinen}).

\subsection{A priori $W^{2,p}$ estimates along approximate geodesics}\label{sect:W2p}

In order to achieve our goal, we need to prove refined estimates for the curve of weak solutions of the vortex equation along a finite-energy geodesic, obtained in Proposition \ref{prop:W-1-2-estimate}.
For this, we will derive a priori estimates along a family of $\epsilon$-geodesics, as considered in Section \ref{sec:finenergy}.

We fix a pair of smooth K\"ahler potentials $\varphi_0, \varphi_1 \in \mathcal{P}_{\omega_0}$. Let $\{\varphi_\epsilon=\varphi_\epsilon(t)\}_{t\in [0,1]}$ be Chen's $\epsilon$-geodesic connecting $\varphi_0$ and
$\varphi_1$, i.e., solving
\begin{equation}
\label{eqn:Chen's-geodesic}
\left( \varphi_\epsilon'' - \left| \mathrm{d}\varphi_\epsilon'\right|_{\omega_{\varphi_\epsilon}}^2\right) \omega_{\varphi_\epsilon}
=
\epsilon\cdot \omega_0,
\end{equation}
and such that $\lim_{t \to j} \varphi_\epsilon(t,\cdot) = \varphi_j$ uniformly on $\Sigma$, for $j = 0,1$. Recall that $\varphi_{\epsilon}(t)$ has a uniform $C^{1,\bar 1}$ bound, that is,
$\Delta_{\omega_0}\varphi_{\epsilon}(t)$ is bounded, for a uniform constant independent of $\epsilon$. In particular, there exists a uniform constant $C >0$  such that
$$
0\leq \omega_{\varphi_\epsilon(t)}\leq C\omega_{\omega_0}.
$$
In the sequel, all (positive) uniform constants are denoted by $C$. Assume inequality \eqref{eqn:ineqtauV} and, for each $t \in [0,1]$, let $f_\epsilon(t) \in C^\infty(\Sigma,\RR)$ be the unique solution of the
vortex equation
\[
i F_{h_{\epsilon}(t)} + \frac{1}{2}\left( \left|\bm\phi\right|_{h_{\epsilon}(t)}^2 - \tau\right) \omega_\epsilon(t) =0
\]
given by Theorem \ref{th:B-GP}, where $h_{\epsilon}(t)=h_0 e^{2f_{\epsilon}(t)}$  and $\omega_\epsilon(t) = \omega_{\varphi_\epsilon(t)}$. Note that $f_{\epsilon}(t)$ satisfies
\begin{equation}
i\Lambda_{\omega_0} F_{h_0}+\Delta_{\omega_0} f_{\epsilon}(t)
=\frac{1}{2} \(\tau-\vert \bm\phi \vert^2_{h_{\epsilon}(t)}\)
\frac{\omega_{\epsilon}(t)}{\omega_0}.
\label{eq:forfepsilon(t)}
\end{equation}
for each $t\in [0,1]$.  The parameter $t$ will be suppressed if no confusion arises.

We start by proving a uniform a priori $W^{2,p}$-estimate for $f_\epsilon$. The absolute bound $\vert \bm\phi \vert^2_{h_{\epsilon}}\leq \tau$ (see ~\cite[Proposition 5.2]{Brad}) implies now that $2f_\epsilon \leq
-\log \lvert\bm\phi\rvert_{h_0}^2 + \log \tau$ and therefore
\[
\fint f_{\epsilon} \omega_0\leq C,
\]
and the uniform $C^{1,\bar 1}$ bound of $\omega_\epsilon(t)$ yields
\begin{equation}
-C\leq \Delta_{\omega_0} f_{\epsilon} \leq C.
\end{equation}
Using Green's representation formula,
\begin{equation}
f_{\epsilon}(P)=\fint f_{\epsilon} \omega_0+\int G_{\omega_0} (P,Q) \Delta_{\omega_0} f_{\epsilon}(Q) \omega_0(Q) \leq C,
\label{eq:greenformulaforf}
\end{equation}
where $0\leq G_{\omega_0}(P,Q)\leq C(1+\vert\ln(d_0 (P,Q)) \vert)$, we obtain that
\[
\Vert \tilde{f}_{\epsilon} \Vert_{C^0}
\leq
\sup_{P\in \Sigma}
\int_{\Sigma} G_{\omega_0}(P,Q) \left(C+C\right)\omega_0(Q)
\leq C,
\]
where $\tilde{f}_{\epsilon}=f_{\epsilon}-\fint f_{\epsilon} \omega_0$.  If $\fint f_{\epsilon_n}\omega_0$ tends to $-\infty$ for a sequence $\epsilon_n$ converging to $0$, then $f_{\epsilon_n}$ goes to $-\infty$
uniformly on $\Sigma$. Therefore, the inequality
\[
0< \tau V - 4\pi N
=
\int \vert \bm\phi \vert_{h_{\epsilon_n}}^2 \omega_{\epsilon_n}
\leq
C \int \lvert\bm\phi\rvert_{h_0}^2 e^{2f_{\epsilon_n}(t)}\omega_0
\]
(obtained by integrating the vortex equation) gives a contradiction as the RHS goes to $0$.  Together with the above uniform upper bound of $\fint f_\epsilon \omega_0$, we obtain the uniform $C^0$ bound of
$f_\epsilon$ and the standard $W^{2,p}$ estimate for the Equation \eqref{eq:forfepsilon(t)} implies
\begin{equation}
\label{est:W2pestimate}
\Vert f_{\epsilon} \Vert_{W^{2,p}}\leq C_p
\end{equation}
for any $p\geq 1$ and a constant $C_p>0$ independent of $\epsilon$ and $t$. The consequence of this bound is that a subsequence of $ f_{\epsilon_n}$ converges (using a diagonal argument) to $f$ in weak sense which
is in $W^{2,p}$ for all $p\geq 1$, and the Sobolev inequality actually implies the convergence happens in $C^{1,\alpha}$ for any $\alpha\in (0,1)$. Moreover, $\Delta_{\omega_0} f$ is an $L^{\infty}$ current.

\subsection{$C^1$ estimate of $f_\epsilon'$}\label{sect:C1estimate}

We need to get a similar regularity on $f$ in terms of the complex time parameter as well. Going back to Equation \eqref{eq:forfepsilon(t)} and differentiating with respect to the real time $t$ (in any case,
$\omega_{\epsilon}$ and hence $f_{\epsilon}$ is $S^1$-invariant), we get the following equation:
\begin{equation}
\Delta_{\omega_0} f'_{\epsilon}
=-\vert \bm\phi \vert_{h_{\epsilon}}^2 f_{\epsilon} '\frac{\omega_{\epsilon}}{\omega_0}-\frac{1}{2} (\tau-\vert \bm\phi \vert^2_{h_{\epsilon}})\Delta_{\omega_0} \varphi'_{\epsilon}.
\label{eq:firsttemporalderivative}
\end{equation}
To simplify the notation, the norms and differential operators associated to the metric $\omega_0$ will be labelled with a $0$-subscript. Multiplying by $f_{\epsilon}'$ and integrating by parts, we obtain
\begin{align}\nonumber
\int \left| \nabla_0 f'_{\epsilon} \right|^2_0 &\omega_0
+ \int \vert \bm\phi \vert_{h_{\epsilon}}^2 (f_{\epsilon} ')^2\omega_{\epsilon}
=-\frac{1}{2} \int f'_{\epsilon} (\tau-\vert \bm\phi \vert^2_{h_{\epsilon}})\Delta_0 \varphi'_{\epsilon} \omega_0
\\\nonumber
& = -\frac{1}{2} \int \nabla_0 \Big( f'_{\epsilon} (\tau-\vert \bm\phi \vert^2_{h_{\epsilon}})\Big)\cdot (\nabla_0 \varphi'_{\epsilon}) \omega_0
\\\label{ineq:integrationf'epsilonbeforegeodesic}
&\leq  \frac{1}{4} \int \vert \nabla_0 f'_{\epsilon} \vert^2 \omega_0
+ \frac{\tau^2}{4} \int \left| \nabla_0 \varphi_\epsilon'\right|_0^2 \omega_0
+ \frac{1}{4} \int \left| f'_{\epsilon}\right|\cdot \left|  \nabla_0 \vert \bm\phi \vert^2_{h_{\epsilon}} \right|_0 \cdot \vert \nabla_0 \varphi_{\epsilon}'\vert_0 \omega_0.
\end{align}
In particular, for the last term we can use the Cauchy--Schwarz inequality
\begin{equation}
\label{ineq:cs}
\begin{split}
\int \left| f'_{\epsilon}\right|\cdot \left|  \nabla_0 \vert \bm\phi \vert^2_{h_{\epsilon}} \right|_0 \cdot \vert \nabla_0 \varphi_{\epsilon}'\vert_0
\omega_0
\leq
\sqrt{ \int \left(f_\epsilon'\right)^2 \left|\bm\phi\right|_{h_\epsilon}^2 \cdot \left|\nabla_0\varphi_\epsilon'\right|_0^2 \omega_0}
\sqrt{
	\int \frac{\left| \nabla_0\left|\bm\phi\right|_{h_\epsilon}^2\right|_0^2}{\left|\bm\phi\right|_{h_\epsilon}^2}\omega_0.
}
\end{split}
\end{equation}
At this point, the $\epsilon$-geodesic equation $\eqref{eqn:Chen's-geodesic}$ and the uniform $C^{1,\bar 1}$ estimates for $\varphi_\epsilon$ imply
	\[
	\vert \nabla_0 \varphi'_{\epsilon} \vert_0 ^2 \omega_0 =
		\vert \nabla_{\omega_\epsilon} \varphi'_{\epsilon} \vert_{\omega_\epsilon} ^2 \omega_\epsilon
		=
		\left(\varphi_\epsilon'' - \epsilon \Lambda_{\omega_\epsilon}\omega_0\right)\omega_\epsilon
		 \leq C\omega_{\epsilon}.
	\]
Using this inequality, it follows that the right-hand side of \eqref{ineq:cs} is controlled by
\[
\frac{1}{2}\int \vert \bm\phi \vert_{h_{\epsilon}}^2 (f_{\epsilon} ')^2\omega_{\epsilon}
+
C 	\int \frac{\left| \nabla_0\left|\bm\phi\right|_{h_\epsilon}^2\right|_0^2}{\left|\bm\phi\right|_{h_\epsilon}^2}\omega_0,
\]
for some $C>0$. We thus arrive at the inequalities
\begin{equation}
\label{ineq:firstestimateforf'}
\begin{split}
\frac{3}{4}\int \vert \nabla_0 f'_{\epsilon} \vert^2 \omega_0 + \frac{1}{2}\int \vert \bm\phi \vert_{h_{\epsilon}}^2 (f_{\epsilon} ')^2\omega_{\epsilon}
&   \leq
\frac{\tau^2}{4}\int \left| \nabla_0 \varphi'_{\epsilon} \right|_0 ^2 \omega_0
+
C \int \frac{\left| \nabla_0 \vert \bm\phi \vert_{h_{\epsilon}}^2 \right|_0^2}{\vert \bm\phi \vert_{h_{\epsilon}}^2}\omega_0 \leq C,
\end{split}
\end{equation}
by the $W^{2,p}$ estimate of $f_\epsilon$ and the $W^{1,2}$ estimate of $\varphi_\epsilon'$. Now, the Poincar\'e inequality implies
\begin{equation}
\left\lVert f_{\epsilon}' -\fint f'_{\epsilon} \omega_0 \right\rVert_{L^2} \leq C.
\label{ineq:Lpestimateforaveragefreef'}
\end{equation}
Using \eqref{ineq:firstestimateforf'} and \eqref{ineq:Lpestimateforaveragefreef'}, we see that
\begin{equation}
\begin{split}
\displaystyle C
& \geq  \int\vert \bm\phi \vert_{h_{\epsilon}}^2 (f_{\epsilon} ')^2\omega_{\epsilon}
= \int \vert \bm\phi \vert_{h_{\epsilon}}^2 \left(f_{\epsilon} '-\fint f'_{\epsilon}\omega_0 + \fint f'_{\epsilon}\omega_0\right)^2\omega_{\epsilon}   \\
& \geq \int \vert \bm\phi \vert_{h_{\epsilon}}^2 \left[ \left( \fint f'_{\epsilon}\omega_0 \right)^2
+2 \left(  f_{\epsilon} '-\fint f'_{\epsilon}\omega_0\right) \fint f'_{\epsilon}\omega_0   \right] \omega_{\epsilon}  \\
& \geq \frac{1}{2} \left( \fint f'_{\epsilon} \omega_0 \right)^2  \int \vert \bm\phi \vert_{h_{\epsilon}}^2 \omega_{\epsilon}
-2 \int \vert \bm\phi \vert_{h_{\epsilon}}^2 \left(  f_{\epsilon} '-\fint f'_{\epsilon} \omega_0 \right)^2\omega_{\epsilon} \\
& \geq \frac{1}{2} \left( \fint f'_{\epsilon} \omega_0\right)^2 \int \vert \bm\phi \vert_{h_{\epsilon}}^2  \omega_{\epsilon}
-2\tau C\int  \left(  f_{\epsilon} '-\fint f'_{\epsilon} \omega_0 \right)^2  \omega_0\\
& \geq \frac{1}{2} \left( \fint f'_{\epsilon} \omega_0\right)^2 \int \vert \bm\phi \vert_{h_{\epsilon}}^2  \omega_{\epsilon}
-C,
\end{split}
\end{equation}
where for the inequality in the second line we simply use $2ab\geq -\frac{1}{2}a^2 -2b^2$. This, together with the fact that $\displaystyle \int \vert \bm\phi \vert_{h_{\epsilon}}^2\omega_{\epsilon}=\tau V-4\pi
N>0$, gives
\begin{equation}
\begin{split}
\left\lvert\fint f'_{\epsilon}\omega_0\right\rvert\leq C.
\end{split}
\label{ineq:estimateforintegralf'}
\end{equation}
Going back to \eqref{eq:firsttemporalderivative}, using Green's representation formula we arrive at the following representation of $f_\epsilon'$:
\begin{equation}
\begin{split}
f'_{\epsilon}(P)
& = \fint f'_{\epsilon} \omega_0+\int G_0 (P,Q) \Delta_0 f'_{\epsilon}(Q) \omega_0(Q)  \\
& = \fint f'_{\epsilon} \omega_0-\int G_0 (P,Q) \vert \bm\phi \vert_{h_{\epsilon}}^2 (Q) f_{\epsilon} '(Q)\omega_{\epsilon}(Q)
\\&\quad
-\frac{1}{2} \int G_0 (P,Q)  (\tau-\vert \bm\phi \vert^2_{h_{\epsilon}})(Q)\Delta_0
\varphi'_{\epsilon}(Q) \omega_0(Q).
\end{split}
\end{equation}
This further transforms to
\begin{align}
\nonumber
f'_{\epsilon}(P) & - \fint f'_{\epsilon} \omega_0
\\\nonumber
= &   -\int G_0 (P,Q) \vert \bm\phi \vert_{h_{\epsilon}}^2 (Q) f_{\epsilon} '(Q)\omega_{\epsilon}(Q)
-\frac{1}{2} \int G_0 (P,Q) \Bigg(\Delta_0 \Big( (\tau-\vert \bm\phi \vert^2_{h_{\epsilon}}) \varphi'_{\epsilon} \Big)
\\\nonumber
&\qquad
- \varphi_{\epsilon}'\Delta_0  (\tau-\vert \bm\phi \vert^2_{h_{\epsilon}})
+ 2\nabla_0  (\tau-\vert \bm\phi \vert^2_{h_{\epsilon}}) \cdot \nabla_0 \varphi'_{\epsilon}\Bigg) (Q) \omega_0(Q)
\\\nonumber
= &  - \int G_0 (P,Q) \vert \bm\phi \vert_{h_{\epsilon}}^2 (Q) f_{\epsilon} '(Q)\omega_{\epsilon}(Q)
+ \frac{1}{2}\left(\fint  (\tau-\vert \bm\phi \vert^2_{h_{\epsilon}}) \varphi'_{\epsilon} \omega_0 -  (\tau-\vert \bm\phi \vert^2_{h_{\epsilon}})(P) \varphi'_{\epsilon} (P)\right)
\\\label{eq:Greenforf'}
& +\frac{1}{2}\int G_0 (P,Q) \Bigg(\varphi_{\epsilon}'\Delta_0  (\tau-\vert \bm\phi \vert^2_{h_{\epsilon}})
- 2\nabla_0  (\tau-\vert \bm\phi \vert^2_{h_{\epsilon}}) \cdot \nabla_0 \varphi'_{\epsilon}\Bigg) (Q) \omega_0(Q).
\end{align}
Write $\Phi_0=\left|\bm\phi\right|_{h_0}^2$ and $\Phi_\epsilon = \left|\bm\phi\right|_{h_\epsilon}^2=\Phi_0 e^{2f_\epsilon}$. For each term we can estimate as
\begin{equation}
\label{eqn:various-estimates}
\begin{split}
& \left|
\int G_0(P, Q)   \left|\bm\phi\right|_{h_\epsilon}^2(Q) f_\epsilon'(Q) \omega_\epsilon(Q)
\right|^2
\leq
\int G_0(P,Q)^2 \left|\bm\phi\right|_{h_\epsilon}^2(Q) \omega_\epsilon(Q) \cdot\!
\int \left|\bm\phi\right|_{h_\epsilon}^2 \left(f_\epsilon'\right)^2 \omega_\epsilon\!\\
& \qquad \leq
\tau C\int G_0(P,Q)^2\omega_0(Q) \cdot C\\
& \qquad \leq C, \\
&\left| \varphi_\epsilon'\right|
\leq C, \;\; \left| \nabla_0 \varphi_\epsilon'\right|_0\leq C, \\
& \left| \nabla_0 \Phi_\epsilon\right|_0
= \left| \nabla_0 \Phi_0 e^{2f_\epsilon} + 2 \Phi_\epsilon \nabla_0 f_\epsilon\right|
\leq C, \\
&  \left|
\Delta_0 \Phi_\epsilon\right|
=
\left|
\Delta_0 \Phi_0 e^{2f_\epsilon}
- 4 e^{2f_\epsilon}\nabla_0\Phi_0\cdot \nabla_0 f_\epsilon
+
2\Phi_\epsilon \Delta_0 f_\epsilon
-
4\Phi_\epsilon \left|\nabla_0 f_\epsilon\right|_0^2
\right|
\leq
C.
\end{split}
\end{equation}
Using \eqref{eq:Greenforf'} and \eqref{ineq:estimateforintegralf'}, we see that
\begin{equation}
\left|  f'_{\epsilon} \right| \leq C.
\label{ineq:C0estimateforf'}
\end{equation}
Differentiating \eqref{eq:Greenforf'} with respect to the $P$ variable:
\begin{equation}\begin{split}
\nabla_0 f_\epsilon'(P)
&= - \int \nabla_0 G_0 (P,Q) \vert \bm\phi \vert_{h_{\epsilon}}^2 (Q) f_{\epsilon} '(Q)\omega_{\epsilon}(Q)
-  \frac{1}{2} \nabla_0\Big( (\tau-\vert \bm\phi \vert^2_{h_{\epsilon}})(P) \varphi'_{\epsilon} (P) \Big)
\\&
+\frac{1}{2}\int \nabla_0 G_0 (P,Q) \Big(\varphi_{\epsilon}'\Delta_0  (\tau-\vert \bm\phi \vert^2_{h_{\epsilon}})
- 2\nabla_0  (\tau-\vert \bm\phi \vert^2_{h_{\epsilon}}) \cdot \nabla_0 \varphi'_{\epsilon}\Big) (Q) \omega_0(Q).
\end{split}\end{equation}
and using the fact that
$$
\left|\nabla_0 G_0(P,Q)\right|_0\leq C\left( 1+ \frac{1}{d_0(P,Q)}\right)
$$
(in particular this implies $\nabla_0 G_0\in L^q$ for any $q\in [1,2)$), we obtain
\begin{equation}
\left| \nabla_0 f'_{\epsilon} \right|_0\leq C.
\label{ineq:C1estimateforf'}
\end{equation}

\subsection{$L^\infty$ estimate of $f_{\epsilon}''$}
As in the case of the $C^1$-estimate, differentiating~\eqref{eq:firsttemporalderivative} with respect to the real time $t$, we get the following equation for $f_{\epsilon}''$:
\begin{equation}
\begin{split}
\Delta_0 f_{\epsilon}''
= -2\vert \bm\phi \vert_{h_{\epsilon}}^2 (f_{\epsilon} ')^2\frac{\omega_{\epsilon}}{\omega_0}-\vert \bm\phi \vert_{h_{\epsilon}}^2 f_{\epsilon} ''\frac{\omega_{\epsilon}}{\omega_0}-\frac{1}{2} (\tau-\vert \bm\phi
\vert^2_{h_{\epsilon}})\Delta_{\omega_0} \varphi''_{\epsilon}
+2 \vert \bm\phi \vert^2_{h_{\epsilon}} f_{\epsilon}'\Delta_{\omega_0} \varphi'_{\epsilon}.
\end{split}
\label{eq:secondtemporalderivative}
\end{equation}
It is convenient to rewrite the equation as
\begin{equation}
\left( \Delta_0 + \left| \bm\phi\right|_{h_\epsilon}^2 \frac{\omega_\epsilon}{\omega_0}\right) f_\epsilon''
=
-2\vert \bm\phi \vert_{h_{\epsilon}}^2 (f_{\epsilon} ')^2\frac{\omega_{\epsilon}}{\omega_0}
-\frac{1}{2} (\tau-\vert \bm\phi \vert^2_{h_{\epsilon}})\Delta_{\omega_0} \varphi''_{\epsilon}
+ 2 \vert \bm\phi \vert^2_{h_{\epsilon}} f_{\epsilon}'\Delta_{\omega_0} \varphi'_{\epsilon}.
\end{equation}
Next, we prove a lemma about the \emph{Green function} for the second-order linear elliptic differential operator
$$
L_\epsilon = \Delta_0+\left| \bm \phi\right|_{h_\epsilon}^2 \frac{\omega_\epsilon}{\omega_0}.
$$

\begin{lemma}
Let $\{v_{\epsilon}\}_{\epsilon\in (0,1]}$ be a set of smooth functions on the compact Riemann surface $\left(\Sigma, \omega_0\right)$ such that $0\leq v_{\epsilon}\leq C$ and $\int_\Sigma v_{\epsilon}\omega_0 \geq
\frac{1}{C}$, where $C$ is a constant independent of $\epsilon$ (varying on a compact set). Then there exists a family of functions $\{G_{\epsilon}(P,Q)\}_{\epsilon\in (0,1]}$, each of which
is smooth everywhere except on the diagonal, such that for any $1 < s < 2$ one has
\begin{enumerate}
\item[\textup{(1)}]
$0\leq G_{\epsilon}(P,Q) \leq C_s(1+\log \vert d_{\omega_0} (P,Q)\vert)$.
\item[\textup{(2)}]
$\int_\Sigma \vert \nabla_0 G_{\epsilon} \vert_0^{s} \omega_0 \leq C_s$.
\item[\textup{(3)}]
$(\Delta_0)_Q G_{\epsilon} (P,Q)+v_{\epsilon}(Q) G_{\epsilon} (P,Q)=\delta_P (Q)$ in the sense of distributions (where the Dirac delta function is defined with respect to $\omega_0$).
\end{enumerate}
\label{lem:Greenforperturbation}
\end{lemma}
\begin{proof}
Fix an $1<s<2$. We shall drop the subscript $s$ on constants $C_s$ that appear in this proof.  Recall that the standard Green's function $G_0 = G_{\omega_0}$ satisfies (see \eqref{eq:Greenformula})
$$
(\Delta_0)_Q G_0 (P,Q) = \delta_P (Q)- \frac{1}{V}.
$$
For every $\epsilon\in (0,1], P\in \Sigma$, let $u_{\epsilon,P}$ be the unique distributional solution of the PDE
\begin{gather}\label{eq:auxilliaryPDEforGreen}
(\Delta_0 u_{\epsilon,P}+v_{\epsilon} u_{\epsilon,P})(Q) = \frac{1}{V}-v_{\epsilon}(Q)G_0(P,Q).	
\end{gather}
This equation can be solved uniquely for $u_{\epsilon,P} \in W^{2,2}$, because the right-hand side is in $L^2$, and the kernel of $\Delta_0+v_{\epsilon}$ is trivial because of the assumptions on $v_{\epsilon}$. We
claim that
$$
\Vert u_{\epsilon,P} \Vert_{W^{2,2}}\leq C
$$
(where $C$ is independent of $\epsilon$ and $P$). To see this, we study the first eigenvalue $\lambda_{\epsilon}\geq 0$ of the operator $\Delta_0+v_{\epsilon}$. Suppose there exists a sequence of values of
$\epsilon$ such that $\lambda_{\epsilon}\rightarrow 0$. Then normalising an eigenvector $e_{\epsilon}$ to have $L^2$ norm $1$, we see that
$$
\Vert \nabla_0 e_{\epsilon}\Vert_{L^2}^2+ \int v_{\epsilon} e_{\epsilon}^2\omega_0\rightarrow 0.
$$
Therefore (up to a subsequence) $e_{\epsilon}\rightarrow \tilde{e}$ strongly in $L^2$ and weakly in $W^{1,2}$. Moreover, $\tilde{e}$ must be a constant (that is non-zero since it has $L^2$ norm $1$). However,
	\[
	\frac{\tilde e^2}{C}
	\leq
	\int v_\epsilon \tilde e^2\omega_0
	\leq
	2\int v_\epsilon \left( (\tilde e-e_\epsilon)^2 + e_\epsilon^2\right) \omega_0
	\rightarrow 0,
	\]
	which is a contradiction. Thus $\lambda_{\epsilon}\geq \frac{1}{C}>0$. This implies
	\begin{equation}
	\begin{split}
	\lVert u_{\epsilon, P}\rVert_{L^2}^2
	& \leq
	\frac{1}{\lambda_\epsilon}\int u_{\epsilon, P}\left( \Delta_0 + v_\epsilon\right) u_{\epsilon, P}\omega_0\\
	& \leq
	C\int h_{\epsilon,P} u_{\epsilon, P}\omega_0\\
	& \leq
	C \lVert h_{\epsilon,P}\rVert_{L^2} \lVert u_{\epsilon, P}\rVert_{L^2}
	\end{split}
	\end{equation}
	where $h_{\epsilon,P}(\cdot)= 1- v_\epsilon(\cdot) G_0(P, \cdot)$.
	This immediately implies the desired uniform $W^{2,2}$ bound:
	\begin{equation}
	\begin{split}
	\lVert u_{\epsilon, P}\rVert_{W^{2,2}}
	&  \leq
	C\left( \lVert u_{\epsilon, P}\rVert_{L^2} + \lVert\Delta_0 u_{\epsilon, P}\rVert_{L^2}\right)\\
	& \leq
	C\left( \lVert u_{\epsilon, P}\rVert_{L^2} + \lVert h_{\epsilon,P} \rVert_{L^2} + \lVert v_\epsilon \rVert_{L^\infty} \lVert u_{\epsilon, P}\rVert_{L^2}\right)\\
	& \leq
	C\lVert h_{\epsilon,P}\rVert_{L^2}\\
	& \leq
	C.
	\end{split}
\end{equation}
As a consequence (using the Sobolev embedding theorem), there is a uniform $L^\infty$ bound on $u_{\epsilon,P}$ (which is independent of $\epsilon, P$). Since $G_0(P, \cdot)$ is smooth on $\Sigma\backslash
\{P\}$, we know by elliptic regularity theory that $u_{\epsilon, P}$ is smooth on $\Sigma\backslash \{P\}$ for each fixed $\epsilon$. Let $d_P u_{\epsilon, P}$ be the exterior differential of $u_{\epsilon, P}$ with
respect to $P$. This 1-form satisfies an equation of the form
\begin{equation}
	\begin{split}
	\left( \Delta_0 + v_\epsilon \right) d_P u_{\epsilon,P}
	& = - v_\epsilon d_P G_0.
	\end{split}
\end{equation}
Since $d_P G_0\in L^p$ for any $p<2$, the 1-form $d_P G_0$ can be approximated by smooth 1-forms in the sense of $L^{s}$, for any $1 < s < 2$, and the approximating sequence has a uniform $L^{s}$ bound. For this
approximating family the above equation can be solved uniquely and $d_P u_{\epsilon,P}$ can be understood as the limit of such solutions.
	
Regarding the equation
\[
\left( \Delta_0 + v_\epsilon\right) w = v_\epsilon h,
\]
with $h$ smooth, we claim that the following inequality holds for a uniform constant $C$:
\begin{equation}\label{eqn:L1.9}
	\lVert w\rVert_{L^{s}}\leq C \lVert h\rVert_{L^{s}}.
\end{equation}
Suppose this is not true. Then there exists a sequence $\epsilon_n\to 0$ and $\left( w_n, h_n, v_{\epsilon_n}\right)$ such that $\lVert w_n\rVert_{L^{s}}=1$ for all $n$ but $\lVert h_n\rVert_{L^{s}}\to 0$ as $n\to
+\infty$, while
satisfying
	\[
	\Delta_0 w_n = v_{\epsilon_n}\left( h_n - w_n\right).
	\]
The standard $W^{2,s}$-estimate shows that
\[
\lVert w_n\rVert_{W^{2,s}}\leq C \left( \lVert w_n\rVert_{L^{s}} + \lVert v_{\epsilon_n}\rVert_{L^\infty} \left( \lVert h_n\rVert_{L^{s}}+ \lVert w_n\rVert_{L^{s}}\right)\right)
\leq C.
\]
We now have
	\[
	\int \left|\nabla_0 w_n\right|_0^2 \omega_0
	=
	\int v_{\epsilon_n} \left( h_n w_n - w_n^2\right)\omega_0
    \leq
    C \lVert h_n\rVert_{L^{s}}\lVert w_n\rVert_{L^{s^*}},
	\]
where we use the H\"older inequality and $s^*$ represents the conjugate index of $s$. The Sobolev inequality implies $\lVert w_n\rVert_{L^{s^*}}\leq \lVert w_n\rVert_{W^{2,s}}\leq C$, and $\lVert w_n - \fint w_n
\omega_0\rVert_{L^2}^2 \leq \int
\left|\nabla_0 w_n\right|_0^2 \omega_0\to 0$ by the Poincar\'e inequality. This in particular implies that (up to a subsequence) $w_n\to \tilde w$ in $L^2$ as $n\to \infty$ for some nonzero constant $\tilde w$ (this
is because $\lVert\tilde w \rVert_{L^{s}}=\lim_{n\to \infty}\lVert w_n\rVert_{L^{s}}=1$). Then we obtain a contradiction, since
 \[
\int v_{\epsilon_n} w_n^2\omega_0
=
\int v_{\epsilon_n}h_n w_n\omega_0 - \int |\nabla_0 w_n|_0^2\omega_0\to 0,
 \]
 and
 \[
\liminf_{n\to \infty}\int\! v_{\epsilon_n} w_n^2 \omega_0
=
\liminf_{n\to \infty}\left\{\int\! v_{\epsilon_n} \tilde w^2 \omega_0
+
2\tilde w \!\int\! v_{\epsilon_n} (w_n-\tilde w)\omega_0
+
\!\int\! v_{\epsilon_n}(w_n-\tilde w)^2\omega_0
\right\}
\geq
\frac{\tilde w^2}{C}.
 \]
 This contradiction establishes the estimate \eqref{eqn:L1.9} and thus further implies that $\lVert d_P u_{\epsilon,P}\rVert_{L^{s}}$ is uniformly bounded, and hence so is $u_{\epsilon,P}$ in $W^{1,s}$.

Define next
$$
G_{\epsilon}(P,Q) = G_0(P,Q)+u_{\epsilon}(P, Q).
$$
Then one can easily verify that the statement holds. In particular, the non-negativity of $G_\epsilon$ in item (1) could be obtained by approximating $v_\epsilon$ using $v_\epsilon+\delta_j$ with a sequence
$\delta_j$ decreasing to $0$, and using the maximum principle  for the equation
\[
\left( \Delta_0 \right)_Q G_\epsilon(P,Q) + \left(v_\epsilon+\delta_j\right)G_\epsilon(P,Q) = \delta_P(Q).
\qedhere\]
\end{proof}

Applying Lemma \eqref{lem:Greenforperturbation} to Equation \eqref{eq:secondtemporalderivative}, we see that
\begin{align}\nonumber
&f''_{\epsilon}(P)
=
\int G_{\epsilon} (P,Q) \left(-2\vert \bm\phi \vert_{h_{\epsilon}}^2 (f_{\epsilon} ')^2\frac{\omega_{\epsilon}}{\omega_0}-\frac{1}{2} (\tau-\vert \bm\phi \vert^2_{h_{\epsilon}})\Delta_{\omega_0} \varphi''_{\epsilon}
+2 \vert \bm\phi \vert^2_{h_{\epsilon}} f_{\epsilon}'\Delta_{\omega_0} \varphi'_{\epsilon} \right)(Q) \omega_0 (Q)
\\&\nonumber
=
-2\int G_{\epsilon} \vert \bm\phi \vert_{h_{\epsilon}}^2 (f_{\epsilon} ')^2\omega_{\epsilon}
-\frac{1}{2} \int \Delta_0 \left(G_\epsilon (\tau-\left|\bm\phi\right|_{h_\epsilon}^2)\right) \varphi''_{\epsilon}\omega_0
+ 2
\int \nabla_0 \left(  G_\epsilon \vert \bm\phi \vert^2_{h_{\epsilon}} f_{\epsilon}'\right)\cdot \nabla_0 \varphi'_{\epsilon}\omega_0
\\
\begin{split}
& =
- 2\int G_\epsilon \Phi_\epsilon \left(f_\epsilon'\right)^2 \omega_\epsilon
\\&\;\,
- \frac{1}{2}\int \left[ \left( \delta_P- \Phi_\epsilon\frac{\omega_\epsilon}{\omega_0} G_\epsilon\right)\left(\tau-\Phi_\epsilon\right)
+ G_\epsilon \Delta_0\left(\tau-\Phi_\epsilon\right) - 2\nabla_0 G_\epsilon \cdot \nabla_0\left(\tau-\Phi_\epsilon\right)
\right]
\varphi_\epsilon''\omega_0
\\&\;\,
+ 2 \int \left(  \nabla_0G_\epsilon \Phi_\epsilon f_\epsilon' + G_\epsilon \nabla_0\Phi_\epsilon f_\epsilon' + G_\epsilon \Phi_\epsilon \nabla_0 f_\epsilon'\right)\cdot \nabla_0\varphi_\epsilon'\omega_0.
\end{split}
\end{align}
Finally, by the a priori estimates for $\lVert\Phi_\epsilon\rVert_{L^\infty}$, $\lVert f_\epsilon'\rVert_{L^\infty}$, $\lVert\nabla_0 f_\epsilon'\rVert_{L^\infty}$,  $\lVert\Delta_0\Phi_\epsilon\rVert_{L^\infty}$,
$\lVert\varphi_\epsilon''\rVert_{L^\infty}$, $\lVert\nabla_0\Phi_\epsilon\rVert_{L^\infty}$, $\lVert\nabla_0\varphi_\epsilon'\rVert_{L^\infty}$, and $\lVert\nabla_0 G_\epsilon\rVert_{L^1}$, we arrive at a uniform
$L^\infty$ estimate of
$f_\epsilon''$, i.e.,
\begin{equation}
\lVert f_\epsilon''\rVert_{L^\infty}\leq C.
\end{equation}

\subsection{Convexity}\label{section:Convexity}
Next we apply the above estimates to prove the convexity of the reduced $\alpha$-K-energy $\mathcal{K}_\alpha$ along Chen's $C^{1,\bar 1}$ geodesic connecting any two smooth K\"ahler potentials $\varphi_0,\varphi_1
\in \mathcal{P}_{\omega_0}$. Following the notation from previous sections, for any $t \in [0,1]$, we define
\begin{equation}\label{eq:functional-along-epsilon-geodesic.1}
\mathcal{M}_{\alpha, \epsilon}(t)= \mathcal{M}_\alpha\left(f_\epsilon(t), \varphi_\epsilon(t)\right).
\end{equation}

\begin{proposition}
$\mathcal{M}_{\alpha, \epsilon}'(t)$ is uniformly bounded on $[0,1]$ for $\epsilon\in (0,1]$.
\end{proposition}

\begin{proof}
According to the definition \eqref{defn:primitive},
\begin{align*}
&\mathcal{M}_{\alpha, \epsilon}'(t)
=
-\int_\Sigma \varphi_\epsilon'(t)
\left( \left( \langle S\rangle - c\right) \omega_{\varphi_\epsilon(t)}
-2\alpha i\partial\bar\partial \lvert\bm\phi\rvert_{h_\epsilon(t)}^2
+ \alpha\tau \left( \left| \bm\phi\right|_{h_\epsilon(t)}^2 - \tau\right) \omega_{\varphi_\epsilon(t)}
\right)
\\&\quad
=\left( c-  \langle S\rangle \right)
\int_\Sigma \varphi_\epsilon'(t) \omega_{\varphi_\epsilon(t)}
-
2\alpha\int_\Sigma  i\partial \varphi_\epsilon'(t) \wedge\bar\partial \lvert\bm\phi\rvert_{h_\epsilon(t)}^2
-
\alpha\tau \int_\Sigma \varphi_\epsilon(t) \left( \left| \bm\phi\right|_{h_\epsilon(t)}^2
 - \tau\right) \omega_{\varphi_\epsilon(t)}.
\end{align*}
By the $C^{1,\bar 1}$ estimate of $\varphi_\epsilon$ and the uniform $C^1$ estimate of $f_\epsilon$ (see Section \ref{sect:C1estimate}), we conclude that
\[
\left| \mathcal{M}_{\alpha, \epsilon}'(t)\right|\leq C.
\qedhere\]
\end{proof}

We obtain next a uniform estimate for the second time-derivative along the curve.

\begin{proposition}\label{prop:estimate-second-time-derivative}
$\mathcal{M}_{\alpha, \epsilon}''(t)$ is uniformly bounded on $[0,1]$ for $\epsilon \in (0,1]$. Moreover, it is almost nonnegative at the endpoints $t=0,1$.
\end{proposition}

\begin{proof}
According to the formula for second-order derivatives in Proposition \ref{prop:second}, we have
\begin{align*}
\mathcal{M}_{\alpha, \epsilon}''(t)
& =
4\alpha \int_\Sigma \left| \mathrm{d}f_\epsilon'(t) + \frac{1}{2}\left( \tau - \Phi_\epsilon(t)\right)\mathrm{d}\varphi_\epsilon'(t)\right|_{\omega_0}^2 \omega_0
\\&\qquad
+
4\alpha\int_\Sigma \left| J \eta_{\varphi_\epsilon'(t)}\lrcorner \mathrm{d}_{A_\epsilon(t)}\bm\phi
- f_\epsilon'(t) \bm\phi\right|_{h_0}^2 e^{2f_\epsilon(t)}\omega_{\varphi_\epsilon(t)}
\\&\qquad
-
\int_\Sigma \left( \varphi_\epsilon''(t) - \left| \mathrm{d}\varphi_\epsilon'(t)\right|_{\omega_{\varphi_\epsilon(t)}}^2\right)
\left(\alpha \Delta_{\omega_0}\left|\bm\phi\right|_{h_\epsilon(t)}^2 \right)
\omega_0
\\&\qquad
-\int_\Sigma \left(  \alpha\tau \left( \left| \bm\phi\right|_{h_\epsilon(t)}^2-\tau\right)
+(\langle S\rangle - c) \right) \cdot \epsilon \omega_0,
\end{align*}
where we use the fact that $\dim_\mathbb{C}\Sigma=1$ for the integrands in the first and the third terms, and \eqref{eqn:Chen's-geodesic} for the last one. The integrand in the first term is uniformly
bounded by \eqref{eqn:various-estimates} and \eqref{ineq:C1estimateforf'}.	By the Cauchy--Schwarz inequality
$$
\lvert \langle a,c\rangle+b\vert^2\leq 2\vert a\rvert^2\lvert c\rvert^2+2\lvert b\rvert^2
$$
combined with \eqref{eqn:Chen's-geodesic}, the integrand in the second term is uniformly bounded by
\begin{equation}
\begin{split}
2& \left| \nabla^{1,0}\varphi_\epsilon'(t)\right|_{\omega_{\varphi_\epsilon(t)}}^2 \left| \mathrm{d}_{A_\epsilon(t)}\bm\phi\right|^2_{{h_\epsilon(t)}, \omega_0}\omega_0
+
2 f_\epsilon'(t)^2 \Phi_\epsilon(t) \omega_{\varphi_\epsilon(t)}\\
 &\qquad
 \leq
\varphi_\epsilon''(t) \left| \mathrm{d}_{A_0}\bm\phi + 2\bm\phi \partial f_\epsilon(t)\right|_{h_0,\omega_0}^2e^{2f_\epsilon(t)} \omega_{\varphi_\epsilon(t)}
+ 2  f_\epsilon'(t)^2\cdot \tau \omega_{\varphi_\epsilon(t)},
 \end{split}
\end{equation}
which is uniformly bounded from above by the previously established estimates on $f_\epsilon(t)$, $f_\epsilon'(t)$, $\left|\nabla_0 f_\epsilon(t)\right|_{\omega_0}$  (see \eqref{eqn:various-estimates}) and the
$C^{1,\bar 1}$ estimate of $\varphi_\epsilon$.
	
For the third term, using the bound
$$
\left| \mathrm{d}\varphi_\epsilon'(t)\right|_{\omega_{\varphi_\epsilon(t)}}^2\leq \varphi_\epsilon''(t)
$$
obtained from \eqref{eqn:Chen's-geodesic} and the $C^{1,\bar 1}$ estimate on $\varphi_\epsilon(t)$,  together with the uniform bound of $\left| \Delta_{\omega_0}\Phi_\epsilon(t)\right|$ (see
\eqref{eqn:various-estimates}), we obtain
 \begin{equation*}
 \begin{split}
\left|  -  \left( \varphi_\epsilon''(t) - \left| \mathrm{d}\varphi_\epsilon'(t)\right|_{\omega_{\varphi_\epsilon(t)}}^2\right)
 \left(
 \alpha \Delta_{\omega_0}\left|\bm\phi\right|_{h_\epsilon(t)}^2 \right)\right|
 & \leq
\alpha \varphi_\epsilon''(t)\left| \Delta_{\omega_0}\Phi_\epsilon(t)\right|
\leq
C.
\end{split}
\end{equation*}
The integrand in the last term is clearly uniformly bounded, and so the proof follows.
\end{proof}

Consider now a sequence $\epsilon_n \to 0$ as $n \to +\infty$. By the uniform $W^{2,p}$ bounds on $f_{\epsilon}, \varphi_{\epsilon}$ (\eqref{est:W2pestimate}) and the Sobolev embedding, a subsequence converges in
$C^{1,\alpha}(\Sigma)$ for any fixed $t\in [0,1]$. Actually, by the uniqueness of $\varphi(t)$ and the corresponding solution to the vortex equation of $\omega_{\varphi(t)}$, the above convergence happens for the
whole sequence. We now see (using formula \ref{eqn:M-extension} and the $C^{1,\alpha}$ convergence of $\varphi_{\epsilon_n}$) that
\begin{equation}\label{eq:functional-along-epsilon-geodesic.2}
\mathcal{M}_\alpha(t):= \mathcal{M}_\alpha(\varphi(t)) = \lim_{\epsilon\to 0} \mathcal{M}_{\alpha,\epsilon}(t)
\end{equation}
for any $t\in [0,1]$.

\begin{lemma}
\label{lem:C1continuity}
The function $t\mapsto \mathcal{M}_{\alpha}(t)$ is $C^1$.
\label{lem:c1dependenceformalpha}
\end{lemma}

\begin{proof}
This is a direct consequence of the second derivative estimate in Proposition \ref{prop:estimate-second-time-derivative}.
\end{proof}

We are ready to prove the first main result of this section.

\begin{theorem}\label{thm:K-alpha-convex}
The distributional second derivative of $t\mapsto\mathcal{M}_{\alpha}(t)$ satisfies
$$
\mathcal{M}_{\alpha, \epsilon}''(t)\geq -4\pi \alpha \tau N\epsilon.
$$
Therefore the function $t\mapsto\mathcal{K}_\alpha(t):=\mathcal{K}(t)+\mathcal{M}_\alpha(t)$ is weakly convex (i.e. the distributional second derivative is nonnegative) and continuous on $[0,1]$. In particular,
it is convex.
\end{theorem}

\begin{proof}
Recall the formula
\[
i\partial\bar\partial  \left| \bm\phi\right|_h^2
=
i \mathrm{d}_{A_h}\bm\phi\wedge_h\overline{\mathrm{d}_{A_h}\bm\phi}
-
 \left| \bm\phi\right|_h^2\cdot iF_h.
\]
Using the fact that $h_\epsilon(t)\defeq h_0e^{2f_\epsilon(t)}$ satisfies the vortex equation for $\omega_{\varphi_\epsilon(t)}$, we derive that
\begin{align*}
\mathcal{M}_{\alpha, \epsilon}''(t)
& =
4\alpha \lVert \mathrm{d}f'_\epsilon(t) + \eta_{\varphi'_\epsilon(t)}\lrcorner  iF_{h_\epsilon(t)}
\rVert^2
+ 4\alpha \lVert J\eta_{\varphi'_\epsilon(t)}\lrcorner \mathrm{d}_{A_{h_\epsilon(t)}} \bm\phi - f'_\epsilon(t)\bm\phi\rVert^2	\\
&\quad
+
2\alpha \int_\Sigma  \left( \varphi''_\epsilon(t) - \left| \mathrm{d}\varphi'_\epsilon(t)\right|_{\omega_{\varphi_\epsilon(t)}}^2\right)\cdot i\partial\bar\partial \left| \bm\phi\right|_{h_\epsilon(t)}^2\\
&\quad	-
\int_\Sigma \left( \varphi''_\epsilon(t) - \left| \mathrm{d}\varphi'_\epsilon(t)\right|_{\omega_{\varphi_\epsilon(t)}}^2\right)
\left(
 \alpha\tau \left( \left| \bm\phi\right|_{h_\epsilon(t)}^2-\tau\right)
+
(\langle S\rangle - c) \right)
\omega_{\varphi_\epsilon(t)}\\
& \geq
 \alpha \int_\Sigma \left( \varphi''_\epsilon(t) - \left| \mathrm{d}\varphi'_\epsilon(t)\right|_{\omega_{\varphi_\epsilon(t)}}^2\right) \left[ \left( \tau - \left| \bm\phi\right|_{h_\epsilon(t)}^2\right)^2
 - (\langle S\rangle - c)
 \right] \omega_{\varphi_\epsilon(t)}\\
& =
\alpha\epsilon \int_\Sigma  \left[ \left( \tau - \left| \bm\phi\right|_{h_\epsilon(t)}^2\right)^2
- (\langle S\rangle - c)
\right] \omega_0\\
& \geq
\alpha \left(c-\langle S\rangle\right)V\epsilon\\
& =
-4\pi \alpha\tau N \epsilon.
\end{align*}
Using the decomposition, $\mathcal{K}_\alpha = \mathcal{K}+\mathcal{M}_\alpha$, for any nonnegative test function $T\in C_0^\infty[0,1]$,
\begin{align*}
\langle \mathcal{K}_\alpha'', T\rangle
&=
\langle \mathcal{K}'', T\rangle
+
\langle \mathcal{M}_\alpha'', T\rangle
=\langle \mathcal{K}'', T\rangle
+
\lim_{\epsilon\to 0}\,\langle \mathcal{M}_{\alpha, \epsilon}'', T\rangle\\
& \geq
\lim_{\epsilon\to 0}\,\langle -4\pi\alpha\tau N \epsilon, T\rangle
= 0,
\end{align*}
where we have used the \emph{convexity} of the K-energy along $C^{1,\bar 1}$ geodesics established in \cite{BeBern,CLP}.
Since the term $\mathcal{M}_\alpha(t)$ is continuous by Lemma \ref{lem:C1continuity}, $\mathcal{K}_\alpha(t)$ is also continuous using the continuity of the usual K-energy \cite{BeBern, CLP}.
Convexity of $\mathcal{K}_\alpha(t)$ follows now from the basic fact that weak convexity and continuity imply convexity.
Indeed, if $f$ is such a one-variable function, then its convolution with an approximate of identity is convex and smooth. Hence $f(ta+(1-t)b)\leq tf(a)+(1-t)f(b)$ by making the convolution parameter go
to zero.
\end{proof}

To finish this section, in the next result we establish the convexity of the reduced $\alpha$-K-energy of Theorem \ref{thm:extended-vortex-K-energy}, along finite-energy geodesics.

\begin{theorem}\label{thm:convexfinen}
Given $\psi_0,\psi_1\in \overline{\mathcal{R}}$, let $\{\psi_t\}_{t\in [0,1]}\subset\mathcal{E}_1$ be the unique finite-energy geodesic connecting $\psi_0$ and $\psi_1$.
\begin{enumerate}
\item[\textup{(1)}]
The curve $\{\psi_t\}_{t\in [0,1]}$  is entirely contained in $\overline{\mathcal{R}}\subset\mathcal{E}_1$.
\item[\textup{(2)}]
The map
$$
[0,1]\lto\mathbb{R}\cup\{+\infty\},\quad t\longmapsto\mathcal{K}_\alpha(\psi_t),
$$
defined by Theorem \ref{thm:extended-vortex-K-energy}, is convex on $[0,1]$, and continuous at the endpoints $t=0,1$.
\end{enumerate}
\end{theorem}

\begin{proof}
Part (1) follows because the functional $\AMY$ is linear along weak geodesics (see~\cite[Lemma 5.2(ii)]{Darvas}), so the condition $\AMY(\psi)=0$ in~\eqref{eq:normalisationofphi} is preserved along a finite-energy
geodesic connecting any two points in $\overline{\mathcal{R}}$.
For part (2), we fix $t_0,t_1\in [0,1]$ with $t_0<t_1$. By Theorem~\ref{thm:extended-vortex-K-energy}, there exist sequences $\{\psi_{t_0}^k\}$ and $\{\psi_{t_1}^k\}$ in $\mathcal{P}_{\omega_0}$ such that
$d_1(\psi_{t_0}^k, \psi_{t_0})\to 0$ and $d_1(\psi_{t_1}^k, \psi_{t_1})\to 0$ as $k\to\infty$, and
\[
\mathcal{K}_\alpha(\psi_{t_0})=\lim_{k\to \infty} \mathcal{K}_\alpha(\psi_{t_0}^k), \;\;
\mathcal{K}_\alpha(\psi_{t_1})=\lim_{k\to \infty} \mathcal{K}_\alpha(\psi_{t_1}^k).
\]
Let $\{\psi_t^k\}_{t\in [t_0,t_1]}$ be the unique $C^{1,\bar{1}}$ geodesic connecting $\psi_{t_0}^k$ and $\psi_{t_1}^k$. Then $d_1(\psi_t^k, \psi_t)\to 0$ as $k\to
+\infty$ for all $t\in [t_0,t_1]$, by~\cite[Proposition 4.3]{BDL1}. By the  convexity of $\mathcal{K}_\alpha$ along $C^{1,\bar{1}}$ geodesics (Theorem~\ref{thm:K-alpha-convex}) and the $d_1$-lower semicontinuity of
$\mathcal{K}_\alpha$ (Theorem~\ref{thm:extended-vortex-K-energy}), for any $t\in [t_0, t_1]$, we have
\begin{equation*}
\begin{split}
\mathcal{K}_\alpha(\psi_t)
& \leq
\liminf_{k\to +\infty} \mathcal{K}_\alpha(\psi_t^k)\\
& \leq
\frac{t_1-t}{t_1-t_0}\lim_{k\to +\infty} \mathcal{K}_\alpha(\psi_{t_0}^k)
+
\frac{t-t_0}{t_1-t_0}\lim_{t\to +\infty}
\mathcal{K}_\alpha(\psi_{t_1}^k)\\
& =
\frac{t_1-t}{t_1-t_0} \mathcal{K}_\alpha(\psi_{t_0})
+
\frac{t-t_0}{t_1-t_0}
\mathcal{K}_\alpha(\psi_{t_1}).
\end{split}
\end{equation*}
Consequently, the result follows.
\end{proof}

\section{Existence and uniqueness of gravitating vortices}\label{sec:geometric-applications}

\subsection{Uniqueness of gravitating vortices}\label{section:Uniqueness}

Here we prove uniqueness of solutions to the gravitating vortex equations in any admissible K\"ahler class, in the absence of automorphisms.
As a preliminary result, we also prove that the reduced $\alpha$-K-energy is bounded from below on the space of K\"ahler potentials, provided that these equations have solutions.

We fix a compact Riemann surface $\Sigma$ of genus $g=g(\Sigma)$, a holomorphic line bundle $L$ over $\Sigma$ with positive degree
\[
0 < \int_{\Sigma}c_1(L) = N \in \mathbb{N},
\]
a non-zero holomorphic section $\bm\phi\in H^0(\Sigma,L)$, and positive real constants $V$ and $\tau$ satisfying the inequality \eqref{eqn:ineqtauV}. We also fix a real coupling constant $\alpha > 0$ and consider
the gravitating vortex equations \eqref{eqn:GV0} for pairs $(\omega,h)$ with total volume
\[
\int_\Sigma \omega = V.
\]
By Theorem~\ref{th:B-GP}, the inequality \eqref{eqn:ineqtauV} is the condition required for the existence of a (unique) solution to the first equation in~\eqref{eqn:GV0} for any given background K\"ahler metric
$\omega$.
Following our approach to solve the gravitating vortex equations using symplectic reduction by stages, we consider the system~\eqref{eqn:GV0} as a single equation
\begin{equation}\label{eq:GVagain}
S_\omega + \alpha\Delta_\omega  \lvert\bm\phi\rvert_{h_\omega}^2 -2i\alpha \tau \Lambda_\omega F_{h_\omega} = c,
\end{equation}
for $c \in \mathbb{R}$ as in \eqref{eq:constantcintro}, where the unknown is a K\"ahler metric $\omega$ with volume $V$ and $h_\omega$ is the unique Hermitian metric on $L$ solving the vortex equation (see Theorem
\ref{th:B-GP})
\begin{equation}\label{eq:Vagain}
i\Lambda_\omega F_{h_\omega} + \frac{1}{2}(\lvert\bm\phi\rvert_{h_\omega}^2-\tau)=0.
\end{equation}

We start by proving lower boundedness of the reduced $\alpha$-K-energy if~\eqref{eq:GVagain} has solutions.

\begin{proposition}\label{thm:lower-bound}
Let $\varphi_0, \varphi_1\in \mathcal{P}_{\omega_0}$ and $\{\varphi(t)\}_{t\in [0,1]}$ be the unique $C^{1,\bar 1}$-geodesic connecting $\varphi_0$ and $\varphi_1$. Then the extended reduced $\alpha$-K-energy
$\mathcal{K}_\alpha(t)$ along this $C^{1,\bar 1}$ geodesic, defined in Theorem~\ref{thm:K-alpha-convex}, satisfies
\begin{align*}
\lim_{t\to 0^+}\frac{\mathcal{K}_\alpha(t)-\mathcal{K}_\alpha(0)}{t}
&
 \geq -
\int_\Sigma  \varphi'(0)
\left(\Ric\omega_{\varphi_0}-2\alpha i\partial\bar\partial\lvert\bm\phi\rvert_{h_0e^{2f_0}}^2 - 2\alpha\tau iF_{h_0e^{2f_0}}-c\omega_{\varphi_0}\right).
\end{align*}
As a consequence, the functional $\mathcal{K}_\alpha$ is bounded from below on $\mathcal{P}_{\omega_0}$ if the triple $(\Sigma,L,\bm{\phi})$ admits a smooth solution of the reduced gravitating vortex equation
\eqref{eq:GVagain}.
\end{proposition}

\begin{proof}
The limit exists because the function $t\mapsto\mathcal{K}_\alpha(t):=\mathcal{K}(t)+\mathcal{M}_\alpha(t)$ is convex, by Theorem~\ref{thm:K-alpha-convex}. Furthermore, by definition,
\begin{align*}
\lim_{t\to 0^+}\frac{\mathcal{K}_\alpha(t) - \mathcal{K}_\alpha(0)}{t}
& =
\lim_{t\to 0^+}\frac{\mathcal{K}(t)-\mathcal{K}(0)}{t}
+
\lim_{t\to 0^+}\frac{\mathcal{M}_\alpha(t)-\mathcal{M}_\alpha(0)}{t}.
\end{align*}
According to \cite[Lemma 3.5]{BeBern},
\[
\lim_{t\to 0^+}\frac{\mathcal{K}(t)-\mathcal{K}(0)}{t}
\geq
- \int_\Sigma \varphi'(0)\left(\text{Ric }\omega_{\varphi_0} - \langle S\rangle \omega_{\varphi_0}\right).
\]
By the results of Section~\ref{section:Convexity}, the function $\mathcal{M}_{\alpha,\epsilon}(t)$ defined in~\eqref{eq:functional-along-epsilon-geodesic.1} is uniformly bounded in $C^2([0,1])$ and converges to the
function $\mathcal{M}_\alpha(t)$ defined in~\eqref{eq:functional-along-epsilon-geodesic.2} pointwise as $\epsilon\to 0$, so we obtain that $\mathcal{M}_{\alpha,\epsilon}'(t)$ converges to
$\mathcal{M}_{\alpha}'(t)$ uniformly on $[0,1]$ and
\begin{align*}&
\mathcal{M}_\alpha'(0)=\lim_{\epsilon\to 0}\mathcal{M}_{\alpha,\epsilon}'(0)
\\&\,
=\lim_{\epsilon\to 0}
\left[
\left(c-\langle S\rangle\right)
\!\int_\Sigma\varphi_\epsilon'(0)\omega_{\varphi_0}
-2\alpha\!\int_\Sigma i\partial \varphi_\epsilon'(0)\wedge\bar\partial \lvert\bm\phi\rvert_{h_0e^{2f_0}}^2
- \alpha\tau\!\int_\Sigma \varphi_\epsilon'(0) \left( \left| \bm\phi\right|_{h_0e^{2f_0}}^2 - \tau\right) \omega_{\varphi_0}
\right]
\\&\,
=\left( c-  \langle S\rangle \right)
\int_\Sigma  \varphi'(0) \omega_{\varphi_0}
+2\alpha\int_\Sigma \varphi'(0)\cdot i \partial\bar\partial \lvert\bm\phi\rvert_{h_0e^{2f_0}}^2
+2 \alpha\tau \int_\Sigma \varphi'(0)\cdot
iF_{h_0e^{2f_0}}.
\end{align*}
The formula claimed in the proposition follows.
\end{proof}

Proposition~\ref{thm:lower-bound} is used to prove uniqueness of solutions of \eqref{eq:GVagain} with fixed total volume $V$.
Our strategy follows closely the approach of Berman--Berndtsson to the uniqueness of constant scalar curvature K\"ahler metrics in a fixed K\"ahler class~\cite[Section 4]{BeBern}. The basic idea is to perturb the
reduced $\alpha$-K-energy $\mathcal{K}_\alpha$, adding a term $\lambda\mathcal{J}_\xi$ for small $\lambda>0$, where the functional $\mathcal{J}_\xi$ is defined in terms of a \emph{twisting form} $\xi$. As we will
see, for all $\lambda>0$ the perturbation $\lambda\mathcal{J}_\xi$ makes the twisted functional $\mathcal{K}_\alpha^\lambda:=\mathcal{K}_\alpha+\lambda\mathcal{J}_\xi$ strictly convex along $C^{1,\overline{1}}$
geodesics, so it can have at most one critical point. Uniqueness of solutions of \eqref{eq:GVagain} will follow because, in the absence of infinitesimal automorphisms of $(\Sigma,L,\bm\phi)$, the Implicit Function
Theorem implies that near any smooth critical point of $\mathcal{K}_\alpha$, there exists a critical point of
$\mathcal{K}_\alpha+\lambda\mathcal{J}_\xi$ with $0<\lambda\ll 1$, and so $\mathcal{K}_\alpha$ can have at most one critical point as well.
To apply this strategy, next we establish a perturbation argument.

\begin{lemma}\label{lem:perturb}
Suppose either \textup{(i)} $g(\Sigma)\geq 1$, or \textup{(ii)} $g(\Sigma)=0$ and $\bm\phi$ vanishes at least at $3$ points. Let $(\omega_0,h_0)$ be a solution of the gravitating vortex equations~\eqref{eqn:GV0}
with total volume $V>\frac{4\pi N}{\tau}$. Let $\xi$ be a strictly positive smooth $(1,1)$-form on $\Sigma$ such that $\int_\Sigma\xi=V$. Then there exists $\delta>0$ such that for any real constant $\lambda$ such
that $-\delta<\lambda<\delta$, there exists a solution $(\omega_\lambda,h_\lambda)$ of the \emph{twisted gravitating vortex equations}
\begin{equation}\label{eq:twisted-moment-mapeq}
\begin{split}
i\Lambda_\omega F_h + \frac{1}{2}\left(\left|\bm\phi\right|_h^2 - \tau\right) = 0,\\
S_\omega + \alpha \Delta_\omega \left|\bm\phi\right|_h^2 - 2\alpha\tau \Lambda_\omega iF_h - \lambda\left(\Lambda_\omega \xi - 1\right) = c.
\end{split}
\end{equation}
Moreover, $(\omega_\lambda,h_\lambda)$ is the unique solution to \eqref{eq:twisted-moment-mapeq} sufficiently close to $(\omega_0,h_0)$ in $C^{k,\alpha}$-norm.
\end{lemma}

\begin{proof}
Let $\mathcal{P}_{\omega_0}$ be the space of K\"ahler potentials (see~\eqref{eq:sp-Kahler-potentials.1}) and $\widetilde{\mathcal{P}}_{\omega_0}$ the space of pairs $(\varphi,f)$ with
$\varphi\in\mathcal{P}_{\omega_0}$, $f\in C^\infty(\Sigma,\RR)$ (see~\eqref{eq:sp-metric-potentials.1}). For $\varphi \in \mathcal{P}_{\omega_0}$, $f\in C^\infty(\Sigma,\RR)$, we denote
\[
\omega_\varphi = \omega_0 + 2i\partial\bar\partial \varphi,\quad h_f=h_0e^{2f}.
\]
Adding the twisting term $\lambda(\Lambda_\omega\xi-1)$ to~\eqref{eq:wp.1}, we define the map
$$
\wp=(\wp_1,\wp_2)\colon \widetilde{\mathcal{P}}_{\omega_0} \times \RR \lto C^\infty(\Sigma,\RR) \times C^\infty(\Sigma,\RR)
$$
by the following formulae for all $(f,\varphi,\lambda)\in\widetilde{\mathcal{P}}_{\omega_0} \times\RR$, where $(\omega,h):=(\omega_\varphi,h_f)$:
\begin{equation}\label{eq:twisted-moment-map.1}
\begin{split}
\wp_1(f,\varphi,\lambda)
& = \wp_1(f,\varphi) = i\Lambda_\omega F_h + \frac{1}{2}\left(\left|\bm\phi\right|_h^2 - \tau\right),\\
\wp_2(f,v,\lambda)
& = -S_\omega -\alpha \Delta_\omega \left|\bm\phi\right|_h^2 + 2\alpha\tau \Lambda_\omega iF_h + \lambda\left(\Lambda_\omega \xi - 1\right) + c.
\end{split}
\end{equation}
By the proof of \cite[Lemma 6.4]{Al-Ga-Ga-P} or \cite[Proposition 3.3, Proposition 3.4]{Yao}, the linearized operator
\begin{equation}\label{eq:twisted-moment-map.2}
D\wp|_{(\hat f, \hat \varphi, 0)}: L_{k+4}^2(\Sigma)\times L_{k+4}^2(\Sigma)/\mathbb{R}\longrightarrow L_{k+2}^2(\Sigma)\times L_{k}^2(\Sigma)/\mathbb{R}
\end{equation} is invertible. Assuming injectivity, surjectivity is a direct consequence of the moment-map interpretation of the gravitating vortex equations~\eqref{eqn:GV0} given in
Proposition~\ref{prop:momentmap-inttriples}, which implies that the operator is self-adjoint (cf. the proof of Lemma \ref{lem:sigmaclosed} and \cite[Proposition 4.6]{Al-Ga-Ga}). Alternatively, one can use the fact that
$$
S_{\omega_0}= 2\alpha\left|\mathrm{d}_{h_0}\bm\phi\right|_{\omega_0,h_0}^2 + \alpha\left(\tau -
\left|\bm\phi\right|_{h_0}^2\right)^2 + c,
$$
which guarantees that the first eigenvalue of $\frac{1}{2}\Delta_{\omega_0}$ is strictly greater than $c$. This guarantees the second equation in \cite[Equation
(3.15)]{Yao} is solvable, and hence the first equation is also solvable. As for the injectivity, the proof of \cite[Lemma 6.3]{Al-Ga-Ga-P} shows that, any element of the linearization of
\eqref{eq:twisted-moment-map.1} at $(0,0,0)$, given by
\begin{align*}
\mathrm{d}\dot f + \eta_{\dot \varphi}\lrcorner iF_{\hat h}=0,\\
J\eta_{\dot \varphi}\lrcorner \mathrm{d}_{A_{\hat h}}\bm\phi - \dot f \bm\phi =0,\\
\bar\partial \nabla^{1,0}_{\hat \omega}\dot \varphi=0,
\end{align*}
implies that $\zeta = A_{h_0}^\perp \eta_{\dot \varphi} + i \dot f$ is an infinitesimal automorphism of $(\Sigma,L,\bm\phi)$. Hence, by hypthesis, $\zeta$ must vanish identically and therefore $\dot \varphi$ is
constant and $\dot f=0$. Note that, in the case that $g(\Sigma) = 0$, the fact that $\zeta$ fixes $\bm\phi$ implies that the (real) holomorphic vector field $\eta_{\dot \varphi}$ generates a one-parameter family of
holomorphic automorphism of $\Sigma \cong \mathbb{P}^1$ that fixes the zeroes of $\bm\phi$; this implies the vector field must be zero except for the case when $\bm\phi$ vanishes exactly at two distinct points,
which is ruled out by hypothesis. The proof follows by application of the Implicit Function Theorem  (see \cite{Al-Ga-Ga-P,Yao}).
\end{proof}

We are ready to prove the main result of this section.

\begin{theorem}\label{thm:uniqueness}
Suppose either \textup{(i)} $g(\Sigma)\geq 1$, or \textup{(ii)} $g(\Sigma)=0$ and $\bm\phi$ vanishes at least at $3$ points.
Then smooth solutions to the gravitating vortex equations~\eqref{eq:GVagain} with fixed volume $V>\frac{4\pi N}{\tau}$ are unique.
\end{theorem}

\begin{proof}
We fix a background K\"ahler metric $\omega_0$ on $\Sigma$ with total volume $V$ and consider the space of K\"ahler potentials $\mathcal{P}_{\omega_0}$ (see~\eqref{eq:sp-Kahler-potentials.1}). Without loss of
generality we can assume that $\omega_0$ solves the reduced gravitating vortex equation \eqref{eq:GVagain}. Hence, by Proposition \ref{prop:convexalphaK}, we have a critical point of the reduced $\alpha$-K-energy at
$0 \in \mathcal{P}_{\omega_0}$. Taking any strictly positive smooth $(1,1)$-form $\xi$ on $\Sigma$ such that $\int_\Sigma\xi=V$, we define the functional
$$
\mathcal{J}_\xi \colon \mathcal{P}_{\omega_0}\lto \mathbb{R},
$$
by
\[
\mathcal{J}_\xi(\varphi)
= \int_0^1 dt \int_\Sigma \dot\varphi_t \left( \xi-\omega_{\varphi_t}\right)
=\int_\Sigma \varphi \xi - \frac{1}{2}\int_\Sigma \varphi\left(\omega_0+\omega_\varphi\right),
\]
where $\varphi_t$ is any smooth path connecting $0$ with $\varphi$. Applying now Lemma \ref{lem:perturb}, we obtain a one-parameter family of critical points $\varphi_{0}^\lambda$ (parametrized by small real
constants $-\delta<\lambda<\delta$) of the the family of functionals
$$
\mathcal{K}_\alpha^\lambda:=\mathcal{K}_\alpha+\lambda\mathcal{J}_\xi.
$$
Suppose we have another solution $(\omega_1,h_1)$ of the gravitating vortex equation \eqref{eqn:GV0}. Then $\omega_1$ solves the reduced gravitating vortex equation \eqref{eq:GVagain}, and hence by Proposition
\ref{prop:convexalphaK} we have two critical points of the reduced $\alpha$-K-energy on $\mathcal{P}_{\omega_0}$. Applying again Lemma \ref{lem:perturb}, we obtain a second one-parameter family of critical points
$\varphi_{1}^\lambda$ of the functional $\mathcal{K}_\alpha^\lambda$.

Let $\varphi_{t}^\lambda$ be the $C^{1,\bar 1}$ geodesic connecting $\varphi_0^\lambda$ and $\varphi_1^\lambda$ in the space of K\"ahler potentials, and $f_t^\lambda= f_{\varphi_t^\lambda}$ the weak solution of the
vortex equation, given by Proposition \ref{prop:W-1-2-estimate}, corresponding to $\omega_{\varphi_t^\lambda}$.
By construction, the functional $\mathcal{J}_\xi(\varphi_t^\lambda)$ is differentiable along the $C^{1,\bar 1}$ geodesic $\varphi_t^\lambda$, since its derivative is
\[
\frac{\mathrm{d}}{\mathrm{d}t}\bigg\rvert_{t=0}\mathcal{J}_\xi(\varphi_t^\lambda)=\int_\Sigma \dot\varphi_t^\lambda\bigg|_{t=0}\cdot \left(\xi - \omega_{\varphi_0^\lambda}\right).
\]
(cf., e.g.,~\cite[Section 3.1.1]{BeBern}). Therefore, as in Theorem \ref{thm:K-alpha-convex}, it makes sense to consider the functional $\mathcal{K}_\alpha^\lambda:=\mathcal{K}_\alpha+\lambda\mathcal{J}_\xi$ along
the  $C^{1,\bar 1}$ geodesic, and we have `sub-slope inequality' analogous to Proposition \ref{thm:lower-bound}, i.e.,
 \[
 \lim_{t\to 0^+} \frac{\mathcal{K}_\alpha^\lambda(\varphi_t^\lambda) - \mathcal{K}_\alpha^\lambda(\varphi_0^\lambda)}{t}
 \geq
\int_\Sigma \dot\varphi_t^\lambda\bigg\rvert_{t=0^+}\cdot \wp_2(f_{\varphi_0^\lambda}, \varphi_0^\lambda,\lambda)\omega_{\varphi_0^\lambda}
=0,
 \]
and similarly
 \[
 \lim_{t\to 1^-}\frac{\mathcal{K}_\alpha^\lambda(\varphi_1^\lambda) - \mathcal{K}_\alpha^\lambda(\varphi_t^\lambda)}{1-t}
 \leq
 \int_\Sigma \dot\varphi_t^\lambda\bigg\rvert_{t=1^-}\cdot \wp_2(f_{\varphi_1^\lambda}, \varphi_1^\lambda,\lambda)\omega_{\varphi_1^\lambda}
 =0.
 \]
As a consequence,
\begin{align*}
0&\geq\frac{\mathrm{d}}{\mathrm{d}t}\bigg|_{t=1}^-\mathcal{K}_\alpha^\lambda\left(\varphi_t^\lambda\right)
-\frac{\mathrm{d}}{\mathrm{d}t}\bigg|_{t=0}^+\mathcal{K}_\alpha^\lambda\left(\varphi_t^\lambda\right)
\\
&=\frac{\mathrm{d}}{\mathrm{d}t}\bigg|_{t=1}^-\mathcal{K}_\alpha\left(\varphi_t^\lambda\right)
-\frac{\mathrm{d}}{\mathrm{d}t}\bigg|_{t=0}^+\mathcal{K}_\alpha\left(\varphi_t^\lambda\right)
+\lambda\left[\frac{\mathrm{d}}{\mathrm{d}t}\bigg|_{t=1}\mathcal{J}_\xi\left(\varphi_t^\lambda\right)
-\frac{\mathrm{d}}{\mathrm{d}t}\bigg|_{t=0}\mathcal{J}_\xi\left(\varphi_t^\lambda\right)
\right]\\
& \geq
\lambda\left[\frac{\mathrm{d}}{\mathrm{d}t}\bigg|_{t=1}\mathcal{J}_\xi\left(\varphi_t^\lambda\right)
-\frac{\mathrm{d}}{\mathrm{d}t}\bigg|_{t=0}\mathcal{J}_\xi\left(\varphi_t^\lambda\right)
\right].
\end{align*}
Finally, we apply the \emph{strict convexity} of the functional $\mathcal{J}_\xi$.
More precisely, the Mabuchi distance $d\left( \omega_{\varphi_0^\lambda}, \omega_{\varphi_1^\lambda}\right)$ has an upper bound~\cite[Proposition 4.1]{BeBern}
\begin{equation}\label{eq:thm:uniqueness.1}
\frac{\mathrm{d}}{\mathrm{d}t}\bigg|_{t=1}\mathcal{J}_\xi\left(\varphi_t^\lambda\right)
-
\frac{\mathrm{d}}{\mathrm{d}t}\bigg|_{t=0}\mathcal{J}_\xi\left(\varphi_t^\lambda\right)
\geq
C \cdot d\left( \omega_{\varphi_0^\lambda}, \omega_{\varphi_1^\lambda}\right)^2,
\end{equation}
where $C>0$ depends on the upper bound of $\omega_{\varphi_t^\lambda}$ (which can be taken uniformly independent of $t$ and $\lambda$) and lower bound of $\xi$ in terms of the fixed background K\"ahler metric
$\omega_0$.
Therefore $d\left(\omega_{\varphi_0^\lambda}, \omega_{\varphi_1^\lambda}\right)=0$ and so $\omega_{\varphi_0^\lambda}=\omega_{\varphi_1^\lambda}$ for all $\lambda\in (0,\delta)$. We conclude now that $\omega_1 =
\omega_{\varphi_1^0}  = \omega_{\varphi_0^0} = \omega_0$ by taking the limit $\lambda\to 0^+$.
\end{proof}

We finish this section with an explicit example of the strict convexity of the functional $\mathcal{J}_\xi$ on the Riemann sphere.

\begin{example}\label{ex:strict-convexity}
Suppose $\Sigma=\PP^1$ is the complex projective line and $\omega_0=\frac{V}{2\pi}\omega_{FS}$, where $\omega_{FS}$ is the Fubini--Study metric on $\mathbb{P}^1$ (with total volume $2 \pi$).
Let $\{\sigma_t\}_{t\in\RR}$ be the 1-parameter subgroup of automorphisms of $\PP^1$ given by $\sigma_t([z_0:z_1])=[e^{-t}z_0: e^t z_1]$ in homogeneous coordinates $[z_0:z_1]$, and $\left\{\varphi_t\right\}_{t\in
\mathbb{R}}$ the corresponding geodesic line in the space of K\"ahler potentials $\mathcal{P}_{\omega_0}$ (with $\varphi_0=0$), as in Example~\ref{ex:autgeod}.
Choosing coordinates $z=z_1/z_0$ on $\PP^1$, the automorphisms are $\sigma_t(z)=e^{2t}z$ and the geodesic on $\mathcal{P}_{\omega_0}$ is given by
\begin{equation}\label{eq:ex:strict-convexity.1}
\varphi_t= \frac{V}{4\pi}\left( \log \frac{1+e^{4t}\lvert z\rvert^2}{1+\lvert z\rvert^2}-2t\right).
\end{equation}
Then
\[
\frac{\mathrm{d}^2}{\mathrm{d}t^2}\mathcal{J}_{\omega_0}(\varphi_t)
=
\left( \frac{V}{2\pi}\right)^2\int_\mathbb{C}
\frac{8 e^{4t}|z|^2}{\left(1+e^{4t}\lvert z\rvert^2\right)^2} \frac{i\mathrm{d}z\wedge\mathrm{d}\bar{z}}{\left(1+\lvert z\rvert^2\right)^2},
\]
and
\[
\frac{\mathrm{d}}{\mathrm{d}t}\mathcal{J}_{\omega_0}(\varphi_t)\bigg|_0^{t=+\infty}=
\frac{V^2}{2\pi}.
\]

\end{example}

\subsection{Properness of the reduced $\alpha$-K-energy}\label{section:Properness}

Our first result in this section is uniqueness and regularity of minimizers of the extended reduced $\alpha$-K-energy $\mathcal{K}_\alpha$, provided that the gravitating vortex equations have smooth solutions. More
precisely, we show that the following $\mathcal{E}_1$-minimizer set of the extended reduced $\alpha$-K-energy $\mathcal{K}_\alpha$ has a unique element:
\begin{equation}\label{eq:E_1-minimizer.1}
\mathfrak{M}^1:=\{\psi\in\overline{\mathcal{R}} \mid \mathcal{K}_\alpha(\psi)=\inf_{\varphi\in\overline{\mathcal{R}}} \mathcal{K}_\alpha(\varphi)\}.
\end{equation}
Note here that $\inf_{\mathcal{E}_1}\mathcal{K}_\alpha=\inf_{\overline{\mathcal{R}}}\mathcal{K}_\alpha$, because $\mathcal{K}_\alpha$ is invariant under addition of constants.
We follow the approach of Berman--Darvas--Lu~\cite{BDL2}, with the K-energy replaced by $\mathcal{K}_\alpha$.

\begin{theorem}\label{thm:uniqueness-of-weak-minimizer}
Suppose either \textup{(i)} $g(\Sigma)\geq 1$, or \textup{(ii)} $g(\Sigma)=0$ and $\bm\phi$ vanishes at least at $3$ points.
If there exists a smooth solution $\omega_\psi=\omega_0+2i\partial\bar\partial \psi$ to the gravitating vortex equation~\eqref{eq:GVagain} on $(\Sigma,L,\bm\phi)$, where $\psi\in\mathcal{R}$, then
$\mathfrak{M}^1=\{\psi\}$.
\end{theorem}

\begin{proof}
Let $v\in \mathfrak{M}^1$. We need to show that $v=\psi$. Let $\{v_j\}\subset\mathcal{R}$ be a sequence such that $d_1(v, v_j)\to 0$ as $j\to\infty$.
As in the proof of Theorem~\ref{thm:uniqueness}, we consider the twisted reduced $\alpha$-K-energy $\mathcal{K}_\alpha + \lambda \mathcal{J}_{\xi_j}$, where the twisting form is $\xi_j=\omega_{v_j}:=\omega_0+2i\partial\bar\partial v_j$.

Following the proof of \cite[Proposition 3.1]{BDL2}, one can easily show that for any $\lambda>0$ there exists $v_j^\lambda\in\overline{\mathcal{R}}$ that belongs to the $\mathcal{E}_1$-minimizer set of the
functional $\mathcal{K}_\alpha+\lambda\mathcal{J}_{\xi_j}$ (this set is defined as in~\eqref{eq:E_1-minimizer.1}, replacing $\mathcal{K}_\alpha$ by $\mathcal{K}_\alpha+\lambda\mathcal{J}_{\xi_j}$; note again that $\inf_{\mathcal{E}_1}(\mathcal{K}_\alpha+\lambda\mathcal{J}_{\xi_j})=\inf_{\overline{\mathcal{R}}}(\mathcal{K}_\alpha+\lambda\mathcal{J}_{\xi_j})$, because $\mathcal{K}_\alpha$ and $\mathcal{J}_{\xi_j}$ are invariant under addition of constants).
Additionally, for all $w\in \mathfrak{M}^1$ and $\lambda\geq 0$,
\[
I(v_j^\lambda, v_j)\leq 2I(w, v_j),
\]
where the Aubin--Mabuchi--Yau functional (extended to $\mathcal{E}_1$) is given for all $u_0,u_1\in\mathcal{E}_1$ by
\[
I(u_0,u_1)=V^{-1}\int_\Sigma(u_0-u_1)(\omega_{u_1}-\omega_{u_0})
\]
(cf. \cite[Section 2.3]{BDL2}).
Furthermore, by Lemma~\ref{lem:perturb}, there exists $\delta>0$ such that for all $\lambda\in(-\delta,\delta)$, the smooth solution $\psi\in\mathcal{R}$ of the gravitating vortex equation can be perturbed to a
smooth solution $\psi_j^\lambda\in\mathcal{R}$ of the twisted gravitating vortex equation
\begin{equation}\label{eq:thm:uniqueness-of-weak-minimizer.2}
\wp_{j,2}(f_{\psi_j^\lambda},\psi_j^\lambda,\lambda)=0,
\end{equation}
where $\wp_{j,2}$ is given by~\eqref{eq:twisted-moment-map.1} with $\xi$ replaced by $\xi_j$.

Now, when $0<\lambda<\delta$, the functional-theoretic minimizer $v_j^\lambda\in\overline{\mathcal{R}}$ and the (smooth) twisted gravitating vortex solution $\psi_j^\lambda\in\mathcal{R}$ can be joined by a
finite-energy geodesic $\{u_{j,t}^\lambda\}_{t\in [0,1]}$, that is contained in $\overline{\mathcal{R}}$, by Theorem~\ref{thm:convexfinen}.
By Theorem~\ref{thm:convexfinen} and the positivity of $\xi_j$, the functional $\mathcal{K}_\alpha + \lambda \mathcal{J}_{\xi_j}$ is convex along this finite-energy geodesic, so
\begin{equation}\label{eq:thm:uniqueness-of-weak-minimizer.1}
\frac{\mathrm{d}}{\mathrm{d}t}\bigg|_{t=0^+}\left( \mathcal{K}_\alpha + \lambda \mathcal{J}_{\xi_j}\right)(u_{j,t}^\lambda)
\leq
\frac{\mathrm{d}}{\mathrm{d}t}\bigg|_{t=1^-}\left( \mathcal{K}_\alpha + \lambda \mathcal{J}_{\xi_j}\right)(u_{j,t}^\lambda).
\end{equation}
The left-hand side of~\eqref{eq:thm:uniqueness-of-weak-minimizer.1} is non-negative, as $u_{j,0}^\lambda = v_j^\lambda$ is a minimizer of the functional $\mathcal{K}_\alpha+\lambda \mathcal{J}_{\xi_j}$ on
$\mathcal{E}_1$.
For the right-hand side, we first observe that for any two (smooth) K\"ahler potentials $\varphi_0,\varphi_1\in\mathcal{R}\subset\mathcal{P}_{\omega_0}$, the ``sub-slope'' inequality for the convex function
$\left(\mathcal{K}_\alpha + \lambda \mathcal{J}_{\xi_j}\right)(\varphi_t)$ along the $C^{1,\bar{1}}$ geodesic $\{\varphi_t\}_{t\in [0,1]}$ connecting $\varphi_0$ and $\varphi_1$  gives
 \[
\frac{\left( \mathcal{K}_\alpha+\lambda \mathcal{J}_{\xi_j}\right)(\varphi_1) - \left(\mathcal{K}_\alpha + \lambda \mathcal{J}_{\xi_j}\right)(\varphi_t)}{1-t}
\leq
\int_\Sigma \dot\varphi_t\bigg\lvert_{t=1^-}\cdot \wp_{j,2}(f_{\varphi_1}, \varphi_1, \lambda)\omega_{\varphi_1}.
\]
Actually, in our case the endpoint $\varphi_0=v_j^\lambda$ has only $\mathcal{E}_1$ regularity, while the endpoint $\varphi_1=\psi_j^\lambda$ is smooth, so we first take a sequence
$\{v_{j,k}^\lambda\}\subset\mathcal{R}$ such that $v_{j,k}^\lambda\to v_j^\lambda$ as $k\to\infty$ and then apply the above inequality to the $C^{1,\bar{1}}$ geodesic $\{\varphi^\lambda_{j,k,t}\}_{t\in [0,1]}$
joining $v_{j,k}^\lambda$ and $\psi_j^\lambda$.
By~\cite[Lemma 3.4]{BDL2}, there exists a subsequence of $C^{1,\bar{1}}$ geodesics, still denoted $\{\varphi^\lambda_{j,k,t}\}_{t\in [0,1]}$, such that the velocities at $t=1$ satisfy
$\dot{\varphi}^\lambda_{j,k,1}\to\dot{u}_{j,1}^\lambda$ as $k\to\infty$ almost everywhere with respect to the smooth measure $\omega_{\varphi_1}$, where $\varphi_1=\psi_j^\lambda$.
Then the above ``sub-slope" inequality can be passed to the limit finite-energy geodesic $\{u_{j,t}^\lambda\}_{t\in [0,1]}$ by the Dominated Convergence Theorem, yielding the following inequality for the right-hand
side of~\eqref{eq:thm:uniqueness-of-weak-minimizer.1}:
\[
\frac{\mathrm{d}}{\mathrm{d}t}\bigg|_{t=1^-}\left( \mathcal{K}_\alpha + \lambda \mathcal{J}_{\xi_j}\right)(u_{j,t}^\lambda)
\leq
\int_\Sigma \dot u_{j,t}^\lambda\bigg\lvert_{t=1^-}\cdot \wp_{j,2}(f_{\psi_j^\lambda}, \psi_j^\lambda, \lambda)\omega_{\psi_j^\lambda}
=
0.
\]
 For $\lambda>0$, the functional $\mathcal{K}_\alpha+\lambda \mathcal{J}_{\xi_j}$ is strictly convex, since $\mathcal{K}_\alpha$ is convex and $\mathcal{J}_{\xi_j}$ is strictly convex. The fact that $\mathcal{J}_{\xi_j}$ is strictly convex along geodesics with \emph{smooth} endpoints is a consequence of \cite[Proposition 4.1]{BeBern}. However, in the present situation, the geodesic is a finite-energy geodesic connecting a smooth endpoint with a non-smooth endpoint. To justify strict convexity, we appeal to Theorem 4.12 of \cite{BDL1}. This theorem  requires a hypothesis of $\omega_{\varphi_t}$ being subordinate to $\omega_0$ (in the sense of measures). This hypothesis holds in our case. Indeed, if $\mathcal{K}_{\alpha}$ is linear, then since we are considering a finite-energy geodesic (with the $I$-functional being zero), $\mathcal{M}_{\alpha}$ is bounded and hence $\mathcal{K}$ is bounded. Hence, the entropy term is bounded. This can only happen if
$\omega_{\varphi_t}$ is subordinate to $\omega$, i.e., it has a Radon-Nikodym density. Otherwise, by definition, the entropy is $+\infty$.
The strictly convexity implies that $d_1$-distance between $\omega_{v_j^\lambda}$ and $\omega_{\psi_j^\lambda}$ vanishes, i.e. $v_j^\lambda=\psi_j^\lambda+C$ for some $C\in \mathbb{R}$, that must be $0$ because
$\AMY(v_j^\lambda)=\AMY(\psi_j^\lambda)=0$.
Now, by \cite[Equation (3.1)]{BDL2},
\[
I(\psi_j^{\lambda},v)=I(v_j^\lambda, v)
\leq
\frac{3}{c_1}I(v_j,v),
\]
where $c_1>0$ is a constant such that $c_1I(u_0,u_1)\leq I(u_0,v)+I(u_1,v)$ for all $u_0,u_1,v\in\mathcal{E}_1$ (see~\cite[(2.16)]{BDL2}).
Taking the limit $\lambda\to 0$ for any fixed $j$, we have
\[
I(\psi, v)\leq\frac{3}{c_1}I(v_j,v).
\]
This implies $\psi=v+C$ by taking $j\to +\infty$ and using the $d_1$-continuity of the functional $I$. Finally, we conclude that $C=0$, by the normalization condition $\AMY(v)=\AMY(\psi)$, so $v=\psi$.
\end{proof}

Theorem~\ref{thm:uniqueness-of-weak-minimizer} is now used to prove properness of the reduced $\alpha$-K-energy.

\begin{theorem}\label{thm:K-alpha-proper}
Suppose either \textup{(i)} $g(\Sigma)\geq 1$, or \textup{(ii)} $g(\Sigma)=0$ and $\bm\phi$ vanishes at least at $3$ points.
If there exists a smooth solution $\omega_\psi=\omega_0+2i\partial\bar\partial\psi$ to the gravitating vortex equation~\eqref{eq:GVagain} on $(\Sigma,L,\bm\phi)$, where $\psi\in\mathcal{R}$, then the reduced
$\alpha$-K-energy $\mathcal{K}_\alpha\colon\mathcal{R}\to\mathbb{R}$ given by~\eqref{def:K-energy.before-extension} is proper, that is, there exists constants $C,D>0$ such that for all $u\in\mathcal{R}$,
\[
\mathcal{K}_\alpha(u)\geq Cd_1(0,u)-D.
\]
\end{theorem}

\begin{proof}
It suffices to check that $\mathcal{K}_\alpha\colon\mathcal{R}\to\mathbb{R}$ and its largest $d_1$-lower semicontinuous extension $\mathcal{K}_\alpha\colon\overline{\mathcal{R}}\to\mathbb{R}\cup\{+\infty\}$ given
by~\eqref{def:extended-alpha-K-energy} (see Theorem~\ref{thm:extended-vortex-K-energy}), satisfy the conditions in Theorem 3.4 of Darvas--Rubinstein~\cite{Darvas}, listed in~\cite[Hypothesis 3.2]{Darvas}:
\begin{enumerate}
\item[(i)] Property (P1) follows because any $\varphi_0,\varphi_1\in\mathcal{R}$ can be joined by Chen's $C^{1,\bar 1}$ geodesic $\{\varphi_t\}_{t\in [0,1]}$, this geodesic is contained in
    $\overline{\mathcal{R}}$, and $\mathcal{K}_{\alpha}(\varphi_t)$ is continuous and convex, by Theorems~\ref{thm:Chen.1}, \ref{thm:K-alpha-convex} and~\ref{thm:convexfinen}.
\item[(ii)] To prove property (P2), we fix a sequence $\{u_j\}\subset\overline{\mathcal{R}}$ such that $\lim_{j\to +\infty}\mathcal{K}_\alpha(u_j)= \inf_{\overline{\mathcal{R}}}\mathcal{K}_\alpha$ and there exists
    $C>0$ such that $d_1(0,u_j)\leq C$ for all $j$.
Then $\lVert u_j\rVert_{W^{1,2}}\leq C$ by Lemma \ref{lem:equivalent-norm}.
By the uniform $W^{1,2}$ estimates for the sequences $f_{u_j}$ and $\vert\bm\phi\vert_{h_0e^{2f_{u_j}}}^2=\vert\bm\phi\vert_{h_0}^2 e^{2f_{u_j}}$ obtained in the proof of Proposition \ref{prop:W-1-2-estimate}, we
have $\Vert f_{u_j} \Vert_{W^{1,2}}\leq C$ and $\big\lVert\vert\bm\phi\vert_{h_0}^2e^{2f_{u_j}}\big\rVert_{W^{1,2}}\leq C$.
Therefore, applying formulae~\eqref{def:extendedK-energy} and~\eqref{eqn:M-extension} in~\eqref{def:extended-alpha-K-energy}, it is easy to see that all the terms for $\mathcal{K}_\alpha(u_j)$ other than the
entropy term are bounded.
Hence the entropy term $\text{Ent}(\omega_0,\omega_{u_j})$ is uniformly bounded. Green's formula shows that $\lvert\sup_\Sigma u_j \rvert\leq C$.
By~\cite[Theorem 5.6]{Darvas}, these last two uniform estimates imply that there exists a $d_1$-convergent subsequence $\{u_{j_k}\}$ such that $d_1(u_{j_k}, u)\to 0$ for some $u\in\overline{\mathcal{R}}$.
By the lower semicontinuity of $\mathcal{K}_\alpha$ on $\overline{\mathcal{R}}$ (Theorem~\ref{thm:extended-vortex-K-energy}), we have
\[
\mathcal{K}_\alpha(u)\leq \liminf_{n\to +\infty} \mathcal{K}_\alpha(u_{j_n})= \inf_{\overline{\mathcal{R}}}\mathcal{K}_\alpha,
\]
i.e., $u$ is a minimizer in $\overline{\mathcal{R}}$, so $u=\psi$ by Theorem~\ref{thm:uniqueness-of-weak-minimizer}, and hence $u\in\mathcal{R}$, as required.
\item[(iii)] Properties (P3) and (P5) follow from Theorem~\ref{thm:uniqueness-of-weak-minimizer}.
\item[(iv)] Properties (P4), (P6), and (P7) are automatic because $G$ is trivial.
\qedhere
\end{enumerate}
\end{proof}

\subsection{Stability in genus $g = 0$}\label{section:Polystability}

Now that we have established properness of the reduced $\alpha$-K-energy (Theorem~\ref{thm:K-alpha-proper}), we are in a position to prove the main theorem of this paper.

\begin{theorem}\label{thm:existence-implies-stability}
If there exists a smooth solution to the gravitating vortex equation~\eqref{eq:GVagain} on $(\mathbb{P}^1,L,\bm\phi)$, then the effective divisor $D=\sum_j n_jp_j$ corresponding to the pair $(L,\bm\phi)$ is GIT
polystable.
\end{theorem}

This result provides an algebro-geometric obstruction to the existence of solutions to the gravitating vortex equation on the complex projective line (see~\cite[Proposition 2.6]{Al-Ga-Ga-P} for an explicit
formulation of GIT polystability). The proof given here rectifies the proof in~\cite[Theorem 1.3]{Al-Ga-Ga-P}.

A basic link between the gravitating vortex equation and Geometric Invariant Theory, used in the proof of Theorem~\ref{thm:existence-implies-stability}, is given by the \emph{gravitating vortex Futaki
invariant}~\cite[Section 4.2]{Al-Ga-Ga-P}.
To review its construction, we fix a K\"ahler metric $\omega$ on $\PP^1$ and a hermitian metric $h$ on $L$.
Then the value of this invariant on any $(1,0)$-holomorphic vector field $\bm v^{1,0}$ whose real part generates holomorphic automorphisms of the triple $(\Sigma, L, \bm\phi)$ is given by
\begin{equation}
\label{def:Gravitating vortexFutaki}
\begin{split}
\langle\cF_{\alpha,\tau},\bm v^{1,0}\rangle
&\defeq
8\alpha \int_\Sigma i\psi_h
\left(iF_h + \frac{1}{2}\left(\left|\bm\phi\right|_h^2-\tau\right)\omega\right)\\&\quad
-2\int_\Sigma \varphi_\omega \left(\text{Ric }\omega - 2\alpha i\partial\bar\partial \left|\bm\phi\right|_h^2 - 2\alpha\tau \cdot iF_h
-c\omega
\right).
\end{split}\end{equation}
Here, $\varphi_\omega$ and $\psi_h$ are the potential functions such that
\begin{gather*}
\iota_{\bm v^{1,0}}\omega
= \bar\partial \varphi_\omega, \qquad \int_\Sigma \varphi_\omega \omega=0,\\
\iota_{\bm v^{1,0}}\mathrm{d}_{A_h} \bm\phi
= \psi_h\bm\phi,
\end{gather*}
where $A_h$ is the Chern connection of $h$ on $L$.
One can show that $\langle\cF_{\alpha,\tau},\bm v^{1,0}\rangle$ does not depend on the choice of $(\omega, h)$, for $\omega$ with fixed volume $V$ (see~\cite[Proposition 4.4]{Al-Ga-Ga-P} or \cite[Lemma
2.5]{Al-Ga-Ga-P-Y}), so it vanishes if $(\Sigma, L, \bm\phi)$ has a gravitating vortex solution $(\omega,h)$ with volume $V$.

\begin{proof}[Proof of Theorem~\ref{thm:existence-implies-stability}]
We will show that the asymptotic slope of $\mathcal{K}_\alpha$ along a particularly chosen geodesic ray in $\mathcal{P}_{\omega_0}$, generated by a $\mathbb{C}^*$-action and  initiated from the (rescaled)
Fubini--Study metric, is the gravitating vortex Futaki invariant for the limiting Higgs field, in much a way similar to Ding--Tian's generalized Futaki invariant in the study of K\"ahler--Einstein metrics.
The result will then follow by combining an explicit calculation for the Futaki invariant and the properness of $\mathcal{K}_\alpha$ on $\mathcal{R}$.

Since this result is already known when $\bm\phi$ vanishes at no more than two points~\cite[Theorem 1.1]{Al-Ga-Ga-P-Y}, we only need to show that $\bm\phi$ is stable when it vanishes at least at three points.
It is enough to prove that $n_1<\frac{N}{2}$.
Without loss of generality, we can assume that
\[
\bm\phi=z_1^{n_1}\left(p_2z_0-z_1\right)^{n_2}\cdots \left( p_dz_0- z_1\right)^{n_d},
\]
where we use homogeneous coordinates $[z_0:z_1]$ on $\mathbb{P}^1$, so $\bm\phi$ has distinct zero set $\{p_1=0,p_2,\ldots,p_d\}\subset\mathbb{C}\subset\mathbb{P}^1$ and the multiplicities are $n_1,n_2,\ldots,n_d$,
respectively.

Let $\omega_0= \frac{V}{2\pi}\omega_{FS}$, where $\omega_{FS}$ is the Fubini--Study metric on $\mathbb{P}^1$ (with total volume $2 \pi$).
We consider the 1-parameter subgroup $\{\sigma_t\}_{t\in\RR}\subset\operatorname{PSL}(2,\mathbb{C})$ of automorphisms of $\PP^1$ given by $\sigma_t([z_0:z_1])=[e^{-t}z_0: e^t
z_1]$ as in Example~\ref{ex:strict-convexity}, and the corresponding geodesic line $\left\{\varphi_t\right\}_{t\in \mathbb{R}}$ (with $\varphi_0=0$) such that $\omega_t\defeq\sigma_t^*\omega_0=\omega_0+2i\partial\bar\partial\varphi_t$, as in
Example~\ref{ex:autgeod}. Choosing coordinates $z=z_1/z_0$ on $\PP^1$, the 1-parameter subgroup $\{\sigma_t\}_{t\in\mathbb{R}}$ is generated by the real holomorphic vector field $\text{Re}
\left(4z\frac{\partial}{\partial z}\right)$, and the geodesic ray $\{\varphi_t\}_{t\in [0, +\infty)}$ on $\mathcal{P}_{\omega_0}$ is given by~\eqref{eq:ex:strict-convexity.1}.
Recall that $\left\{ \varphi_t\right\}_{t\in [0,+\infty)}\subset \mathcal{R}$
because $\dot\varphi_t=\sigma_t^*\dot\varphi_0$ so $\frac{\mathrm{d}}{\mathrm{d}t}I(\varphi_t)=\int_{\mathbb{P}^1}\dot\varphi_t
\omega_t=\int_{\mathbb{P}^1}\dot\varphi_0\omega_0=0$ and $I(\varphi_0)=0$ (cf. Theorem~\ref{thm:convexfinen}(1)).

For each $\omega_t$, let $h_t=h_0 e^{2f_t}$ be the solution to the $t$-dependent vortex equation
\begin{equation}
iF_{h_t}+\frac{1}{2}\left( \left|\bm\phi\right|_{h_t}^2-\tau\right)\omega_t=0.
\end{equation}
Note that the line bundle $L$ has a canonical $\operatorname{GL}(2,\CC)$-linearization.
Let $\widetilde{\sigma}_t$ be the lift of the action $\sigma_t$ to $L$. Define $\widetilde h_t=(\widetilde{\sigma}_t^{-1})^*h_t=h_0 e^{2\widetilde f_t}$ and
\[
\widetilde{\bm\phi(t)}=(\widetilde{\sigma}_t^{-1})^*\bm\phi = e^{\left(N-2n_1\right)t}z_1^{n_1}\left(p_2 z_0 - e^{-2t}z_1\right)^{n_2}\cdots \left(p_d z_0-
  e^{-2t}z_1\right)^{n_d}=e^{\left(N-2n_1\right)t}\bm\phi(t).
\]
Then
\begin{equation}
iF_{\widetilde h_t} + \frac{1}{2} \left( \left|\widetilde{ \bm\phi(t)}\right|_{\widetilde h_t}^2 - \tau\right) \omega_0
=
iF_{\widetilde h_t}
+
\frac{1}{2}\left( \left| \bm\phi(t)\right|_{\widetilde h_t}^2 e^{2\left(N-2n_1\right)t} -\tau\right)\omega_0
=0.
\end{equation}
Define $\widehat f_t=\widetilde f_t + \left(N-2n_1\right)t$. Then the standard $W^{2,p}$ estimate about $\Delta_{\omega_0}$ as above, i.e. the proof of \eqref{est:W2pestimate},  shows that for any $p\geq 1$,
$\lVert\widehat f_t\rVert_{W^{2,p}}\leq C$ for $C>0$
independent of $t$. This implies that $\lVert\widehat f_t\rVert_{C^{1,\beta}}\leq C$ for any $\beta\in (0,1)$. By Schauder estimates and a bootstrapping argument, for any $k\in \mathbb{N}_+$ there exists $C_k>0$
independent of $t$ such that
\[
\lVert\widehat f_t\rVert_{C^k}\leq C_k.
\]
This uniform estimate on the pulled-back and normalized Hermitian potential $\widehat f_t$ shows that for any sequence $\{t_i\}$ such that $t_i\to +\infty$, there exists a subsequence $\{t_{i_l}\}$ such that
$\widehat f_{t_{i_l}}\to \widehat f_\infty$ (in any $C^k$ sense) which is a solution to the vortex equation
\[
iF_{h_0e^{2\widehat f_\infty}}+ \frac{1}{2}\left( \left| \bm\phi_\infty\right|_{h_0e^{2\widehat f_\infty}}^2-\tau\right) \omega_0=0,
\]
where $\bm\phi_\infty =\lim_{t\to +\infty} \bm\phi(t)=p_2^{n_2}\cdots p_d^{n_d}\cdot z_0^{N-n_1}z_1^{n_1}$ corresponds to the divisor $n_1[0]+(N-n_1)[\infty]$. By the uniqueness of the vortex solution with a given
Higgs field (see Theorem~ \ref{th:B-GP}), the subsequential convergence can be strengthened to convergence of the whole family as $t\to +\infty$.
	
The slope of $\mathcal{K}_\alpha(t)=\widetilde{\mathcal{K}}_\alpha(f_t, \varphi_t)$ at time $t$ is
\begin{equation*}
\begin{split}
\mathcal{K}_\alpha'(t)
& =
4\alpha
\int_\Sigma \dot f_t
\left( iF_{h_t} + \frac{1}{2} (\lvert\bm\phi\rvert_{h_t}^2 - \tau) \omega_t
\right)
-
\int_\Sigma \dot \varphi_t
\left( \text{Ric }\omega_t - 2\alpha i\partial\bar\partial \lvert\bm\phi\rvert_{h_t}^2 - 2\alpha\tau iF_{h_t}
- c\omega_t
\right)\\
& =
-
\int_\Sigma \dot \varphi_t
\left( \text{Ric }\omega_t - 2\alpha i\partial\bar\partial \lvert\bm\phi\rvert_{h_t}^2 - 2\alpha\tau iF_{h_t}
- c\omega_t
\right)\\
& =
-
\int_\Sigma \dot \varphi_0
\left(\text{Ric }\omega_0  -  2\alpha i\partial\bar\partial |\bm\phi(t)|_{\widehat h_t}^2 - 2\alpha\tau iF_{\widehat h_t}
- c\omega_0 \right).
\end{split}
\end{equation*}
As a consequence,
\begin{align*}
\lim_{t\to +\infty} \mathcal{K}_\alpha'(t)
&=-\int_\Sigma \dot\varphi_0
  \left(\text{Ric }\omega_0  -  2\alpha i\partial\bar\partial |\bm\phi_\infty|_{h_0 e^{2\widehat f_\infty}}^2 - 2\alpha\tau iF_{h_0 e^{2\widehat f_\infty}}
	- c\omega_0 \right)
\\&
=-i\left\langle\cF_{\alpha,\tau},z\frac{\partial}{\partial z}\right\rangle
=2\alpha\tau \left(N-2n_1\right)\left(V -\frac{4\pi N}{\tau}\right)
\end{align*}
where the last line follows by an explicit calculation~\cite[Lemma 4.7]{Al-Ga-Ga-P} of the gravitating vortex Futaki invariant on the triple $(\omega_0, h_0e^{2\widehat f_\infty}, \bm\phi_\infty)$ (see
also~\cite[Lemma 2.8]{Al-Ga-Ga-P-Y}, where~\eqref{def:Gravitating vortexFutaki} is denoted $\mathcal{F}_{n_1, N, V, \sqrt{\tau}}$, with $\alpha,\tau$ in this article corresponding to $\frac{a}{2\tau^2},\tau^2$
in~\cite{Al-Ga-Ga-P-Y}, and $\dot\varphi_0= \frac{V}{2\pi}\frac{|z|^2-1}{|z|^2+1}=-2i\varphi_\eta$ where $\varphi_\eta$ is the potential function in \cite[Def. 2.2]{Al-Ga-Ga-P-Y}).
In conclusion, the asymptotic slope of the reduced $\alpha$-K-energy along the geodesic ray in $\mathcal{R}$ which is initiated from $0$ and generated by
$\text{Re}\left(4z\frac{\partial}{\partial z}\right)$ is $-i\langle\cF_{\alpha,\tau},z\frac{\partial}{\partial z}\rangle$.

Since $\bm\phi$ vanishes at least at 3 points by assumption, the existence of a solution to the gravitating vortex equations implies that $\mathcal{K}_\alpha$ is proper on $\mathcal{R}$, by Theorem
\ref{thm:K-alpha-proper}. It is clear by~\cite[Theorem 4.3]{Darvas} that $d_1(0, \varphi_t) = t\int_\Sigma \left|\dot \varphi_0\right|\omega_0\to +\infty$ as $t\to +\infty$ and therefore
$\mathcal{K}_\alpha(\varphi_t)\to +\infty$ as $t\to +\infty$. This implies the asymptotic slope is strictly positive, since $\mathcal{K}_\alpha(t)$ is convex on $[0, +\infty)$ (Theorem \ref{thm:K-alpha-convex}).
This further implies that $2n_1<N$, because we assumed
$V-\frac{4\pi N}{\tau}>0$. We conclude that $\left[\bm\phi=0\right]$ is stable.
\end{proof}

\subsection{Existence of gravitating vortices in genus $g\geq 1$}
%
%
Our next theorem is similar to, but logically independent from, another existence result in~\cite[Theorem 6.1]{Al-Ga-Ga-P} under different numerical conditions (see Remark~\ref{rem:thm:existenceforggeq1.1} for
details).
The main difference comes from the use of convexity properties of the reduced $\alpha$-K-energy $\mathcal{K}_\alpha$ in the proof of Theorem~\ref{thm:existenceforggeq1}

\begin{theorem}\label{thm:existenceforggeq1}
%
Let $\Sigma$ be a compact Riemann surface of genus $g\geq 1$, and $L$ a holomorphic line bundle over $\Sigma$ of degree $N>0$ equipped with a holomorphic section $\bm\phi\neq 0$. Suppose the following three
numerical conditions are satisfied, where $m$ is the maximum multiplicity of the zeroes of $\bm\phi$:
\begin{enumerate}
\item[\textup{(1)}] $V- \frac{4\pi N}{\tau}>0$,
\item[\textup{(2)}] $\alpha\tau \left( \frac{8\pi N}{\tau} - V\right) < \frac{2g-2}{N}\left( V- \frac{4\pi N}{\tau}\right)$,
\item[\textup{(3)}] $2\alpha\tau m<1$.
\end{enumerate}
Then the gravitating vortex equation~\eqref{eq:GVagain} has a unique smooth solution of volume $V$.
\end{theorem}

\begin{remark}\label{rem:thm:existenceforggeq1.1}
Condition (1) is the bound~\eqref{eqn:ineqtauV}, that we already assumed to have solutions to the vortex equation~\eqref{eq:Vagain}. We record it to make it clear that Theorem~\ref{th:B-GP} applies.
Condition (2) can be compared with the hypotheses of a similar existence result~\cite[Theorem 6.1]{Al-Ga-Ga-P} that required $g\geq 2$ and the inequality $c+\alpha\tau^2\leq 0$ (see~\cite[(6.2)]{Al-Ga-Ga-P}),
or equivalently,
\[
\alpha\tau^2\left( V- \frac{4\pi N}{\tau} \right) \leq 2\pi \left(2g-2\right).
\]
This is different from our condition (2) in Theorem \ref{thm:existenceforggeq1}. In particular, condition (2) is vacuous when $V-\frac{8\pi N}{\tau}>0$, so Theorem~\ref{thm:existenceforggeq1} implies the existence
of solutions for any
coupling constant $\alpha\in [0, \frac{1}{2\tau m})$ in the regime of volume $V\in \left(\frac{8\pi N}{\tau}, +\infty\right)$. It is reasonable to conjecture that existence holds with only the first admissible
volume condition (1).
\label{rem:afterexistence}
\end{remark}

\begin{lemma}\label{lem:existenceforggeq1.1}
The set of real numbers $\alpha\geq 0$ for which there exists a smooth solution to the gravitating vortex equation~\eqref{eq:GVagain} with fixed volume $V>\frac{4\pi N}{\tau}$ is open.
\end{lemma}

\begin{proof}
This is a direct consequence of the moment-map interpretation of the gravitating vortex equations~\eqref{eqn:GV0} given in Proposition~\ref{prop:momentmap-inttriples} (cf. the proof of Lemma \ref{lem:perturb}).
Actually, the proof is in~\cite[Lemma~6.4]{Al-Ga-Ga-P} for $g\geq 2$, because the equation~\eqref{eq:GVagain} is equivalent to the system~\eqref{eqn:GV0}, by Theorem~\ref{th:B-GP}. 
For genus $g=1$, the proof of Lemma~6.4 of~\cite{Al-Ga-Ga-P} still works, because the condition $g\geq 2$ was only used there to conclude that the triple $(\Sigma,L,\bm{\phi})$ has no infinitesimal automorphisms,
and this still occurs when $g=1$ and $\bm{\phi}\neq 0$, by~\cite[Proposition 4.2]{Al-Ga-Ga-P}.
\end{proof}

\begin{proof}[Proof of Theorem~\ref{thm:existenceforggeq1}]
To prove existence, we apply the method of continuity, with $\alpha$ as the continuity parameter, to the gravitating vortex equation~\eqref{eq:GVagain}. Recall~\eqref{eq:GVagain} is equivalent to the system
\begin{equation}\label{eq:GVsystem-again}
\begin{split}
i\Lambda_\omega F_{h} + \frac{1}{2}(|\bm\phi|_{h}^2-\tau)&=0,
\\
S_\omega + \alpha\Delta_\omega  |\bm\phi|_{h}^2 -2i\alpha \tau \Lambda_\omega F_{h}&=c,
\end{split}
\end{equation}
where the unknown is $(\omega,h)$.
Let $\omega_0$ be the constant curvature K\"ahler metric on $\Sigma$ with volume $V$, and $h_0$ the constant curvature Hermitian metric on $L$, normalized so that  $|\bm \phi |_{h_0}^2\leq 1$.
Let $h_{f_0}=h_0e^{2f_0}$ be the unique solution to the vortex equation (see Theorem~\ref{th:B-GP})
\begin{equation}\label{eq:vortex-0.1}
i\Lambda_{\omega_0} F_{h_{f_0}}+\frac{1}{2}(\lvert\bm\phi\rvert_{h_{f_0}}^2-\tau) =0.
\end{equation}
At $\alpha=0$, the pair $(\omega_0,h_{f_0})$ solves~\eqref{eq:GVsystem-again}, so by Lemma~\ref{lem:existenceforggeq1.1}, we only need to prove that the set of $\alpha\geq 0$ for which there exists a smooth solution
to~\eqref{eq:GVagain} of volume $V$ is closed.
%

Since $\omega_0$ is the constant curvature K\"ahler metric on $\Sigma$ with volume $V$, the system~\eqref{eq:GVsystem-again} for $\omega=\omega_0+2i\partial\bar\partial\varphi=\omega_0(1-\Delta_{\omega_0}\varphi)$
and $h=h_0e^{2f}$ is equivalent to the system
\begin{align}\label{eq:rewrittenGVcontmethod}
\begin{split}
\Delta_{\omega_0} f +\frac{1}{2}(\lvert\bm \phi\rvert_{h}^2-\tau)e^{4\alpha \tau f-2\alpha\lvert\bm \phi\rvert_{h}^2-2c\varphi} &= -\frac{ 2\pi N}{V},\\
\Delta_{\omega_0} \varphi+e^{4\alpha \tau f-2\alpha\lvert\bm \phi\rvert_{h}^2-2c\varphi}&=1,
\end{split}
\end{align}
where the unknowns are $\varphi,f\in C^\infty(\Sigma)$ (see~\cite[(2.6)]{Al-Ga-Ga-P}), and $\lvert\bm\phi\rvert_{h}^2=e^{2f}\lvert\bm\phi\rvert_{h_0}^2$. It is easy to see that it is enough to prove $C^0$ \emph{a
priori} estimates
for $f$ and $\varphi$. By the Bochner identity
\begin{align}\label{ineq:Bochnerestimate}
\begin{split}
\lvert\bm \phi\rvert_{h}^2&\leq \tau, \\
\fint_\Sigma f\omega_0&\leq C
\end{split}\end{align}
(see~\cite[Proposition 5.2]{Brad}). Integrating on both sides of the equations~\eqref{eq:rewrittenGVcontmethod}, we obtain
\begin{align}\label{eq:normalisations}
\begin{split}
\displaystyle \int_\Sigma \left(\lvert\bm\phi\rvert_{h}^2-\tau \right)e^{4\alpha \tau f-2\alpha\lvert\bm\phi\rvert_{h}^2 -2c\varphi} \omega_0&=- 4\pi N , \\
\displaystyle \int_\Sigma e^{4\alpha \tau f-2\alpha\lvert\bm \phi\rvert_{h}^2 -2c\varphi} \omega_0&= V.
\end{split}
\end{align}
For the remainder of the proof, we assume that $\alpha$ is in a compact set $[\alpha_0, \alpha_*]$, where $\alpha_0, \alpha_*>0$. All the constants (usually denoted by $C$) will depend on this compact set. Using the
Green representation formula, we obtain
\begin{equation}\label{ineq:Greenrepforfandv}
\begin{split}
f&\geq \fint_\Sigma f\omega_0 -C, \\
\varphi&\leq \fint_\Sigma \varphi\omega_0 +C.
\end{split}\end{equation}
Applying the maximum principle for the second equation in \eqref{eq:rewrittenGVcontmethod}, we see that
\begin{equation}
-4\alpha \tau \max_\Sigma f \leq -2c\varphi \leq -4\alpha \tau \min_\Sigma f +2\alpha\tau \leq -4\alpha \tau \fint_\Sigma f\omega_0 +C.
\label{ineq:maxprinciplev}
\end{equation}
At this point we recall that any smooth solution $\omega=\omega_0+2i\partial\bar\partial\varphi$ to the gravitating vortex equation~\eqref{eq:GVagain} minimizes the reduced $\alpha$-K-energy
$\mathcal{K}_{\alpha}=\mathcal{K}+\mathcal{M}_{\alpha}$ (Proposition~\ref{thm:lower-bound}). Thus,
\begin{equation}
\begin{split}
\mathcal{K}_{\alpha}(\varphi)
& \leq
 \mathcal{K}_{\alpha}(0) \\
 & =
 2\alpha \int_\Sigma f_0\left( iF_{h_0} + iF_{h_{f_0}} \right)
 +
 \alpha \int_\Sigma \left( \lvert\bm\phi\rvert_{h_{f_0}}^2 - \left| \bm \phi\right|_{h_0}^2\right )\omega_0
 -
 2\alpha\tau \int_\Sigma f_0 \omega_0,
\end{split}
\end{equation}
where $h_{f_0}=h_0e^{2f_0}$ is the unique solution to the vortex equation~\eqref{eq:vortex-0.1}.
Using the explicit formula for $\mathcal{K}_\alpha(\varphi)=\widetilde{\mathcal{K}}_\alpha(f, \varphi)$ (Proposition~\ref{prop:explicit}
), this implies
\begin{equation}
 C
 \geq
   \frac{1}{2} \int_\Sigma \log  \frac{\omega_\varphi}{\omega_0} \cdot\omega_\varphi
 -c \int_\Sigma i\partial \varphi \wedge\bar{\partial}\varphi
 -2\alpha \tau \int_\Sigma f\omega_\varphi
 +2\alpha \int_\Sigma f\left( iF_{h_0}+ iF_{h_f}\right),
\label{ineq:lowerboundKenergy}
 \end{equation}
 which further implies
 \begin{equation}
C-8\pi\alpha N  \fint_\Sigma  f\omega_0
 \geq
  -c \int i\partial \varphi \wedge\bar{\partial}\varphi
  +4\alpha \int i\partial f\wedge \bar{\partial} f
  +\frac{1}{2} \int_\Sigma \log \frac{\omega_\varphi}{\omega_0} \cdot\omega_\varphi
  -2\alpha \tau \int_\Sigma f\omega_\varphi.
\label{ineq:energyestimate1}
\end{equation}
Recall now that $f\leq -\frac{1}{2}\log\lvert\bm\phi\rvert_{h_0}^2 + \frac{1}{2}\log \tau$, by~\eqref{ineq:Bochnerestimate}. Since we are assuming $2\alpha\tau m<1$, we can fix a real number $\epsilon$ such that
$0<\epsilon<2(1-2\alpha \tau m)$, and consider the region $S=\left\{\omega_\varphi\leq\lvert\bm \phi\rvert_{h_0}^{-\frac{2-\epsilon}{m}}\omega_0\right\}\subset\Sigma$. Decomposing the last two integrals
in~\eqref{ineq:energyestimate1} into integration on the region $S$ and its complement $S^c$, we obtain the following estimates:
\begin{equation*}
\begin{split}
\frac{1}{2} \int_\Sigma \log \frac{\omega_\varphi}{\omega_0} \cdot\omega_\varphi
& \geq
- \frac{e^{-1}}{2}V - \frac{2-\epsilon}{4m}\int_{S^c}\log \lvert\bm\phi\rvert_{h_0}^2 \omega_\varphi,\\
-2\alpha\tau \int_\Sigma f\omega_\varphi
&\geq - 2\alpha\tau \int_\Sigma \left( -\frac{1}{2}\log\lvert\bm \phi\rvert_{h_0}^2 + \frac{1}{2}\log \tau \right) \omega_\varphi \\
& =
\alpha\tau \int_S \log \lvert\bm\phi\rvert_{h_0}^2 \omega_\varphi
+
\alpha\tau \int_{S^c} \log \lvert\bm\phi\rvert_{h_0}^2 \omega_\varphi
-
V\alpha\tau\log \tau \\
& \geq
\alpha\tau \int_S \log \lvert\bm\phi\rvert_{h_0}^2\cdot\lvert\bm \phi\rvert_{h_0}^{-\frac{2-\epsilon}{m}}\omega_0
+
\alpha\tau \int_{S^c} \log \lvert\bm\phi\rvert_{h_0}^2 \omega_\varphi
-
V\alpha\tau\log \tau.
\end{split}\end{equation*}
Using the normalization condition $\lvert\bm\phi\rvert_{h_0}^2\leq 1$ and the integrability of $\log \lvert\bm\phi\rvert_{h_0}^2\cdot\lvert\bm \phi\rvert_{h_0}^{-\frac{2-\epsilon}{m}}$ (that follows from the
definition of $m$ and integration in polar coordinates), it follows easily that
\begin{gather}
C-8\pi\alpha N  \fint_\Sigma  f\omega_0
\geq -c \int_\Sigma i\partial \varphi \wedge\bar{\partial}\varphi +4\alpha \int_\Sigma i \partial f\wedge \bar{\partial} f.
\label{ineq:energyestimate2}
\end{gather}
Using the fact that $f$ satisfies the vortex equation, we can derive a different form of \eqref{ineq:energyestimate2}. Indeed, going back to the inequality \ref{ineq:lowerboundKenergy} and using the vortex equation
(i.e., the first equation of \eqref{eq:GVsystem-again}), we see that the following inequality holds:
\begin{align*}
C\geq&\frac{1}{2} \int_\Sigma \log  \frac{\omega_\varphi}{\omega_0} \cdot\omega_\varphi
 -c \int_\Sigma i\partial \varphi \wedge\bar{\partial}\varphi
 -\alpha  \int_\Sigma f\cdot \lvert\bm\phi\rvert_h^2 \omega_\varphi
\\&
-\alpha\tau \left(V -\frac{4\pi N}{\tau}\right) \fint_\Sigma f \omega_0
 -2\alpha \tau \int_\Sigma fi\partial \bar{\partial} \varphi.
\end{align*}
By calculations similar to those used above to estimate the term $\int_\Sigma f\omega_\varphi$, we have
\[
-\alpha \int_\Sigma f \cdot |\bm \phi|_h^2 \omega_\varphi
\geq
\frac{\alpha\tau}{2}\int_S \log |\bm \phi|_{h_0}^2 \omega_\varphi
+
\frac{\alpha\tau}{2}\int_{S^c} \log \lvert\bm\phi\rvert_{h_0}^2 \omega_\varphi
-
\frac{\alpha\log \tau}{2}\int_\Sigma \lvert\bm\phi\rvert_h^2 \omega_\varphi.
\]
It follows (by completing the square) that
\begin{align}\nonumber
C & \geq  -c \int_\Sigma i\partial \varphi \wedge\bar{\partial}\varphi
-\alpha\tau \left(V -\frac{4\pi N}{\tau}\right) \fint_\Sigma f \omega_0
 +2\alpha \tau \int_\Sigma i\partial f\wedge \bar{\partial} \varphi
\\\nonumber&
 = -\alpha\tau \left(V -\frac{4\pi N}{\tau}\right) \fint_\Sigma f \omega_0
 - c \int_\Sigma i\partial\left(\varphi-\frac{\alpha\tau}{c}f\right)\wedge \bar\partial \left(\varphi-\frac{\alpha\tau}{c}f\right)
 +
 \frac{\left(\alpha\tau\right)^2}{c}\int_\Sigma i\partial f\wedge\bar\partial f
\\\label{ineq:secondenergyestimate2}&
\geq
  -\alpha\tau \left(V -\frac{4\pi N}{\tau}\right) \fint_\Sigma f \omega_0
  +
   \frac{\left(\alpha\tau\right)^2}{c}\int_\Sigma i\partial f\wedge\bar\partial f.
\end{align}
Comparing the inequalities \eqref{ineq:energyestimate2} and \eqref{ineq:secondenergyestimate2}, we observe that if the numerical condition
\begin{equation}
\alpha\tau \left( \frac{8\pi N}{\tau} - V\right)
<
\frac{2g-2}{N}\left( V- \frac{4\pi N}{\tau}
\right)
\end{equation}
holds, then they imply $\fint_\Sigma f\omega_0\geq -C$, $\int_\Sigma i\partial f\wedge\bar\partial f\leq C$ and $\int_\Sigma i\partial \varphi\wedge \bar\partial \varphi \leq C$.

These inequalities imply that $\varphi\leq C$ and $f\geq -C$, by \eqref{ineq:Greenrepforfandv} and \eqref{ineq:maxprinciplev}. In particular, the $W^{1,2}$ norm of $\tilde f= f-\fint_\Sigma
f\omega_0$ is uniformly bounded. By the Moser--Trudinger inequality \cite{Li}, for any $\gamma>0$, there exists $C_\gamma>0$ such that
\[
\int_\Sigma e^{\gamma \tilde f}\omega_0\leq C_\gamma.
\]
The first equation in \eqref{eq:rewrittenGVcontmethod} implies that the $L^\gamma$ norm of $\Delta_{\omega_0} f$ is uniformly bounded, and $W^{2,\gamma}$ estimates for elliptic equations imply that the
$W^{2,\gamma}$  norm of $f$ is uniformly bounded. Since $\gamma$ can be taken arbitrarily large, for any $\beta\in (0,1)$ the $C^{1,\beta}$ norm of $f$ is uniformly bounded by Sobolev's inequality.

Inequality \eqref{ineq:maxprinciplev} implies $\varphi$ is uniformly bounded in $C^0$. The second Poisson equation in \eqref{eq:rewrittenGVcontmethod} then implies $\varphi$ is uniformly bounded in $W^{2,p}$ for any
$p>1$, and therefore it is uniformly bounded in $C^{1,\beta}$ for any $\beta\in(0,1)$. The standard bootstrap arguments applied to the elliptic system, \eqref{eq:rewrittenGVcontmethod} shows all higher order
estimates. Hence the closedness of the continuity method is proven, and so the existence of solutions is established.

Finally, uniqueness follows from Theorem~\ref{thm:uniqueness}. Alternatively, this is similar to the proof of the uniqueness part of~\cite[Theorem 6.1]{Al-Ga-Ga-P}, following an approach inspired by work of Bando
and Mabuchi.
\end{proof}

\appendix

\section{A variational approach to the vortex equation}
\label{sub:variational-approach-vortex-eq}

In this appendix, we revisit the existence and uniqueness problem for the vortex equation
\begin{equation}\label{eq:app.vortex.1}
i\Lambda_{\omega} F_{h}+\frac{1}{2}(\lvert\bm\phi\rvert_{h}^2-\tau)=0,
\end{equation}
providing a self-contained elementary proof of Theorem~\ref{th:B-GP}. Although our proof does not involve new ideas (cf., e.g.,~\cite{Mundet}), it serves as an introduction to our variational approach to the gravitating vortex equations, without the complexity associated to coupling the K\"ahler form $\omega$ on the Riemann surface and the Hermitian metric $h$ on the line bundle.

Let $\omega$ and $h_0$ be a smooth background K\"ahler metric of volume $V$ and a smooth Hermitian metric on a holomorphic line bundle $L$ over a compact Riemann surface $\Sigma$, respectively. We define a 1-form
$\widehat{\sigma}$ on the space of Hermitian potentials $\mathcal{P}_{h_0}=C^\infty(\Sigma,\mathbb{R})$ on $L$ by
\begin{equation}\label{eq:vortex-functional.0}
\widehat{\sigma}|_f(\dot{f})
=\int_\Sigma\dot{f}\left(iF_{h_0e^{2f}}+\frac{1}{2}\left( \lvert\bm\phi\rvert_{h_0e^{2f}}^2-\tau\right)\omega\right),
\end{equation}
for all $f,\dot{f}\in C^\infty(\Sigma)$, where $\dot{f}$ represents a tangent vector at $f\in \mathcal{P}_{h_0}$ (cf.~\eqref{defn:primitive}).
One can show as in the calculation~\eqref{eq:differential-mathfrak-S.1} that the 1-form $\widehat{\sigma}$ is closed, and so integrates to a functional
$\widehat{\mathcal{M}}_{\omega}\colon\mathcal{P}_{h_0}\to\mathbb{R}$.
Choosing the Hermitian metric $h_0$ on $L$ that satisfies $iF_{h_0}=\frac{2\pi N}{V}\omega$, this functional is given for all $f\in\mathcal{P}_{h_0}$ by
\begin{equation}\label{eq:vortex-functional.1}
\widehat{\mathcal{M}}_{\omega}(f)
=
\int_\Sigma i\partial f\wedge \bar\partial f
+
\frac{1}{4}\int_\Sigma \left|\bm\phi\right|_{h_0}^2 e^{2f}\omega
-
\frac{\tau}{2V}\left( V-\frac{4\pi N}{\tau}\right) \int_\Sigma f\omega
-
\frac{1}{4}\int_\Sigma \left|\bm\phi\right|_{h_0}^2 \omega.
\end{equation}
This follows as the calculation of the functional $\mathcal{M}_\alpha(f,\varphi)$ in the proof of Proposition~\ref{prop:explicit}. In fact, it is clear that when $S_{\omega_0} =\langle S\rangle $, the functional
$\mathcal{M}_\alpha$ of Proposition~\ref{prop:explicit} satisfies
\[
\mathcal{M}_\alpha(f,0)=4\alpha\widehat{\mathcal{M}}_{\omega_0}(f).
\]

\begin{theorem}\label{thm:classicalvortex}
Let $\omega$ be a background smooth K\"ahler metric of volume $V$ on $\Sigma$ and $h_0$ the Hermitian metric on $L$ such that $iF_{h_0}=\frac{2\pi N}{V}\omega$.
Then the following statements hold:
\begin{enumerate}
\item[\textup{(1)}]
A function $f\in\mathcal{P}_{h_0}$ is a critical point of $\widehat{\mathcal{M}}_{\omega}$ if and only if $h=h_0e^{2f}$ is a solution to the vortex equation~\eqref{eq:app.vortex.1}.
\item[\textup{(2)}]
Given any $f_0, f_1\in \mathcal{P}_{h_0}$, the function $\widehat{\mathcal{M}}_{\omega}(f_t)$ is strictly convex along the affine segment $f_t= (1-t)f_0+tf_1$, except in the case $f_1=f_0$. Therefore the solution
to the vortex equation~\eqref{eq:app.vortex.1} is unique.
\item[\textup{(3)}]
The existence of a solution $h=h_0e^{2\widehat f}$ to the vortex equation~\eqref{eq:app.vortex.1}, where $\widehat{f}\in\mathcal{P}_{h_0}$, implies
$\widehat{\mathcal{M}}_{\omega}(f)>\widehat{\mathcal{M}}_{\omega}(\widehat f)$ for any $f\in\mathcal{P}_{h_0}$ such that $f\neq\widehat{f}$.
\item[\textup{(4)}]
The existence of a solution to the vortex equation~\eqref{eq:app.vortex.1} implies the inequality \eqref{eqn:ineqtauV}. 
\item[\textup{(5)}]
The inequality \eqref{eqn:ineqtauV} implies the existence of a solution to the vortex equation~\eqref{eq:app.vortex.1}.
\end{enumerate}
\end{theorem}

\begin{proof}
Part (1) follows because the 1-form $\widehat{\sigma}$ is the differential of $\widehat{\mathcal{M}}_{\omega}$, as already observed, and $\widehat{\sigma}|_f=0$ if and only if $f$ is a vortex solution.
Parts (2) and (3) follow because along an affine segment $f_t=(1-t)f_0+tf_1$,~\eqref{eq:vortex-functional.1} (or more directly,~\eqref{eq:vortex-functional.0}) implies
\begin{equation}
\label{eq:thm:classicalvortex.1}
\frac{\mathrm{d}^2\,}{\mathrm{d}t^2}\widehat{\mathcal{M}}_{\omega}(f_t)
=
2\int_\Sigma i \partial (f_1-f_0)\wedge \bar\partial (f_1-f_0)
+
\int_\Sigma (f_1-f_0)^2 \left|\bm\phi\right|_{h_t}^2 \omega.
\end{equation}
To prove part (4), we assume $h=h_0e^{2\hat f}$ is a vortex solution, and evaluate the functional along the ray $f_t=\hat f-t=(1-t)\hat{f}+t(\hat{f}-1)$, $t\in [0, +\infty)$:
\begin{align*}
\widehat{\mathcal{M}}_{\omega}(f_t)
& =
\int_\Sigma i\partial \hat f\wedge \bar\partial \hat f
+
\frac{e^{-2t}}{4}\int_\Sigma \left|\bm\phi\right|_{h_0}^2 e^{2\hat f}\omega
+ t\cdot \frac{\tau}{2}\left(V-\frac{4\pi N}{\tau}\right) \\
&	-
\frac{\tau}{2V}\left( V- \frac{4\pi N}{\tau}\right) \int_\Sigma \hat f\omega
-
\frac{1}{4}\int_\Sigma \left|\bm\phi\right|_{h_0}^2 \omega.
\end{align*}
Hence the uniform lower bound of $\widehat{\mathcal{M}}_{\omega}(f_t)$ (that follows from part (1) and~\eqref{eq:thm:classicalvortex.1}) implies part (4).
For the converse, we first notice the inequality
\begin{align*}
\widehat{\mathcal{M}}_{\omega}(f)
& \geq
\frac{\lambda_1(\omega)}{2}\int_\Sigma \left( f - \fint_\Sigma f\omega\right)^2\omega
+
\frac{V}{4}e^{\fint_\Sigma \log\left|\bm\phi\right|_{h_0}^2\omega} \cdot e^{2\fint_\Sigma f\omega}\\
&\quad
-
\frac{\tau}{2}\left(V-\frac{4\pi N}{\tau}\right) \fint_\Sigma f\omega
-
\frac{1}{4}\int_\Sigma \left|\bm\phi\right|_{h_0}^2 \omega\\
& \geq
\frac{\lambda_1(\omega)}{2}\int_\Sigma \left( f - \fint_\Sigma f\omega\right)^2\omega
+
\frac{\tau}{2}\left(V-\frac{4\pi N}{\tau}\right) \left|\fint_\Sigma f\omega\right| \\
& -\left|
\frac{\tau}{2}\left(V-\frac{4\pi N}{\tau}\right)
\left[
1+ \log \frac{V}{4}e^{\fint_\Sigma \log\left|\bm\phi\right|_{h_0}^2\omega}
-
\log \frac{\tau}{2}\left(V-\frac{4\pi N}{\tau}\right)
\right]
 \right|
-
\frac{1}{4}\int_\Sigma \left|\bm\phi\right|_{h_0}^2 \omega,
\end{align*}
where $\lambda_1(\omega)$ is the first eigenvalue of $\Delta_{\omega}= 2i\Lambda_\omega \bar\partial \partial $ and we have used the elementary inequality $a e^{2x} - b x \geq b|x| -\left| b\left( 1+ \log a - \log
b\right)\right|$ for $a,b>0$ and $x\in \mathbb{R}$.
This inequality firstly implies that $\widehat{\mathcal{M}}_\omega$ is bounded from below on $\mathcal{P}_{h_0}$ and we can take a minimizing sequence $\{f_n\}$, that is,
\[
\lim_{n\to +\infty} \widehat{\mathcal{M}}_{\omega}(f_n)= \inf_{f\in \mathcal{P}_{h_0}} \widehat{\mathcal{M}}_{\omega}(f).
\]
By strengthening the above inequality to
\[
\widehat{\mathcal{M}}_{\omega}(f)\geq
\frac{1}{2}\int_\Sigma \left|\partial f\right|_\omega^2\omega
+
\frac{\lambda_1(\omega)}{4}\int_\Sigma\left( f- \fint_\Sigma f\omega\right)^2\omega
+
\frac{\tau}{2}\left(V-\frac{4\pi N}{\tau}\right) \left|\fint_\Sigma f\omega\right|
- C,
\]
where
\begin{equation}
C=
\left|\frac{\tau}{2}\left(V-\frac{4\pi N}{\tau}\right)
\left[
1+ \log \frac{V}{4}e^{\fint_\Sigma \log\left|\bm\phi\right|_{h_0}^2\omega}
-
\log \frac{\tau}{2}\left(V-\frac{4\pi N}{\tau}\right)
\right]
 \right|
+
\frac{1}{4}\int_\Sigma \left|\bm\phi\right|_{h_0}^2 \omega,
\end{equation}
we derive that
\[
\lVert  f_n\rVert_{W^{1,2}}\leq C.
\]
This implies that there exists $f_\infty\in W^{1,2}$ and a subsequence of $\{f_n\}$, still denoted $\{f_n\}$, such that $f_n\to f_\infty$ weakly in $W^{1,2}$ and $f_n\rightarrow f_\infty$ in $L^2$. In particular, we
have convergence almost everywhere for a subsequence. By Fatou's Lemma,
\[
\int_\Sigma \left|\bm\phi\right|_{h_0}^2e^{2f_\infty}\omega \leq \liminf_{n\to +\infty}\int_\Sigma \left|\bm\phi\right|_{h_0}^2e^{2f_n}\omega
\]
(actually, by the Moser--Trudinger inequality \cite{Li}, there exists $\beta>0$ such that $\int_\Sigma e^{\beta f_\infty^2}\omega<+\infty$ and further $\int_\Sigma e^{\gamma |f_\infty|}\omega<+\infty$ for any
$\gamma>0$). By weak convergence in $W^{1,2}$ and strong convergence in $L^2$, we have $\int_\Sigma i\partial f_\infty\wedge \bar\partial f_\infty \leq {\displaystyle\liminf_{n\to +\infty}}\int_\Sigma i \partial f_n
\wedge \bar\partial f_n$. These combine to show
\[
\widehat{\mathcal{M}}_{\omega}(f_\infty)\leq \liminf_{n\to +\infty}\widehat{\mathcal{M}}_{\omega}(f_n).
\]
Therefore any minimizing sequence yields a $W^{1,2}$ minimizer to the functional $\widehat{\mathcal{M}}_{\omega}$ on $ \mathcal{P}_{h_0}^{1,2}\defeq W^{1,2}(\Sigma,\omega)$. For any $v\in C^\infty(\Sigma)$ and
$t\in \mathbb{R}$, the inequality $\widehat{\mathcal{M}}_{\omega}(f_\infty + tv)\geq \widehat{\mathcal{M}}_{\omega}(f_\infty)$ actually implies
$\frac{\mathrm{d}\,}{\mathrm{d}t}\big|_{t=0}\widehat{\mathcal{M}}_{\omega}(f_\infty+tv)=0$, which precisely means that $f_\infty$ is a $W^{1,2}$ weak solution to
\[
iF_{h_0} - 2i\partial\bar\partial f_\infty
+
\frac{1}{2}\left( \left|\bm\phi\right|_{h_0}^2e^{2f_\infty}-\tau\right)\omega = 0.
\]
The standard Moser iteration applied to the weak sub equality $2i\partial\bar\partial f_\infty \geq - \frac{\tau}{2}\omega$ implies $f_\infty \in L^\infty$. Then the standard $W^{2,p}$ estimate to the equation
implies that $f_\infty\in W^{2,p}$ for any $p>1$.  This further implies that $f_\infty\in C^{1,\alpha}$ for any $\alpha\in (0,1)$. Finally, the Schauder estimate and bootstrapping arguments imply that $f_\infty\in
C^\infty$.
\end{proof}

\end{document}